\documentclass[11pt,reqno]{amsart}

\usepackage{xcolor}
\usepackage[T1]{fontenc}
\usepackage{amsmath}							
\usepackage{amssymb}
\usepackage{amsthm}
\usepackage{amscd}
\usepackage{amsfonts}
\usepackage{mathabx}
\usepackage{stmaryrd}
\usepackage[all]{xy}

\usepackage{caption}
\usepackage{subcaption}
\usepackage{euler}

\usepackage{extarrows}

\usepackage[colorlinks, linktocpage, citecolor = purple, linkcolor = blue]{hyperref}
\usepackage{color}

\usepackage{tikz}									
\usetikzlibrary{matrix}
\usetikzlibrary{patterns}
\usetikzlibrary{positioning}
\usetikzlibrary{decorations.pathmorphing}
\usetikzlibrary{cd}

\usepackage{ytableau}
\usepackage{fullpage}
\usepackage[shortlabels]{enumitem}

\newtheorem{maintheorem}{Theorem}	
								
\newtheorem*{kirchhoff}{Kirchhoff--Prym formula}
\newtheorem{theorem}{Theorem}[section]

\newtheorem{proposition}[theorem]{Proposition}
\newtheorem{corollary}[theorem]{Corollary} 
\newtheorem{conjecture}[theorem]{Conjecture} 

\theoremstyle{definition}
								
\newtheorem{definition}[theorem]{Definition}
\newtheorem{example}[theorem]{Example}

\newtheorem{mainconstruction}{Construction}	
\newtheorem{remark}[theorem]{Remark}
\newtheorem*{remark*}{Remark}
\newtheorem{question}[theorem]{Question}

\newcommand{\ZZ}{\mathbb{Z}}

\newcommand{\QQ}{\mathbb{Q}}
\newcommand{\RR}{\mathbb{R}}

\newcommand{\te}{\widetilde{e}}
\newcommand{\tf}{\widetilde{f}}
\newcommand{\tg}{\widetilde{g}}
\newcommand{\thh}{\widetilde{h}}
\newcommand{\tp}{\widetilde{p}}
\newcommand{\tu}{\widetilde{u}}
\newcommand{\tv}{\widetilde{v}}
\newcommand{\tC}{\widetilde{C}}
\newcommand{\tD}{\widetilde{D}}
\newcommand{\tF}{\widetilde{F}}
\newcommand{\tG}{\widetilde{G}}
\newcommand{\tM}{\widetilde{M}}
\newcommand{\tP}{\widetilde{P}}
\newcommand{\tQ}{\widetilde{Q}}
\newcommand{\tT}{\widetilde{T}}

\newcommand{\tZ}{\widetilde{Z}}
\newcommand{\tGa}{\widetilde{\Gamma}}
\newcommand{\tga}{\widetilde{\gamma}}
\newcommand{\tcalB}{\widetilde{\calB}}
\newcommand{\ga}{\gamma}
\newcommand{\Ga}{\Gamma}
\newcommand{\la}{\lambda}
\newcommand{\La}{\Lambda}
\newcommand{\de}{\delta}

\newcommand{\si}{\sigma}
\newcommand{\Si}{\Sigma}
\newcommand{\om}{\omega}
\newcommand{\Om}{\Omega}
\renewcommand{\oe}{\overline{e}}

\DeclareMathOperator{\Ker}{Ker}
\DeclareMathOperator{\Coker}{Coker}
\DeclareMathOperator{\Hom}{Hom}
\DeclareMathOperator{\Sym}{Sym}

\DeclareMathOperator{\Vol}{Vol}

\let\Im\relax
\DeclareMathOperator{\Im}{Im}

\newcommand{\calB}{\mathcal{B}}

\newcommand{\frakp}{\mathfrak{p}}

\DeclareMathOperator{\Pic}{Pic}
\DeclareMathOperator{\Alb}{Alb}
\DeclareMathOperator{\Prin}{Prin}

\DeclareMathOperator{\val}{val}

\DeclareMathOperator{\Id}{Id}
\DeclareMathOperator{\Div}{Div}

\DeclareMathOperator{\Nm}{Nm}

\DeclareMathOperator{\Prym}{Prym}
\DeclareMathOperator{\Gram}{Gram}

\let\ddiv\relax
\DeclareMathOperator{\ddiv}{div}

\DeclareMathOperator{\Jac}{Jac}

\newcommand\crule[3][black]{\textcolor{#1}{\rule{#2}{#3}}}

\definecolor{14}{RGB}{121,138,171}
\definecolor{15}{RGB}{147,192,157}
\definecolor{16}{RGB}{192,128,147}
\definecolor{17}{RGB}{196,209,121}
\definecolor{34}{RGB}{197,114,112}
\definecolor{35}{RGB}{109,109,204}
\definecolor{36}{RGB}{112,112,112}
\definecolor{37}{RGB}{174,174,174}
\definecolor{45}{RGB}{191,114,180}
\definecolor{46}{RGB}{130,207,115}
\definecolor{47}{RGB}{119,156,113}
\definecolor{56}{RGB}{134,134,134}
\definecolor{57}{RGB}{134,208,203}

\title[Kirchhoff's theorem for Prym varieties]{Kirchhoff's theorem for Prym varieties}

 \author{Yoav Len}
 \address{Mathematical Institute, University of St Andrews, St Andrews KY16 9SS, UK}
 \email{\href{mailto:yoav.len@st-andrews.ac.uk}{yoav.len@st-andrews.ac.uk}}
 
   \author{Dmitry Zakharov}
  \address{Department of Mathematics, Central Michigan University, Mount Pleasant, MI 48859, USA}
 \email{\href{mailto:dvzakharov@gmail.com}{dvzakharov@gmail.com}}


\begin{document}

\begin{abstract} 
We prove an analogue of Kirchhoff's matrix tree theorem for computing the volume of the tropical Prym variety for double covers of metric graphs. We interpret the formula in terms of a semi-canonical decomposition of the tropical Prym variety, via a careful study of the tropical Abel--Prym map. In particular, we show that the map is harmonic, determine its degree at every cell of the decomposition, and prove that its global degree is  $2^{g-1}$. Along the way, we use the Ihara zeta function to provide a new proof of the analogous result for finite graphs. As a counterpart, the appendix by Sebastian Casalaina-Martin shows that the degree of the algebraic Abel--Prym map is $2^{g-1}$ as well.

\end{abstract}

\maketitle

\setcounter{tocdepth}{1}
\tableofcontents


\section{Introduction}

Kirchhoff's celebrated matrix tree theorem states that the number of spanning trees of a connected finite graph $G$, also known as the \emph{complexity} of $G$, is equal to the absolute value of the determinant of the reduced Laplacian matrix of $G$. From a tropical viewpoint, this number is also equal to the order of the Jacobian group $\Jac(G)$ of $G$. 

In~\cite{ABKS_Canonical}, Kirchhoff's theorem was generalized to metric graphs and given a geometric interpretation. The Jacobian variety $\Jac(\Ga)$ of a metric graph $\Ga$ of genus $g$ is a real torus of dimension $g$, and its volume can be computed as a weighted sum over all spanning trees of $\Ga$. Given a set $F\subset E(\Ga)$ of $g$ edges of $\Ga$ (with respect to a choice of model), denote by $w(F)$ the product of the lengths of the edges in $F$. Then (see Theorem 1.5 in~\cite{ABKS_Canonical})
\begin{equation}
\Vol^2(\Jac(\Ga))=\sum_{F}w(F),
\label{eq:introABKS}
\end{equation}
where the sum is taken over those subsets $F$ such that $\Ga\backslash F$ is a spanning tree of $\Ga$.

The weighted matrix-tree theorem can be proved by a direct application of the Cauchy--Binet formula (see Remark 5.7 in~\cite{ABKS_Canonical}), but the authors give a geometric proof in terms of a canonical representability result for tropical divisor classes, that we briefly recall. Let $\Phi:\Sym^g(\Ga)\to \Pic^g(\Ga)$ be the tropical Abel--Jacobi map, sending an effective degree $g$ divisor $D$ to its linear equivalence class. A divisor $D=P_1+\cdots+P_g$ is a called a \emph{break divisor} if each $P_i$ is supported on an edge $e_i$ in such a way that $\{e_1,\ldots,e_g\}$ is the complement of a spanning tree of $\Ga$. By a result of Mikhalkin and Zharkov~\cite{MikhalkinZharkov}, the map $\Phi$ has a canonical continuous section, whose image is the set of break divisors in $\Sym^g(\Ga)$. Hence $\Pic^g(\Ga)$ (and, by translation, $\Jac(\Ga)$) has a canonical cellular decomposition coming from the cells of $\Sym^g(\Ga)$ parametrized by the spanning trees of $\Ga$. Computing the volume of $\Jac(\Ga)$ in terms of this decomposition gives Equation~\eqref{eq:introABKS}, where the terms in the right hand side correspond to the volumes of the individual cells. We note that the results of~\cite{ABKS_Canonical} can be reinterpreted as saying that the Abel--Jacobi map $\Phi$ is a \emph{harmonic morphism of polyhedral spaces of degree one} (see Remark~\ref{rem:ABKSdeg1}).

The purpose of this paper is to prove analogous results for the tropical Prym variety associated to a free double cover of metric graphs. Given an \'etale double cover $f:\widetilde{C}\to C$ of smooth algebraic curves of genera $2g-1$ and $g$ respectively, the kernel of the norm map $\Nm:\Jac(\widetilde{C})\to \Jac(C)$ has two connected components, and the even component  is an abelian variety of dimension $g-1$, known as the \emph{Prym variety}
$\Prym(\widetilde{C}/C)$ of the double cover. Prym varieties have been extensively studied following Mumford's seminal paper~\cite{Mumford_Prym}, as they are one of only few instances of abelian varieties that can be  described explicitly. Furthermore, they play a key role in rationality questions for threefolds \cite{ClemensGriffiths_3folds} and in constructing compact hyper-K\"ahler manifolds \cite{LSV_Kahler}. 

The notion of an \'etale cover of algebraic curves has two natural analogues in tropical geometry. One can consider \emph{free covers} $\pi:\tGa\to \Ga$, which are covering spaces in the topological sense: the map $\pi$ is a local homeomorphism at each point, and an isometry if the graphs are metric. It is often necessary to consider the more general \emph{unramified covers}, which are finite harmonic morphisms of metric graphs satisfying a numerical Riemann--Hurwitz condition. This notion does not have an analogue for finite graphs. The tropicalization of an \'etale cover of algebraic curves is an unramified cover of metric graphs, but not necessarily free. 

The tropical Prym variety $\Prym(\tGa/\Ga)$ associated to an unramified double cover $\pi:\tGa\to\Ga$ of metric graphs is defined in analogy with its algebraic counterpart~\cite[Definition 6.2]{JensenLen_thetachars}. Specifically, $\Prym(\tGa/\Ga)$ is the connected component of the identity of the kernel of the tropical norm map $\Nm:\Jac(\tGa)\to \Jac(\Ga)$ (note that in the tropical case, the kernel has two connected components if $\pi$ is free, and one if $\pi$ is unramified but not free). As shown in \cite[Theorem B]{Len_Ulirsch_Skeletons}, this construction commutes with tropicalization. Namely, if $\pi$ is the tropicalization of an \'etale double cover $f:\widetilde{C}\to C$ of algebraic curves, then the tropical abelian variety $\Prym(\tGa/\Ga)$ is the skeleton of the Berkovich analytification of $\Prym(\widetilde{C}/C)$, and the corresponding Abel--Prym maps commute (the corresponding result for Jacobians was proved in~\cite{BakerRabinoff_skelJac=Jacskel}). This observation has recently led to new results concerning the dimensions of Brill--Noether loci in Prym varieties \cite[Corollary B]{CLRW_PBN}.

In the current paper, we consider only free double covers of finite and metric graphs. We first compute the order of the \emph{Prym group} $\Prym(\tG/G)$ of a free double cover $p:\tG\to G$ of a finite graph $G$ of genus $g$. The finite group $\Prym(\tG/G)$ is a canonically defined index two subgroup of the kernel of the norm map $\Nm:\Jac(\tG)\to \Jac(G)$. In the spirit of Kirchhoff's formula, the order of $\Prym(\tG/G)$ is a weighted sum over certain $(g-1)$-element subsets of $E(G)$: given a subset $F\subset E(G)$ of $g-1$ edges of $G$, we say that $F$ is an \emph{odd genus one decomposition of rank $r$} if $G\backslash F$ consists of $r$  connected components of genus one, each having connected preimage in $\tG$.

\begin{kirchhoff} [Proposition~\ref{thm:discretePrym}] The order of the Prym group $\Prym(\tG/G)$ of a free double cover $p:\tG\to G$ of finite graphs is equal to
$$
|\Prym(\tG/G)|=\frac{1}{2}|\Ker \Nm|=\sum_{r=1}^g 4^{r-1}C_r,
$$
where $C_r$ is the number of odd genus one decompositions of $G$ of rank $r$.
\end{kirchhoff}

This formula has already been obtained by Zaslavsky in the seminal paper~\cite{zaslavsky1982signed} as the determinant of the \emph{signed Laplacian matrix} of the graph $G$ (see Theorem~8A.4 in \emph{loc.~cit.}), and was later explicitly interpreted as the order of the kernel of the norm map by Reiner and Tseng (see Proposition 9.9 in \cite{reiner2014critical}). We give an alternative proof, by comparing the Ihara zeta functions $\zeta(s,\tG)$ and $\zeta(s,G)$ of the graphs $\tG$ and $G$. By the work of Stark and Terras~\cite{stark1996zeta,stark2000zeta}, the quotient $\zeta(s,\tG)/\zeta(s,G)$ for a free double cover $p:\tG\to G$ is the $L$-function of the cover evaluated at the nontrivial representation of the Galois group $\ZZ/2\ZZ$, and we use the $L$-function to compute the order of the Prym group. To the best of our knowledge, this is the first application of the Ihara zeta function to tropical geometry. 

We then derive a weighted version of the Kirchhoff--Prym formula for the volume of the Prym variety of a free double cover of metric graphs, in the same way that Equation~\eqref{eq:introABKS} generalizes Kirchhoff's theorem.

\begin{maintheorem}[Theorem~\ref{thm:Prymvolume}] The volume of the tropical Prym variety $\Prym(\tGa/\Ga)$ of a free double cover $\pi:\tGa\to \Ga$ of metric graphs is given by
$$
\Vol^2(\Prym(\tGa/\Ga))=\sum_{F\subset E(\Ga)} 4^{r(F)-1}w(F),
$$
where the sum is taken over all odd genus one decompositions $F$ of $\Ga$, and where $w(F)$ is the product of the lengths of the edges in $F$.
\label{thm:B}
\end{maintheorem}

In the second part of our paper, we derive a geometric interpretation for the volume formula for the tropical Prym variety, in the spirit of~\cite{ABKS_Canonical}. Let $\pi:\tGa\to \Ga$ be a free double cover of metric graphs, and let $\iota:\tGa\to \tGa$ be the associated involution. Consider the \emph{Abel--Prym map} $\Psi$ associated to $\pi$
$$
\Psi:\Sym^{g-1}(\tGa)\to \Prym^{[g-1]}(\tGa/\Ga),\quad \Psi(D)=D-\iota(D),
$$
where $\Prym^{[g-1]}(\tGa/\Ga)$ denotes the component of $\Ker \Nm$ of the same parity as $g-1$.

Our principal result states that $\Psi$ is a \emph{harmonic morphism of polyhedral spaces of degree $2^{g-1}$} (as in Definition~\ref{def:harmonic}). The space $\Sym^{g-1}(\tGa)$ has a natural polyhedral decomposition, with the top-dimensional cells $C(\tF)$ indexed by multisets $\tF\subset E(\tGa)$ of $g-1$ edges of $\tGa$. We define the \emph{degree} of a top-dimensional cell to be $\deg_{\Psi}(\tF)=2^{r(\tF)-1}$ if $p(\tF)$ consists of distinct edges and is an odd genus one decomposition of rank $r(\tF)$, and zero otherwise. Then the Abel--Prym map $\Psi$ contracts the cell $C(\tF)$ if and only if $\deg_{\Psi}(\tF)=0$. Furthermore $\Psi$ is harmonic with respect to the degree, meaning that it satisfies a balancing condition around every codimension one cell of $\Sym^{g-1}(\tGa)$. This implies that we can extend the degree function to all of $\Sym^{g-1}(\tGa)$ in such a way that the sum of the degrees in each fiber of $\Psi$ is a finite constant, called the global degree of $\Psi$. To compute the global degree, we first observe that the harmonicity of the Abel--Prym map allows us to express  the volume of $\Prym(\tGa/\Ga)$ in terms of its degree. Comparing the result with Theorem~\ref{thm:B}, we find that the global degree is in fact $2^{g-1}$. The factors $4^{r(F)-1}$ in the weighted Kirchhoff--Prym formula represent squares of the local degrees of $\Psi$.

Summarizing, we obtain a semi-canonical representability result for tropical Prym divisors:

\begin{maintheorem} [Theorem~\ref{thm:APharmonic}] The Abel--Prym map  $\Psi:\Sym^{g-1}(\tGa)\to \Prym^{[g-1]}(\tGa/\Ga)$ associated to a free double cover $\pi:\tGa\to \Ga$ of metric graphs is a harmonic morphism of polyhedral spaces of degree $2^{g-1}$. In particular, there is a degree map $\deg_{\Psi}:\Sym^{g-1}(\tGa)\to \ZZ_{\geq 0}$ such that any element of $\Prym^{[g-1]}(\tGa/\Ga)$ has exactly $2^{g-1}$ representatives of the form $\tD-\iota(\tD)$ counted with multiplicity $\deg_{\Psi} (\tD)$,  where $\tD$ is an effective divisor of degree $g-1$.

\label{thm:C}
\end{maintheorem}

We note that a divisor in $\Prym^{[g-1]}(\tGa/\Ga)$ may have infinitely many representatives of the form $\tD-\iota(\tD)$ with $\deg_{\Psi}(\tD)=0$, but a generic divisor in $\Prym^{[g-1]}(\tGa/\Ga)$ only has representatives $\tD-\iota(\tD)$ with $\deg_{\Psi}(\tD)>0$, and hence finitely many in total.

The canonical representability result of~\cite{ABKS_Canonical} also holds in the integral setting, after fixing a generic element $\lambda\in\Jac(\Gamma)$: given a  model $G$ of $\Gamma$, any class in $\Pic^g(G)$ is represented by a unique break divisor $D\in \Sym^g(G)$.
Shifting the break divisor by $\lambda$, we obtain a divisor supported on the complement of a spanning tree of $G$. 
 See \cite[Remark 4.26]{ABKS_Canonical} and \cite[Example 1.3.4]{BBY_bijections} for more detail. 
An analogous correspondence  result does not hold for Prym groups. In fact, the discrete Abel--Prym map $\Sym^{g-1}(\tG)\to \Prym^{[g-1]}(\tG/G)$ associated to a free double cover $p:\tG\to G$ of finite graphs is not even surjective in general (see Example~\ref{ex:doublecover1}). 

We believe that suitable generalizations of Theorems~\ref{thm:B} and~\ref{thm:C} hold for unramified double covers of metric graphs, which is the more general framework considered in~\cite{JensenLen_thetachars} and~\cite{Len_Ulirsch_Skeletons}. To derive and prove them using the methods of our paper, it would first be necessary to develop a theory of $L$-functions of unramified Galois covers of graphs, extending the theory for free covers developed in~\cite{stark1996zeta} and~\cite{stark2000zeta}. Such a theory should be a part of a more general theory of Ihara zeta functions of graphs of groups. This first step in this direction is the paper~\cite{2020Zakharov} by the second author. It would also be interesting to determine whether the Prym construction generalizes to other tropical abelian covers (see~\cite{LenUlirschZakharov_AbelianCovers}).

\subsection{The algebraic Abel--Prym map and its tropicalization} Let $C$ be a smooth algebraic curve of genus $g$, and let $\Phi^d:\Sym^d(C)\to \Pic^d(C)$ be the degree $d$ Abel--Jacobi map. It is a classical result that $\Phi^d$ has degree 1 when $d\leq g$, and is birational when $d=g$ \cite[Chapter 1.3]{ACGHI}. 
The degree $d$ Abel--Prym map $\Psi^d:\Sym^d(\widetilde{C})\to \Prym^{[d]}(\widetilde{C}/C)$ corresponding to an unramified double cover $\pi:\widetilde{C}\to C$ of smooth algebraic curves is defined by $\Psi^d(\tD)=\tD-\iota(\tD)$. Unlike the Abel--Jacobi map, the degree of $\Psi^d$ depends non-trivially on the Brill--Noether type of $C$. For example, if $d=1$ then the degree is equal to $2$ if $C$ is hyperelliptic and $1$ otherwise. However, the degree of the Abel--Prym map when $d=g-1$ is always $2^{g-1}$. We are very grateful to Sebastian Casalaina-Martin for a proof of this result (and a number of others) about the Abel--Prym map, which we have included as an Appendix to this paper. 

Given that the algebraic Abel--Prym map $\Psi^{g-1}$ has degree $2^{g-1}$, it is tempting to derive Theorem~\ref{thm:C} from the corresponding algebraic statement by a tropicalization argument (the same argument would also give an alternative proof of one of the principal results of~\cite{ABKS_Canonical}, namely the existence of a canonical section of the Abel--Jacobi map). It is well known that the tropicalization of a degree $d$ map of algebraic curves is a harmonic morphism of metric graphs of the same degree $d$. However, we are unaware of a suitable generalization of this result to higher dimension, and the derivation of such a result is beyond the scope of this paper. 

Motivated by this similarity, and by the results of the Appendix, we propose the following conjecture:

\begin{conjecture} Let $f:\widetilde{C}\to C$ be an \'etale double cover of algebraic curves tropicalizing to a free double cover $\pi:\tGa\to\Ga$ of metric graphs. Then the degrees of the algebraic and tropical Abel--Prym maps $\Psi^d$ for $d\leq g-2$ associated to $f$ and $\pi$ coincide. In particular, the degree of $\Psi^d$ is bounded by $2^d$.
\end{conjecture}

We stress that the tropical and algebraic results presented in this paper are derived via entirely different techniques, and are independent of each other. 

\subsection{Degenerations of abelian varieties} 
Polyhedral decompositions of real tori, such as the ones described above, suggest an interesting connection with degenerations of abelian varieties and compactifications of their moduli spaces. 

The Jacobian of a nodal curve is a semi-abelian variety that is not proper in general. There are numerous compactifications constructed by various authors that depend on a choice of degree and an ample line bundle (e.g. \cite{Esteves_Compactification, Simpson_Compactification}). In degree $g$, these constructions coincide \cite{Caporaso_CompactifiedPicard}, and the strata in the compactification are in bijection with certain orientations on the dual graph of the curve \cite[Theorem 3.2.8]{Christ_PhD}. In fact, the same strata are in an order reversing bijection with the cells in the ABKS decomposition of the tropical Jacobian \cite[Theorem 4.3.4]{Caporaso_ICM}. More generally, each Simpson and Esteves compactified Jacobian of $C$ can be constructed from a polyhedral decomposition of the tropical Jacobian of the dual graph of $C$ \cite[Theorem 1.1]{ChristPayneShen_CompactifiedJacobians}.  An analogous statement in degree $g$ holds uniformly over the moduli space of curves \cite{AAPT_UniversalTropicalJacobian}. 

The situation is more subtle for Prym varieties. Given an admissible double cover $\widetilde{C}\to C$ of nodal curves, the identity component of the kernel of the norm map is, again, a non-proper semi-abelian variety. There are various approaches for compactifying the Prym variety (e.g. \cite{ABH_PrymDegenerations, CGHR_Prym}). However, unlike the case of Jacobians, the Prym--Torelli map $\mathcal{R}_g\to A_{g-1}$ from the moduli space of \'etale double covers to the moduli space of abelian varieties does not extend to the boundary for any reasonable toroidal compactification of $A_{g-1}$ \cite{Vologodsky_Prym, FriedmanSmith_Prym}. 

We therefore ask the following. 
 
\begin{question} Given an admissible double cover $\widetilde{C}\to C$ with tropicalization $\tGa\to \Ga$, do the cells of the semi-canonical decomposition of the tropical Prym variety $\Prym(\tGa/\Ga)$ described in Theorem~\ref{thm:C} correspond to the boundary strata of an appropriate compactification of the Prym variety $\Prym(\widetilde{C}/C)$?
\end{question}
\noindent A positive answer would suggest a path to a natural compactification of the moduli space of abelian varieties such that the map $\mathcal{R}_g\to A_{g-1}$ extends to the boundary. 

\subsection*{Acknowledgments} We would like to thank Matthew Baker, Samuel Grushevsky, Sam Payne, and Dhruv Ranganathan for useful discussions, and David Jensen, Martin Ulirsch, and Chi-Ho Yuen for comments on an older version of the paper. We are very thankful to Victor Reiner for pointing out the history of the Kirchhoff--Prym formula in the context of critical groups and signed graphs. We are deeply grateful to Sebastian Casalaina-Martin for a comprehensive Appendix dedicated to the algebraic Abel--Prym map.

\section{Preliminaries}\label{sec:preliminaries}

In this section, we review the necessary material about graphs, metric graphs, tropical ppavs, Jacobians, Prym varieties, and polyhedral spaces. The only new material is found in Section~\ref{sec:PrymGroups}, where we define the Prym group of a free double cover of graphs. Throughout this paper, we consider both finite and metric graphs, which we distinguish by using Latin and Greek letters, respectively. Graphs are allowed to have loops and multi-edges but not legs, and we do not consider the more general setting of graphs with vertex weights. All graphs are assumed to be connected unless stated otherwise.

\subsection{Graphs and free double covers} We denote the vertex and edge sets of a finite graph $G$ by respectively $V(G)$ and $E(G)$, and its \emph{genus} by $g(G)=|E(G)|-|V(G)|+1$. An \emph{orientation} of a graph $G$ is a choice of direction for each edge, allowing us to define {\it source} and \emph{target} maps $s,t:E(G)\to V(G)$. For a vertex $v\in V(G)$, the \emph{tangent space} $T_vG$ is the set of edges emanating from  $v$, and the \emph{valency} is $\val(v)=\#T_vG$ (where each loop at $v$ counts twice towards the valency). A \emph{metric graph} $\Ga$ is the compact metric space obtained from a finite graph $G$ by assigning positive lengths $\ell:E(G)\to \RR_{>0}$ to its edges, and identifying each edge $e\in E(G)$ with a closed interval of length $\ell(e)$. The pair $(G,\ell)$ is called a {\it model} of $\Ga$, and we define $g(\Ga)=g(G)$. A metric graph has infinitely many models, obtained by arbitrarily subdividing edges, but the genus $g(\Ga)$ does not depend on the choice of model.

The only maps of finite graphs that we consider in our paper are \emph{free double covers} $p:\tG\to G$. Such a map consists of a pair of surjective 2-to-1 maps $p:V(\tG)\to V(G)$ and $p:E(\tG)\to E(G)$ that preserve adjacency, and such that the map is an isomorphism in the neighborhood of every vertex of $\tG$. Specifically, for any pair of vertices $\tv$ and $v$ with $p(\tv)=v$, and for each edge $e\in E(G)$ attached to $v$, there is a unique edge $\te\in E(G)$ attached to $\tv$ that maps to $e$. We say that $p:\tG\to G$ is \emph{oriented} if $\tG$ and $G$ are oriented graphs, and if the map $p$ preserves the orientation. There is a naturally defined \emph{involution} $\iota:\tG\to \tG$ on the source graph that exchanges the two sheets of the cover. It is easy to see that if $G$ has genus $g$, then any connected double cover $\tG$ of $G$ has genus $2g-1$.

\begin{remark} If $p:\tG\to G$ is a free double cover and $e\in E(G)$ is a loop at $v$, then the preimage of $e$ is either a pair of loops, one at each of the two vertices in $p^{-1}(v)$, or a pair of edges connecting the two vertices in $p^{-1}(v)$ (oriented in the opposite directions if $e$ is oriented).

\end{remark}

A free double cover of metric graphs $\pi:\tGa\to \Ga$ is a free double cover $p:\tG\to G$ of appropriate models $(\tG,\ell)$ and $(G,\ell)$ of respectively $\tGa$ and $\Ga$ that preserves edge length, so that $\ell(p(\te))=\ell(\te)$ for all $\te\in E(\tGa)$. A free double cover is the same as a finite harmonic morphism of global degree two and local degree one everywhere, and we do not consider the more general case of unramified harmonic morphisms of degree two studied in~\cite{JensenLen_thetachars} and~\cite{Len_Ulirsch_Skeletons}. From a topological viewpoint, free double covers are the same as normal covering spaces with Galois group $\ZZ/2\ZZ$.

We consistently use the following construction, due to~\cite{Waller_DoubleCovers}, to describe a double cover $p:\tG\to G$ of a graph $G$ of genus $g$.

\begin{mainconstruction} \label{con:A} 
Let $G$ be a graph of genus $g$. Fix a spanning tree $T\subset G$ and a subset $S\subset \{e_0,\ldots,e_{g-1}\} $ of the edges in the complement of $T$. Let $\tT^+$ and $\tT^-$ be two copies of $T$, and for a vertex $v\in V(T)=V(G)$ denote $\tv^{\pm}$ the corresponding vertices in $\tT^{\pm}$. We define the graph $\tG$ as 
$$
\tG=\tT^+\cup \tT^- \cup \{\te_0^{\pm},\ldots,\te_{g-1}^{\pm}\}.
$$
The map $p:\tG\to G$ sends $\tT^{\pm}$ isomorphically to $T$ and $\te_i^{\pm}$ to $e_i$. For $e_i\in S$, each of the two edges $\te^{\pm}_i$ above it has one vertex on $\tT^+$ and one on $\tT^-$, while for $e_i\notin S$ both vertices of $\te^{\pm}_i$ lie on the tree $\tT^{\pm}$.  It is clear that if $G$ is connected, then $\tG$ is connected if and only if $S$ is nonempty. In the latter case, we may and will assume that $e_0\in S$, and then $\tT=\tT^+\cup \tT^-\cup \{\te_0^+\}$ is a spanning tree for $\tG$. We furthermore always assume that the starting and ending vertices of $\te^+_0$ lie respectively on $\tT^+$ and $\tT^-$, and conversely for $\te^-_0$:
$$
s(\te^{\pm}_0)=\widetilde{s(e_0)}^{\pm},\quad t(\te^{\pm}_0)=\widetilde{t(e_0)}^{\mp}.
$$
We do not make the same assumptions about the lifts of the remaining edges $e_i\in S$. 

The set of connected free double covers of $G$ is thus identified with the set of nonempty subsets of $\{e_0,\ldots,e_{g-1}\}$. Alternatively, the fundamental cycle construction defines a basis for $H_1(G,\ZZ)$ corresponding to the edges $e_i$ (the $i$-th basis element is the unique cycle supported on $T\cup\{e_i\}$ and containing $+e_i$). The set of nonempty subsets of $\{e_0,\ldots,e_{g-1}\}$ is then identified with the set of nonzero elements of $\Hom(H_1(G,\ZZ),\ZZ/2\ZZ)=H^1(G,\ZZ/2\ZZ)$ (a subset is identified with its indicator function), and the latter is canonically identified with the set of connected free double covers of $G$ by covering space theory.

\end{mainconstruction}

\begin{remark} Let $p:\tG\to G$ be a free double cover corresponding to a tree $T\subset G$ and a subset $S\subset E(G)\backslash E(T)$, and let $G'\subset G$ be a subgraph. Then the preimage $p^{-1}(G')$ is connected (equivalently, the restricted cover $p|_{p^{-1}(G')}:p^{-1}(G')\to G'$ is a nontrivial free double cover) if and only if there is a cycle on $G'$ that contains an odd number of edges from $S$. 
\label{rem:oddS}
\end{remark}

\subsection{Chip-firing and linear equivalence}

We now briefly recall the basic notions of divisor theory for finite and metric graphs (see~\cite[Section 1]{BakerNorine_RiemannRoch} and~\cite[Section 2]{LimPaynePotashnik_BrillNoether} respectively for details).

Let $G$ be a finite graph. The \emph{divisor group} $\Div(G)$ of $G$ is the free abelian group on $V(G)$, and the \emph{degree} of a divisor is the sum of its coefficients:
$$
\Div(G)=\left\{\sum a_v v:a_v\in \ZZ\right\},\quad \deg \sum a_v v=\sum a_v.
$$
A divisor $D=\sum a_vv$ is called \emph{effective} if all $a_v\geq 0$, and we denote the set of divisors of degree $d$ by $\Div^d(G)$.

Let $n=|V(G)|$ be the number of vertices, and let $Q$ and $A$ be the $n\times n$ \emph{valency} and \emph{adjacency} matrices:
\begin{equation}
Q_{uv}=\delta_{uv}\val(u),\quad A_{uv}=|\{\text{edges between $u$ and $v$}\}|.
\label{eq:QA}
\end{equation}
The \emph{Laplacian} $L=Q-A$ of $G$ is a symmetric degenerate matrix whose rows and columns sum to zero.
Given a vertex $v$, the divisor obtained via \emph{chip-firing from} $v$ is 
$$
D_v=-\sum_{u\in V(G)}L_{uv}u.
$$
Such a divisor has degree zero, hence the set of \emph{principal divisors} $\Prin(G)$, which are defined as the image of the chip-firing map
$$
\ZZ^{V(G)}\to \ZZ^{V(G)}=\Div(G),\quad a\mapsto -La,
$$
lies inside $\Div^0(G)$. The \emph{Picard group} and \emph{Jacobian} of $G$ are defined as 
$$
\Pic(G)=\Div(G)/\Prin(G),\quad \Jac(G)=\Div^0(G)/\Prin(G).
$$
Since any principal divisor has degree zero, the degree function descends to $\Pic(G)$, and we denote $\Pic^k(G)$ the set of equivalence classes of degree $k$ divisors, so that $\Jac(G)=\Pic^0(G)$. The group $\Pic(G)$ is infinite, but $\Jac(G)$ is a finite group whose order is equal to the absolute value of any cofactor of the Laplacian $L$. \emph{Kirchhoff's matrix tree theorem} states that $|\Jac(G)|$ is equal to the number of spanning trees of $G$ (see~\cite[Theorem 6.2]{BakerShokrieh_Trees}).

The Picard variety of a metric graph $\Ga$ of genus $g$ is defined as follows (see~\cite{BakerFaber_tropicalAbelJacobi}). A \emph{divisor} on a metric graph $\Ga$ is a finite linear combination of the form
$$
D = a_1 p_1 + a_2 p_2 +\cdots + a_k p_k,
$$
where $a_i\in\ZZ$ and $p_i$ can be any point of $\Ga$, and $\deg D=a_1+\cdots+a_k$. We denote by $\Div(\Ga)$ the divisor group and by $\Div^k(\Ga)$ the set of divisors of degree $k$. A \emph{rational function} $M$ on $\Ga$ is a piecewise-linear real-valued function with integer slopes. The principal divisor $\ddiv (M)$  associated to $M$ is the degree zero divisor whose value at each point $p\in \Ga$ is the sum of the incoming slopes of $M$ at $p$. It is clear that $\ddiv(M+N)=\ddiv(M)+\ddiv(N)$ and $\ddiv(-M)=-\ddiv(M)$, so the principal divisors $\Prin(\Ga)$ form a subgroup of $\Div^0(\Ga)$, and the degree function descends to the quotient:
$$
\Pic(\Ga)=\Div(\Ga)/\Prin(\Ga),\quad \Pic^k(\Ga)=\{[D]\in \Pic(\Ga):\deg D=k\}.
$$
The \emph{Picard variety} $\Pic^0(\Ga)$ is a real torus of dimension $g$ and is isomorphic to the Jacobian variety of $\Ga$, which we review in the next section, while each $\Pic^k(\Ga)$ is a torsor over $\Pic^0(\Ga)$.

\subsection{Tropical abelian varieties}

The Jacobian variety of a metric graph $\Ga$ is a \emph{tropical principally polarized abelian variety} (tropical ppav for short). We review the theory of tropical ppavs, following are~\cite{FosterRabinoffShokriehSoto_tropicaltheta} and~\cite{Len_Ulirsch_Skeletons}, though we have found it convenient to slightly modify the main definitions (see Remark~\ref{rem:tppavs}). In brief, a tropical ppav is a real torus $\Si$ whose universal cover is equipped with a distinguished lattice (used to define integral local coordinates on $\Si$, and in general distinct from the lattice defining the torus itself), and an inner product.

Let $\La$ and $\La'$ be finitely generated free abelian groups of the same rank, and let $[\cdot,\cdot]:\La'\times \La\to \RR$ be a nondegenerate pairing. The triple $(\La,\La',[\cdot,\cdot])$ defines a \emph{real torus with integral structure} $\Si=\Hom(\La,\RR)/\La'$, where the "integral structure" refers to the lattice $\Hom(\La,\ZZ)\subset \Hom(\La,\RR)$, and where $\La'$ is embedded in $\Hom(\La,\RR)$ via the assignment $\la'\mapsto [\la',\cdot]$. The transposed data $(\La',\La,[\cdot,\cdot]^t)$ define the \emph{dual torus} $\Si'=\Hom(\La',\RR)/\La$. 

Let $\Si_1=(\La_1,\La'_1,[\cdot,\cdot]_1)$ and $\Si_2=(\La_2,\La'_2,[\cdot,\cdot]_2)$ be two real tori with integral structure, and let $f_*:\La_1'\to \La_2'$ and $f^*:\La_2\to \La_1$ be a pair of maps satisfying
\begin{equation}
    [\la'_1,f^*(\la_2)]_1=[f_*(\la'_1),\la_2]_2
\label{eq:integral}
\end{equation}
for all $\la'_1\in \La'_1$ and $\la_2\in \La_2$. The map $f^*$ defines a dual map $\overline{f}:\Hom(\La_1,\RR)\to \Hom (\La_2,\RR)$, and condition~\eqref{eq:integral} implies that $\overline{f}(\La'_1)\subset \La'_2$ (in fact, $\overline{f}|_{\La'_1}=f_*$). Hence the pair $(f_*,f^*)$ defines a \emph{homomorphism} $f:\Si_1\to \Si_2$ of real tori with integral structures. The transposed pair $(f^*,f_*)$ defines the \emph{dual homomorphism} $f':\Si'_2\to \Si'_1$.

Let $f=(f_*,f^*):\Si_1\to \Si_2$ be a homomorphism of real tori with integral structures $\Si_i=(\La_i,\La'_i,[\cdot,\cdot]_i)$. We can naturally associate two real tori to $f$: the connected component of the identity of the kernel of $f$, denoted $(\Ker f)_0$, and the cokernel $\Coker f$. It is easy to see that $(\Ker f)_0$ and $\Coker f$ also have integral structures, and the natural maps $i:(\Ker f)_0\to \Si_1$ and $p:\Si_2\to \Coker f$ are homomorphisms of real tori with integral structure.

Indeed, let $K=(\Coker f^*)^{tf}$ be the quotient of $\Coker f^*$ by its torsion subgroup (equivalently, the quotient of $\La_1$ by the saturation of $\Im f^*$), and let $K'=\Ker f_*$. Then $\Hom(K,\RR)$ is naturally identified with the kernel of the map $\Hom(\La_1,\RR)\to \Hom (\La_2,\RR)$ dual to $f^*$, and therefore $(\Ker f)_0=(K,K',[\cdot,\cdot]_K)$, where $[\cdot,\cdot]_K:K'\times K\to \RR$ is the pairing induced by $[\cdot,\cdot]_1$. We note that this pairing is well-defined: given $\la'_1\in K'$ and $\la_2\in \La_2$, Equation~\eqref{eq:integral} implies that
$$
[\la'_1,f^*(\la_2)]_1=[f_*(\la'_1),\la_2]_2=[0,\la_2]_2=0.
$$
Therefore, for $\la'\in K'$ and $\la\in K$, the pairing $[\la',\la]_K=[\la',\la]_1$ does not depend on a choice of representative for $\la\in K$. The natural maps $i^*:\La_1\to K$ and $i_*:K'\to \La_1'$ define $(\Ker f)_0$ as an integral subtorus of $\Si_1$. Similarly, $\Coker f=(C,C',[\cdot,\cdot]_C)$, where $C=\Ker f^*$, $C'=(\Coker f_*)^{tf}$, the pairing $[\cdot,\cdot]_C$ is induced by $[\cdot,\cdot]_2$, and $p$ is given by the natural maps $p_*:\La'_2\to C'$ and $p^*:C\to \La_2$. We note that a morphism $f$ of real tori with integral structure has finite kernel if and only if $K$ and $K'$ are trivial, in other words if $f_*$ is injective (equivalently, if $\Im f^*$ has finite index in $\La_1$).

Let $\Si=(\La,\La',[\cdot,\cdot])$ be a real torus with integral structure. A \emph{polarization} on $\Si$ is a map $\xi:\La'\to \La$ (necessarily injective) with the property that the induced bilinear form
$$
(\cdot,\cdot):\La'\times \La'\to \RR,\quad (\la',\mu')=[\la',\xi(\mu')]
$$
is symmetric and positive definite. Given a polarization $\xi$ on $\Si$, the pair $(\xi,\xi)$ defines a homomorphism $\eta:\Si\to \Si'$ to the dual, whose finite kernel is identified with $\La/\Im \xi$. The pair $(\Si,\xi)$ is called a \emph{tropical polarized abelian variety}. The map $\eta$ is an isomorphism if and only if $\xi$ is an isomorphism, in which case we say that the polarization $\xi$ is \emph{principal}. 

Let $\Si=(\La,\La',[\cdot,\cdot])$ be a $g$-dimensional tropical polarized abelian variety. The associated bilinear form $(\cdot,\cdot)$ on $\La'$ extends to an inner product on the universal cover $V=\Hom(\La,\RR)$, which we also denote $(\cdot,\cdot)$, and hence to a translation-invariant Riemannian metric on $\Si$. Let $C\subset \Si$ be a parallelotope framed by vectors $v_1,\ldots,v_g\in V$, then the volume of $C$ is equal to the square root $\sqrt{\det(v_i,v_j)}$ of the Gramian determinant of the $v_i$. In particular, if $\la'_1,\ldots,\la'_g$ is a basis of $\La'$, then
$$
\Vol^2(\Si)=\det(\la'_i,\la'_j).
$$

Finally, let $f:\Si_1\to \Si_2$ be a homomorphism of real tori with integral structures given by $f^*:\La_2\to \La_1$ and $f_*:\La'_1\to \La'_2$, and assume that $f$ has finite kernel (equivalently, $f_*$ is injective). Given a polarization $\xi_2:\La'_2\to \La_2$ on $\Si_2$ with associated bilinear form $(\cdot,\cdot)_2$, we define the \emph{induced polarization} $\xi_1:\La'_1\to \La_1$ by $\xi_1=f^*\circ \xi_2\circ f_*$. This is indeed a polarization, because by~\eqref{eq:integral} the associated bilinear form $(\cdot,\cdot)_1$ on $\La'_1$ is given by
$$
(\la'_1,\mu'_1)_1=[\la'_1,\xi_1(\mu'_1)]_1=[\la'_1,f^*(\xi_2(f_*(\mu'_1)))]_2=
[f_*(\la'_1),\xi_2(f_*(\mu'_1))]_2=(f_*(\la'_1),f_*(\mu'_1))_2,
$$
so it is symmetric and positive definite because $f_*$ is injective. Hence, in particular, an integral subtorus $i:\Pi\to \Si$ of a tropical polarized abelian variety $(\Si,\xi)$ has an induced polarization, which we denote $i^*\xi$. We note that the polarization induced by a principal polarization is not necessarily itself principal.

\begin{remark} In~\cite{Len_Ulirsch_Skeletons}, a real torus with integral structure is defined as a torus $\Si=N_{\RR}/\La$ with a distinguished lattice $N\subset N_R$ in the universal cover, and a morphism $f:\Si_1\to \Si_2$ as a map $\overline{f}:N_{1,\RR}\to N_{2,\RR}$ satisfying $\overline{f}(\La_1)\subset \La_2$ and induced by a $\ZZ$-linear map $N_1\to N_2$. It is easy to see that this definition is equivalent to ours. 
\label{rem:tppavs}
\end{remark}

\subsection{The Jacobian of a metric graph} 

We now construct the Jacobian variety $\Jac(\Ga)$ of a metric graph $\Ga$ of genus $g$ as a tropical ppav, following~\cite{BakerFaber_tropicalAbelJacobi} and~\cite{Len_Ulirsch_Skeletons}. We first pick an oriented model $G$ of $\Ga$ and consider the corresponding simplicial homology groups. Let $A$ be either $\ZZ$ or $\RR$, and let $C_0(G,A)=A^{V(G)}$ and $C_1(G,A)=A^{E(G)}$ be respectively the \emph{simplicial $0$-chain and $1$-chain groups} of $G$ with coefficients in $A$. The source and target maps $s,t:E(G)\to V(G)$ induce a boundary map
$$
d_A:C_1(G,A)\to C_0(G,A), \quad \sum_{e\in E(G)}a_e e\mapsto
\sum_{e\in E(G)}a_e[t(e)-s(e)], 
$$
and the \emph{first simplicial homology group} of $G$ with coefficients in $A$ is $H_1(G,A)=\Ker d_A$. We also consider the group of \emph{$A$-valued harmonic $1$-forms} $\Om(G,A)$ on $G$, which is a subgroup of the free $A$-module with basis $\{de:e\in E(G)\}$:
$$
\Om(G,A)=\left\{\sum_{e\in E(G)} \omega_e de:\sum_{e:t(e)=v}\omega_e=\sum_{e:s(e)=v}\omega_e\mbox{ for all }v\in V(G)\right\}.
$$
We note that mathematically $H_1(G,A)$ and $\Om(G,A)$ are the same object, but it is convenient to distinguish them, both for historical purposes and for clarity of exposition.

We now define an \emph{integration pairing}
$$
[\cdot,\cdot]:C_1(G,A) \times \Om(G,A) \to \RR
$$
by
$$
[\gamma, \omega]=\int_{\gamma}\omega=\sum_{e\in E(G)}\gamma_e \omega_e \ell(e),\quad
\gamma=\sum_{e\in E(G)}\gamma_ee,\quad \omega=\sum_{e\in E(G)}\omega_ede.
$$
By Lemma 2.1 in~\cite{BakerFaber_tropicalAbelJacobi}, the integration pairing restricts to a perfect pairing on $H_1(G,A)\times \Om(G,A)$. 

Let $G'$ be the model of $\Ga$ obtained by subdividing the edge $e\in E(G)$ into two edges $e_1$ and $e_2$, oriented in the same way as $e$, with $\ell(e_1)+\ell(e_2)=\ell(e)$. The natural embedding $C_1(G,A)\to C_1(G',A)$ sending $e$ to $e_1+e_2$ restricts to an isomorphism $H_1(G,A)\to H_1(G',A)$. Similarly, the groups $\Om(G,A)$ and $\Om(G',A)$ are naturally isomorphic, and these isomorphisms preserve the integration pairing. Hence we can define $\Om(\Ga,A)=\Om(G,A)$ and $H_1(\Ga,A)=H_1(G,A)$ for any model $G$, and by a $1$-chain, or \emph{path}, on $\Ga$ we mean a $1$-chain on any model of $\Ga$.

We now let $\La=\Om(\Ga,\ZZ)$ and $\La'=H_1(\Ga,\ZZ)$, let $[\cdot,\cdot]:\La'\times \La\to \RR$ be the integration pairing, and let $\xi:H_1(\Ga,\ZZ)\to \Om(\Ga,\ZZ)$ be the natural isomorphism sending the $1$-cycle $\sum a_e e$ to the $1$-form $\sum a_e de$. We denote $\Om^*(\Ga)=\Hom(\Om(\Ga,\ZZ),\RR)$, and by the universal coefficient theorem the group $\Hom(H_1(\Ga,\ZZ),\RR)$ is canonically isomorphic to $H^1(\Ga,\RR)$. The \emph{Jacobian variety} and the \emph{Albanese variety} of $\Ga$ are the dual tropical ppavs
$$
\Jac(\Ga)=\Om(\Ga)^*/H_1(\Ga,\ZZ),\quad \Alb(\Ga)=H^1(\Ga,\RR)/\Om(\Ga,\ZZ). 
$$
The group $H_1(\Ga,\ZZ)$ carries an intersection form
\begin{equation}
\label{eq:intersectionform}
(\cdot,\cdot)=[\cdot,\xi(\cdot)]:H_1(\Ga,\ZZ)\times H_1(\Ga,\ZZ)\to \RR,\quad \left(\sum_{e\in E(G)}\ga_ee,\sum_{e\in E(G)}\delta_ee\right)=\sum_{e\in E(G)} \ga_e\delta_e \ell(e)
\end{equation}
that induces an inner product on $\Omega^*(\Ga)$.

Fix a point $q\in \Ga$, and for any $p\in \Ga$ choose a path $\ga(q,p)\in C_1(\Ga,\ZZ)$ from $q$ to $p$. Integrating along $\ga(q,p)$ defines an element of $\Om(\Ga)^*$, and choosing a different path $\ga'(q,p)$ defines the same element modulo $H_1(\Ga,\ZZ)\subset \Om(\Ga)^*$. Hence we have a well-defined \emph{Abel--Jacobi map} $\Phi_q:\Ga\to \Jac(\Ga)$ with base point $q$:
\begin{equation}
\Phi_q:\Ga\to \Jac(\Ga),\quad p\mapsto\left(\omega\mapsto \int_{\ga(q,p)}\omega\right).
\label{eq:AJ}
\end{equation}
The map $\Phi_q$ extends by linearity to $\Div(\Ga)$, and its restriction to $\Div^0(\Ga)$ does not depend on the choice of base point $q$. The tropical analogue of the Abel--Jacobi theorem (see~\cite{MikhalkinZharkov}, Theorem 6.3) states that $\Phi_q$ descends to a canonical isomorphism $\Pic^0(\Ga)\simeq \Jac(\Ga)$. Since any $\Pic^k(\Ga)$ is a torsor over $\Pic^0(\Ga)$, we can define $\Vol(\Pic^k(\Ga))=\Vol(\Jac(\Ga))$.

Finally, we recall the principal results~\cite{ABKS_Canonical}, which concern the tropical Jacobi inversion problem. Consider the degree $g$ Abel--Jacobi map
$$
\Phi:\Sym^g(\Ga)\to \Pic^g(\Ga),\quad \Phi(p_1,\ldots,p_g)=p_1+\cdots+p_g.
$$
A choice of model $G$ for $\Ga$ defines a cellular decomposition
$$
\Sym^g(\Ga)=\bigcup_{F\in \Sym^g(E(G))}C(F),
$$
where for a multiset $F=\{e_1,\ldots,e_g\}\in \Sym^g(E(G))$ of $g$ edges of $G$ the cell $C(F)$ consists of divisors supported on $F$:
$$
C(F)=\{p_1+\cdots+p_g:p_i\in e_i\}.
$$
We say that $F$ is a \emph{break set} if all $e_i$ are distinct and $G\backslash F$ is a tree, and the set of \emph{break divisors} is the union of the cells $C(F)$ over all break sets $F$.

The map $\Phi$ is affine linear on each cell $C(F)$, and has maximal rank precisely when $F$ is a break set. Specifically, the following is true:

\begin{enumerate}
    \item If $F=\{e_1,\ldots,e_g\}$ is a break set, then the restriction of $\Phi$ to $C(F)$ is injective, and
    \begin{equation}
        \Vol(\Phi(C(F)))=\frac{w(F)}{\Vol(\Jac(\Ga))},\quad w(F)=\Vol (C(F))=\ell(e_1)\cdots \ell(e_g).
    \label{eq:volumeofcell}
    \end{equation}
    \item If $F$ is not a break set, then the restriction of $\Phi$ to $C(F)$ does not have maximal rank, and $\Vol(\Phi(C(F)))=0$.
\end{enumerate}

Furthermore, the map $\Phi$ has a unique continuous section whose image is the set of break divisors. Hence the images of the break cells $C(F)$ cover $\Pic^g(\Ga)$ with no overlaps in the interior of cells, and summing together their volumes gives $\Vol(\Jac(\Ga))=\Vol(\Pic^g(\Ga))$:

\begin{theorem}[Theorem 1.5 of~\cite{ABKS_Canonical}] The volume of the Jacobian variety of a metric graph $\Ga$ of genus $g$ is given by
\begin{equation}
\label{eq:ABKSformula}
\Vol^2(\Jac(\Ga))=\sum_{F\subset E(\Ga)}w(F),
\end{equation}
where the sum is taken over $g$-element subsets $F\subset E(\Ga)$ such that $\Ga\backslash F$ is a tree.
\end{theorem}

\begin{remark} This result can be interpreted as saying that $\Phi$ is a \emph{harmonic morphism of polyhedral spaces of degree 1}, where we define the local degree of $\Phi$ on a cell $C(F)$ to be $1$ if $F$ is a break set and $0$ otherwise. Indeed, the harmonicity condition ensures that such a map has a unique continuous section, since each cell of $\Pic^g(\Ga)$ has a unique preimage in $\Sym^g(\Ga)$ and these preimages fit together along codimension one cells. Formula~\eqref{eq:volumeofcell} then implies that the map $\Phi$ has a common volume dilation factor $1/\Vol(\Jac(\Ga))$ on all non-contracted cells.
\label{rem:ABKSdeg1}
\end{remark}

\begin{remark} We also note that, from the point of view of the Riemannian geometry of $\Jac(\Ga)$, the edge lengths on $\Ga$ are measured in units of $[\mbox{length}]^2$, not $[\mbox{length}]$. This is already clear from Formula~\eqref{eq:intersectionform} for the intersection form. Hence, for example, if $\Ga$ is a circle of length $L$ (in other words consists of a single loop of length $L$ attached to a vertex), then $\Ga$ is canonically isomorphic to $\Pic^1(\Ga)$, but the volume of $\Jac(\Ga)$ is $\sqrt{L}$, rather than $L$.

\label{rem:wrongunits}
\end{remark}

\subsection{Prym groups}
\label{sec:PrymGroups}
We now discuss the Prym group of a free double cover of finite graphs. Unlike the case of metric graphs (which we treat in Section~\ref{subsec:Prymvarieties}), finite groups don't have a distinguished connected component of the identity. We therefore require a notion of parity on elements of the kernel of the norm map.

Let $p:\tG\to G$ be a free double cover of graphs. The induced maps $\Nm:\Div(\tG)\to \Div(G)$ and $\iota:\Div(\tG)\to \Div(\tG)$ given by
$$
\Nm\left(\sum a_v v\right)=\sum a_v p(v),\quad \iota\left(\sum a_v v\right)=\sum a_v \iota(v)
$$
preserve degree and linear equivalence, and descend to give a surjective map $\Nm:\Jac(\tG)\to \Jac(G)$ and an isomorphism $\iota:\Jac(\tG)\to \Jac(\tG)$.

A divisor in the kernel $D\in \Ker \Nm\subset \Div(\tG)$ has degree zero and can be uniquely represented as $D=E-\iota(E)$, where $E$ is an effective divisor and the supports of $E$ and $\iota(E)$ are disjoint. We define the \emph{parity} of $D$ as
$$
\epsilon(D)=\deg E\bmod 2.
$$
It turns out that parity respects addition and linear equivalence, and hence gives a surjective homomorphism from $\Ker \Nm\subset \Jac(\tG)$ to $\ZZ/2\ZZ$: 

\begin{proposition} Let $D_1,D_2\in \Ker \Nm\subset \Div^0(\tG)$.

\begin{enumerate}
    \item $\epsilon(D_1+D_2)=\epsilon(D_1)+\epsilon(D_2)$.
    \item If $D_1\simeq D_2$ then $\epsilon(D_1)=\epsilon(D_2)$.
\end{enumerate}

\end{proposition}

\begin{proof} Suppose that $p:\tG\to G$ is defined by a spanning tree $T\subset G$ and a nonempty subset $S\subset E(G)\backslash E(T)$, as in Construction~\ref{con:A}. Every divisor $D\in\Ker \Nm\subset \Div(\tG)$ is of the form
$$
D=\sum_{v\in V(G)}(a_{\tv^+}\tv^++a_{\tv^-}\tv^-),
$$
where $a_{\tv^+}+a_{\tv^-}=0$ for each $v\in V(G)$. It follows that if $D=E-\iota(E)$ then $\deg E=\sum |a_{\tv^+}|$, hence
$$
\epsilon(D)=\sum |a_{\tv^+}|\bmod 2=\sum a_{\tv^+}\bmod 2,
$$
which is clearly preserved by addition.

To complete the proof, we need to show that any principal divisor in $\Ker \Nm\subset \Div(\tG)$ is even. Consider an arbitrary principal divisor
\[
D=\sum_{v\in V(G)}(c_{\tv^+}D_{\tv^+}+c_{\tv^-}D_{\tv^-})
\]
on $\tG$. Its norm is $\Nm(D)=\sum (c_{\tv^+}+c_{\tv^-})D_v\in \Div(G)$, which is the trivial divisor if and only if $c_{\tv^+}+c_{\tv^-}=c$ for a fixed $c\in \ZZ$ and for all $v\in V(G)$. Therefore, if $\Nm(D)=0$ in $\Div(G)$, then setting $a_v=c_{\tv^+}-c=-c_{\tv^-}$ we see that 
\[
D=cD^+ + \sum_{v\in V(G)} a_v (D_{\tv^+} - D_{\tv^-}),
\]
where $D^+$ the principal divisor obtained by firing each vertex $\tv^+$ of the top sheet once, and $a_v\in\ZZ$. We now show that each summand above is even, so $D$ is even as well by the first part of the proof.

First, we consider divisors of the form $D_{\tv^+}-D_{\tv^-}$ for  $v\in V(G)$. Suppose that the double cover $p$ is described by Construction~\ref{con:A}. For any vertex $u\in V(G)$, denote $a_u$ and $b_u$ the number of edges between $u$ and $v$ in $E(G)\backslash S$ and $S$ respectively. Then 
\[
(D_{\tv^+} - D_{\tv^-})(\tu^\pm) = \pm(a_u-b_u),
\]
and 
\[
(D_{\tv^+} - D_{\tv^-}) (\tv^\pm)= -\sum_{u\neq v} \mp{}(a_u + b_u).
\]
It follows that the contribution from each vertex $u$ to the positive part of $D_{\tv^+} - D_{\tv^-}$ is $|a_u-b_u| + a_u + b_u = \max(2 a_u, 2 b_u)$, which is even. 

As for  $D^+$, a direct calculation shows that 
\[
D^+=\sum_{v\in V(G)}D_{\tv^+}=\sum_{e\in S} \left(\widetilde{s(e)}^-+\widetilde{t(e)}^--\widetilde{s(e)}^+-\widetilde{t(e)}^+\right),
\]
hence $D^+$ is even, and the proof is complete.

\end{proof}

\begin{definition} The {\it Prym group} $\Prym(\tG/G)\subset \Jac(\tG)$ of a free double cover $p:\tG\to G$ is the subgroup of even divisors in $\Ker \Nm$.
\label{def:Prym}
\end{definition}

It is clear that the order of the Prym group is equal to 
$$
|\Prym(\tG/G)|=\frac{1}{2}|\Ker \Nm|=\frac{|\Jac(\tG)|}{2|\Jac(G)|},
$$
and one of the principal results of our paper is a combinatorial formula~\eqref{eq:orderofPrym} for $|\Prym(\tG/G)|$. For now, we illustrate with an example.

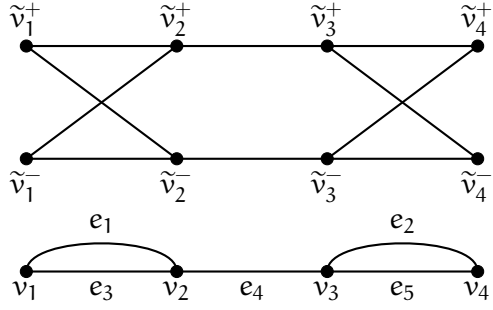
\begin{figure}
    \centering
    \begin{tikzpicture}

\begin{scope}[shift={(0,0)}]

\draw[fill](0,0) circle(.8mm) node[below]{$v_1$};
\draw[fill](2,0) circle(.8mm) node[below]{$v_2$};
\draw[fill](4,0) circle(.8mm) node[below]{$v_3$};
\draw[fill](6,0) circle(.8mm) node[below]{$v_4$};

\draw[thick] (0,0) .. controls (0,0.5) and (2,0.5) .. (2,0) node[midway,above] {$e_1$};
\draw[thick] (0,0) -- (2,0) node[midway,below] {$e_3$};
\draw[thick] (2,0) -- (4,0) node[midway,below] {$e_4$};
\draw[thick] (4,0) -- (6,0) node[midway,below] {$e_5$};
\draw[thick] (4,0) .. controls (4,0.5) and (6,0.5) .. (6,0) node[midway,above] {$e_2$};

\end{scope}

\begin{scope}[shift={(0,1.5)}]

\draw[fill](0,0) circle(.8mm) node[below]{$\tv_1^-$};
\draw[fill](2,0) circle(.8mm) node[below]{$\tv_2^-$};
\draw[fill](4,0) circle(.8mm) node[below]{$\tv_3^-$};
\draw[fill](6,0) circle(.8mm) node[below]{$\tv_4^-$};

\draw[thick] (0,0) -- (2,0);
\draw[thick] (2,0) -- (4,0);
\draw[thick] (4,0) -- (6,0);

\draw[fill](0,1.5) circle(.8mm) node[above]{$\tv_1^+$};
\draw[fill](2,1.5) circle(.8mm) node[above]{$\tv_2^+$};
\draw[fill](4,1.5) circle(.8mm) node[above]{$\tv_3^+$};
\draw[fill](6,1.5) circle(.8mm) node[above]{$\tv_4^+$};

\draw[thick] (0,1.5) -- (2,1.5);
\draw[thick] (2,1.5) -- (4,1.5);
\draw[thick] (4,1.5) -- (6,1.5);

\draw[thick] (0,0) -- (2,1.5);
\draw[thick] (0,1.5) -- (2,0);

\draw[thick] (4,0) -- (6,1.5);
\draw[thick] (4,1.5) -- (6,0);

\end{scope}

\end{tikzpicture}
    \caption{An example of a free double cover}
    \label{fig:doublecover1}
\end{figure}

\begin{example} \label{ex:doublecover1} Consider the free double cover $p:\tG\to G$ shown in Fig.~\ref{fig:doublecover1}. In terms of Construction~\ref{con:A}, we can describe it by choosing $T\subset G$ to be the tree containing $e_3$, $e_4$, and $e_5$, and setting $S=\{e_1,e_2\}$. Using Kirchhoff's theorem, we find that $|\Jac (G')|=64$ and $|\Jac(G)|=4$, therefore $\Ker \Nm$ and $\Prym(\tG/G)$ have orders $16$ and $8$, respectively. The group $\Ker \Nm$ is spanned by the divisors $D_i=\tv_i^+-\tv_i^-$, where $i=1,2,3,4$, and an exhaustive calculation using Dhar's burning algorithm gives a complete set of relations on the $D_i$:
$$
2D_1=0,\quad 8D_2=0,\quad D_4=D_1+4D_2,\quad D_3=3D_2.
$$
It follows that $\Ker \Nm\simeq \ZZ/2\ZZ\oplus \ZZ/8\ZZ$ with generators $D_1$ and $D_2$, and hence $\Prym(\tG/G)\simeq \ZZ/8\ZZ$ with generator $D_1+D_2$.

We note that the Abel--Prym map $\tG=\Sym^1(\tG)\to \Prym^1(\tG/G)$ sending $\tv^{\pm}_i$ to $\pm D_i$ is not surjective: both sets have eight elements, but the images of $\tv^{\pm}_1$ are equal, as well as those of $\tv^{\pm}_4$.
\end{example}

\subsection{Prym varieties}\label{subsec:Prymvarieties} Finally, we recall the definition of the Prym variety of a free double cover $\pi:\tGa\to \Ga$ of metric graphs (see~\cite{JensenLen_thetachars} and~\cite{Len_Ulirsch_Skeletons}). As in the finite case, the cover $\pi$ induces a surjective \emph{norm} map 
$$
\Nm:\Pic^0(\tGa)\to \Pic^0(\Ga),\quad \Nm\left(\sum  a_i \tp_i\right) = \sum a_i\pi(\tp_i),
$$
and a corresponding involution $\iota:\Pic^0(\tGa)\to \Pic^0(\tGa)$.

The kernel $\Ker \Nm$ consists of divisors having a representative of the form $E-\iota(E)$ for some effective divisor $E$ on $\tGa$. Indeed, suppose that $\widetilde{D}$ is a divisor on $\widetilde\Gamma$ such that $\Nm(\widetilde{D})\simeq 0$. Then $\Nm(\widetilde{D}) +\ddiv{f} = 0$ for some piecewise linear function $f$. Defining $\tilde f(x) = f(\pi(x))$, we see that $\widetilde{D}$ is equivalent to a divisor whose pushforward is the zero divisor on the nose. Furthermore, the parity of $E$ is well-defined, and $\Ker \Nm$ has two connected components corresponding to the parity of $E$~\cite[Proposition 6.1]{JensenLen_thetachars} (note that, in the more general case when $\pi$ is a dilated unramified double cover, $\Ker \Nm$ has only one connected component).

\begin{definition}\label{def:metricPrym} The \emph{Prym variety} $\Prym(\tGa/\Ga)\subset \Pic^0(\tGa)$ of the free double cover $\pi:\tGa\to \Ga$ of metric graphs is the connected component of the identity of $\Ker \Nm$. 
\end{definition}

The Prym variety $\Prym(\tGa/\Ga)$ has the structure of a tropical ppav, which we now describe. Denote $\widetilde{\La}=\Om(\tGa,\ZZ)$, $\widetilde{\La}'=H_1(\tGa,\ZZ)$, $\La=\Om(\Ga,\ZZ)$ and $\La'=H_1(\Ga,\ZZ)$. Choose an oriented model $p:\tG\to G$ for $\pi$, and consider the pushforward and pullback maps $\pi_*:H_1(\tGa,\ZZ)\to H_1(\Ga,\ZZ)$ and $\pi^*:\Om(\Ga,\ZZ)\to \Om(\tGa,\ZZ)$ defined by
$$
\pi_*\left[\sum_{\te\in E(\tG)} a_{\te} \te\right]=\sum_{\te\in E(\tG)} a_{\te} \pi(\te),\quad \pi^* \left[\sum_{e\in E(G)}a_e de\right]= \sum_{e\in E(G)}a_e (d\te^++d\te^-).
$$
It is easy to verify that the maps $\pi_*$ and $\pi^*$ satisfy Equation~\eqref{eq:integral} with respect to the integration pairings on $\tGa$ and $\Ga$. Therefore the pair $(\pi_*,\pi^*)$ defines a homomorphism $\Nm:\Jac(\tGa)\to \Jac(\Ga)$ of real tori with integral structure (by Proposition 2.2.3 in~\cite{Len_Ulirsch_Skeletons}, this homomorphism is identified with the norm homomorphism $\Nm:\Pic^0(\tGa)\to \Pic^0(\Ga)$ under the Abel--Jacobi isomorphism, justifying the notation). Hence $\Prym(\tGa/\Ga)$ is in fact the real torus with integral structure $(\Ker\Nm)_0=(K,K',[\cdot,\cdot]_K)$, where $K=(\Coker \pi^*)^{tf}$, $K'=\Ker \pi_*$, and $[\cdot,\cdot]_K$ is the pairing induced by the integration pairing on $\tGa$. Alternatively, we can describe $\Prym(\tGa/\Ga)$ as the quotient
$$
\Prym(\tGa/\Ga)=\frac{\Ker \overline{\pi}:\Om^*(\tGa)\to \Om^*(\Ga)}{\Ker \pi_*:H_1(\tGa,\ZZ)\to H^1(\Ga,\ZZ)},
$$
where $\overline{\pi}$ is the map dual to $\pi^*$.

The polarization $\widetilde{\xi}:H_1(\tGa,\ZZ)\to \Om(\tGa,\ZZ)$ on $\Jac(\tGa)$ induces a polarization $i^*\widetilde{\xi}:K'\to K$ on $\Prym(\tGa/\Ga)$, and Theorem 2.2.7 in~\cite{Len_Ulirsch_Skeletons} states that there exists a \emph{principal} polarization $\psi:K'\to K$ on $\Prym(\tGa/\Ga)$ such that $i^*\widetilde{\xi}=2\psi$. Hence $\Prym(\tGa/\Ga)$ is a tropical ppav. We note that the inner product $(\cdot,\cdot)_P$ on $\Prym(\tGa/\Ga)$ induced by the principal polarization $\psi$ is half of the restriction of the inner product $(\cdot,\cdot)_{\tGa}$ from $\Jac(\tGa)$. In other words for $\ga, \de \in \Ker \pi_*$ we have
\begin{equation}
(\ga,\de)_P=[\ga,\psi(\de)]=\frac{1}{2}[\ga,\widetilde{\xi}(\de)]=\frac{1}{2}(\ga,\de)_{\tGa}=\frac{1}{2}\sum_{\te\in E(\tGa)}\ga_{\te}\de_{\te}\ell(\te),\quad \ga=\sum_{\te\in E(\tGa)}\ga_{\te}\te,\quad
\de=\sum_{\te\in E(\tGa)}\de_{\te}\te,
\label{eq:Prymproduct}
\end{equation}
and similarly for the induced product on $\Ker \overline{\pi}$. When discussing the metric properties of $\Prym(\tGa/\Ga)$, such as its volume, we always employ the inner product $(\cdot,\cdot)_P$ induced by the principal polarization. 

We use a set of explicit coordinates on the torus $\Prym(\tGa/\Ga)$, or more accurately on its universal cover $\Ker \overline{\pi}$. Choose a basis
$$
\tga_j=\sum_{\te\in E(\tG)} \tga_{j,\te}\te,\quad j=1,\ldots,g-1
$$
for $\Ker \pi_*:H_1(\tGa,\ZZ)\to H_1(\Ga,\ZZ)$. The principal polarization $\psi=\frac{1}{2}\widetilde{\xi}$ gives a corresponding basis of the second lattice $(\Coker \pi^*)^{tf}$:
$$
\om_j=\psi(\tga_j)=\frac{1}{2}\sum_{\te\in E(\tG)} \tga_{j,\te}d\te,\quad j=1,\ldots,g-1.
$$
Let $\om^*_j$ denote the basis of $\Ker \overline{\pi}:\Om^*(\tGa)\to \Om^*(\Ga)$ dual to the $\om_j$, so that $\om^*_j(\om_k)=\delta_{jk}$, then elements of $\Prym(\tGa/\Ga)$ can be given (locally uniquely) as linear combinations of the $\om^*_j$.

We compute for future reference the volume of the unit cube $C(\om^*_1,\ldots,\om^*_{g-1})$ in the coordinate system defined by the $\om^*_j$. We know that $\Vol(\Prym(\tGa/\Ga))=\sqrt{\det G}$, where $G_{ij}=(\tga_i,\tga_j)_P$ is the Gramian matrix of the basis $\tga_j$. The $\tga_j$, viewed as elements of $\Ker \overline{\pi}$, are themselves a basis, so we can write $\om^*_i=\sum_j A_{ij}\tga_j$ for some matrix $A_{ij}$. Pairing with $\om_j$ and using that $[\tga_i,\om_j]=G_{ij}$, we see that $A$ is in fact the inverse matrix of $G$. Hence we see that
\begin{equation}
\Vol(C(\om^*_1,\ldots,\om^*_{g-1}))=\det (\om^*_i,\om^*_j)=\det G^{-1} \det(\tga_i,\tga_j)\det G^{-1}=\frac{1}{\det G}=\frac{1}{\Vol(\Prym(\tGa/\Ga))}.
\label{eq:volumeofunitcube}
\end{equation}
In particular, this volume does not depend on the choice of basis $\tga_j$.

\begin{remark}\label{rem:Pryms}
The definition of the Prym group for a free double cover of finite graphs is consistent with the definition for metric graphs in the following sense. Let $p:\tG\to G$ be a free double cover of finite graphs, and let $\pi:\tGa\to\Ga$ be the corresponding double cover of metric graphs, where $\tGa$ and $\Ga$ are obtained from respectively $\tG$ and $G$ by setting all edge lengths to $1$. Then $\Jac(\tG)$ is naturally a subgroup of $\Jac(\tGa)$, consisting of divisors supported at the vertices, and $\Prym(\tG/G)=\Jac(\tG)\cap \Prym(\tGa/\Ga)$. 
\end{remark}

\subsection{Polyhedral spaces and harmonic morphisms}

The spaces $\Sym^d(\Ga)$, $\Jac(\Ga)$, and $\Prym(\tGa/\Ga)$ are examples of \emph{rational polyhedral spaces}, which are topological spaces locally modeled on rational polyhedral sets in $\RR^n$. A rational polyhedral space comes equipped with a structure sheaf, pulled back from the sheaf of affine $\ZZ$-linear functions on the embedded polyhedra. We shall not require the general theory of rational polyhedral spaces, in particular we shall use only the polyhedral decomposition and not the sheaf of affine functions. See \cite{mikhalkin2014tropical}, \cite{gross2019tautological} for details.

A rational polyhedral space $P$ is a finite union of polyhedra, which we call \emph{cells}. We only consider compact polyhedral spaces. The intersection of any two cells is either empty or a face of each. A polyhedral space $P$ is \emph{equidimensional of dimension $n$} if each maximal cell of $P$ (with respect to inclusion) has dimension $n$, and is \emph{connected through codimension one} if the complement in $P$ of all cells of codimension two is connected. A map $f:P\to Q$ of polyhedral spaces is locally given by affine $\ZZ$-linear transformations, and is required to map each cell of $P$ surjectively onto a cell of $Q$. We say that $f$ \emph{contracts} a cell $C$ of $P$ if $\dim(f(C))<\dim(C)$. 

We use an ad hoc definition of harmonic morphisms of polyhedral spaces, modelled on the corresponding definition for metric graphs.

\begin{definition}\label{def:harmonic}[cf.~Definition 2.5 in~\cite{LenRanganathan_EllipticCurves}] Let $f:P\to Q$ be a map of equidimensional polyhedral spaces of the same dimension, and let $\deg_f$ be a non-negative integer-valued function defined on the top-dimensional cells of $P$. Let $C$ be a codimension 1 cell of $P$ mapping surjectively onto a codimension one cell $D$ of $Q$. We say that $f$ is \emph{harmonic} at $C$ (with respect to the degree function $\deg_f$) if the following condition holds: for any codimension zero cell $N$ of $Q$ adjacent to $D$, the sum
\begin{equation}
\deg_f(C)=\sum_{M\subset f^{-1}(N),\, M\supset C} \deg_f (M)
\label{eq:harmonicity}
\end{equation}
of the degrees $\deg_f M$ over all codimension zero cells $M$ of $P$ adjacent to $C$ and mapping to $N$ is the same, in other words does not depend on the choice of $N$. We say that $f$ is \emph{harmonic} if $f$ is harmonic at every codimension one cell of $P$, and in addition if $f(C)=0$ on a codimension zero cell $C$ if and only if $f$ contracts $C$.
\end{definition}

Given a harmonic morphism $f:P\to Q$, Equation~\eqref{eq:harmonicity} extends the degree function $\deg_f$ to codimension one cells of $P$. If $Q$ is connected through codimension one, we can similarly define the degree on cells of any codimension, and hence on all of $P$ (note, however, that for a cell $C$ of positive codimension, $\deg_f(C)=0$ does not imply that $C$ is contracted). The function $\deg_f$ is locally constant in fibers: given $p\in P$ and an open neighborhood $V\ni f(p)$, there exists an open neighborhood $U\ni p$ such that $f(U)\subset V$, and such that for any $q\in V$ the sum of the degrees over all points of $f^{-1}(q)\cap U$ is the same (in particular, this sum is finite). It follows that a harmonic morphism to a target connected through codimension one is surjective, and has a well-defined \emph{global degree}, which is the sum of the degrees over all points of any fiber. 

The structure of a rational polyhedral space on a tropical abelian variety, such as $\Jac(\Ga)$ and $\Prym(\tGa/\Ga)$, is induced by the integral structure: an affine linear function is $\ZZ$-linear if it has integer slopes with respect to the embedded lattice. The correct definition of a rational polyhedral structure on $\Sym^d(\Ga)$, however, requires some care. The space $\Sym^d(\Ga)$ has a natural cellular decomposition, with top-dimensional cells $C(F)$ indexed by $d$-tuples $F\subset E(G)$ of edges of a suitably chosen model $G$ of $\Ga$ (see Eq.~\eqref{eq:AP}). If the $d$-tuple $F=\{e_1,\ldots,e_d\}$ contains neither loops nor repeated edges, then $C(F)$ is the parallelotope obtained by taking the Cartesian product of the $e_i$ inside $\RR^d$. If $F$ contains loop edges, then the corresponding cell will have self-gluings along some of its boundary cells. A more serious issue arises if $F$ has repeated edges; in this case the cell $C(F)$ is a quotient of a parallelotope by a finite permutation action, which requires introducing additional cells at the appropriate diagonals. 

A complete description of the polyhedral structure on $\Sym^d(\Ga)$ is given in the paper~\cite{brandt2018symmetric}. However, neither of the aforementioned issues arise in our paper. First, we can always choose a loopless model for any tropical curve. Second, and more importantly, all the calculations in this paper concern only those cells $C(F)$ for which $F$ has no repeated edges. Specifically, we are interested in the structure of the Abel--Prym map $\Psi^d:\Sym(\tGa)\to \Prym(\tGa/\Ga)$ associated to a double cover $\tGa\to \Ga$, and one of our first results (Thm.~\ref{thm:APlocal} part (1)) states that $\Psi^d$ contracts all those cells $C(F)$ where $F$ has repeated edges. Hence, all top-dimensional cells of the symmetric product $\Sym^d(\Ga)$ can be assumed to be parallelotopes.

\section{Kirchhoff's theorem for the Prym group and the Prym variety}

In this section, we give combinatorial formulas for the order~\eqref{eq:orderofPrym} of the Prym group of a free double cover $p:\tG\to G$ of finite graphs, and the volume~\eqref{eq:Prymvolume} of the Prym variety of a free double cover $\pi:\tGa\to \Ga$ of metric graphs. 

Formula~\eqref{eq:orderofPrym} had already been obtained by Zaslavsky (see Theorem 8A.4 in~\cite{zaslavsky1982signed}). Specifically, a free double cover $p:\tG\to G$ induces the structure of a \emph{signed graph} on $G$: defining the cover $p$ in terms of Construction~\ref{con:A} with respect to a spanning tree $T\subset G$, we attach a negative sign $-$ to each edge $e\in S\subset E(G)\backslash E(T)$ and a positive sign $+$ to all other edges. Zaslavsky then defines the \emph{signed Laplacian matrix} of $G$ and shows that its determinant is given by~\eqref{eq:orderofPrym} (note that the signed Laplacian is non-singular, unlike the ordinary Laplacian). Reiner and Tseng specifically interpret the determinant of the signed Laplacian as the order of the Prym group $\Prym(\tG/G)$ (see Proposition 9.9 in~\cite{reiner2014critical}).

We give an alternative proof of~\eqref{eq:orderofPrym} using the Ihara zeta function. Given a free double cover $p:\tG\to G$, the orders of $\Jac(\tG)$ and $\Jac(G)$ can be computed from the corresponding zeta functions $\zeta(\tG,s)$ and $\zeta(G,s)$ using Northshield's class number formula~\cite{northshield1998note}. Hence the ratio $|\Prym(\tG/G)|=|\Jac(\tG)|/2|\Jac(G)|$ is given by the ratio of the zeta functions. This is equal to the Artin--Ihara $L$-function of the cover and can be explicitly computed from the corresponding determinantal formula, derived by Stark and Terras (see~\cite{stark1996zeta} and~\cite{stark2000zeta}).

The volume formula~\eqref{eq:Prymvolume} is new, to the best of our knowledge, and is derived from~\eqref{eq:orderofPrym} by a scaling argument.

\subsection{The Ihara zeta function and the Artin--Ihara $L$-function}

The Ihara zeta function $\zeta(s,G)$ of a finite graph $G$ is the graph-theoretic analogue of the Dedekind zeta function of a number field and is defined as an Euler product over certain equivalence classes of closed paths on $G$. We recall its definition and properties (see~\cite{terras2010zeta} for an elementary treatment).

Let $G$ be a graph with $n=\#(V(G))$ vertices and $m=\#(E(G))$ edges. A \emph{path} $P$ of \emph{length} $k=\ell(P)$ is a sequence $P=e_1\cdots e_k$ of oriented edges of $G$ such that $t(e_i)=s(e_{i+1})$ for $i=1,\ldots,k-1$. We say that a path $P$ is \emph{closed} if $t(e_k)=s(e_1)$ and \emph{reduced} if $e_{i+1}\neq \oe_i$ for $i=1,\ldots,k-1$ and $e_1\neq \oe_k$. We can define positive integer powers of closed paths by concatenation, and a closed reduced path $P$ is called \emph{primitive} if there does not exist a closed path $Q$ such that $P=Q^k$ for some $k\geq 2$. We consider two reduced paths to be \emph{equivalent} if they differ by a choice of starting point, i.e. we set $e_1\cdots e_k\sim e_j\cdots e_k\cdot e_1\cdots e_{j-1}$ for all $j=1,\ldots k$. A \emph{prime} $\frakp$ of $G$ is an equivalence class of primitive paths, and has a well-defined length $\ell(\frakp)$. We note that a primitive path and the same path traversed in the opposite direction represent distinct primes. 

The \emph{Ihara zeta function} $\zeta(s,G)$ of a graph $G$ is the product
$$
\zeta(s,G)=\prod_{\frakp}(1-s^{\ell(\frakp)})^{-1}
$$
over all primes $\frakp$ of $G$, where $s$ is a complex variable. This product is usually infinite, converges for sufficiently small $s$, and extends to rational function.

The \emph{three-term determinant formula}, due to Bass~\cite{bass1992ihara} (see also~\cite{terras2010zeta}), expresses the reciprocal $\zeta(s,G)^{-1}$ as an explicit polynomial
$$
\zeta(s,G)^{-1}=(1-s^2)^{g-1}\det(I_n-As+(Q-I_n)s^2),
$$
where $Q$ and $A$ are the valency and adjacency matrices (see~\eqref{eq:QA}), and $g=m-n+1$ is the genus of $G$. It is clear from this formula that $\zeta(s,G)^{-1}$ vanishes at $s=1$ to order at least $g$, because for $s=1$ the matrix inside the determinant is equal to the Laplacian $L$ of $G$ and $\det L=0$. In fact, the order of vanishing is equal to $g$, and Northshield ~\cite{northshield1998note} shows that the leading Taylor coefficient computes the complexity, i.e.~the order of the Jacobian of $G$:
\begin{equation}
\zeta(s,G)^{-1}=(-1)^{g-1}2^g(g-1)|\Jac(G)|(s-1)^g+O((s-1)^{g+1}).
\label{eq:Northshield}
\end{equation}
This result may be viewed as a graph-theoretic analogue of the class number formula.

The analogy with number theory was further reinforced by Stark and Terras, who developed (see~\cite{stark1996zeta} and~\cite{stark2000zeta}) a theory of $L$-functions of Galois covers of graphs, as follows. Let $p:\tG\to G$ be a free Galois cover of graphs with Galois group $K$ (we do not define these, since we only consider free double covers, which are Galois covers with $K=\ZZ/2\ZZ$), and fix a representation $\rho$ of $K$. Given a prime $\frakp$ of $G$, choose a representative $P$ with starting vertex $v\in V(G)$, and choose a vertex $\widetilde{v}\in V(\tG)$ lying over $v$. The path $P$ lifts to a unique path $\widetilde{P}$ in $\tG$ starting at $\widetilde{v}$ and mapping to $P$, and the terminal vertex of $\widetilde{P}$ also maps to $v$. The \emph{Frobenius element} $F(P,\tG/G)\in K$ is the unique element of the Galois group mapping $\widetilde{v}$ to the terminal vertex of $\widetilde{P}$.
The \emph{Artin--Ihara $L$-function} is now defined as the product
$$
L(s,\rho,\tG/G)=\prod_{\frakp} \det(1-\rho(F(P,\tG/G))s^{\ell(\frakp)})^{-1}
$$
taken over the primes $\frakp$ of $G$, where for each prime $\frakp$ we pick an arbitrary representative $P$ (Frobenius elements corresponding to different representatives of $\frakp$ are conjugate, so the determinant is well-defined).

Similarly to the zeta function, the product defining the $L$-function converges to a rational function and is given by a determinant formula. Pick a spanning tree $T\subset G$ and index its preimages in $\tG$, called the \emph{sheets} of the cover, by the elements of $K$. Given an edge $e\in E(G)$, the \emph{Frobenius element} $F(e)\in K$ is equal to $h^{-1}g$, where $h$ and $g$ are respectively the indices of the sheets of the source and the target of $e$. Let $d$ be the degree of $\rho$, and define the $nd\times nd$ \emph{Artinized valency} and \emph{Artinized adjacency} matrices as 
$$
Q_{\rho}=Q\otimes I_d,\quad (A_{\rho})_{uv}=\sum\rho(F(e)),
$$
where in the right hand side we sum over all edges $e$ between $u$ and $v$. The three-term determinant formula for the $L$-function states that 
\begin{equation}
L(s,\rho,\tG/G)^{-1}=(1-s^2)^{(g-1)d}\det(I_{nd}-A_{\rho}s+(Q_{\rho}-I_{nd})s^2).
\label{eq:L3term}
\end{equation}

Finally, we relate the zeta and $L$-functions associated to a free Galois cover $p:\tG\to G$ with Galois group $K$. First of all, the zeta functions of $\tG$ and $G$ are equal to the $L$-function evaluated at respectively the right regular and trivial representations $\rho_K$ and $1_K$:
$$
\zeta(s,\tG)=L(s,\rho_K,\tG/G),\quad \zeta(s,G)=L(s,1_K,\tG/G).
$$
Furthermore, for a reducible representation $\rho=\rho_1\oplus \rho_2$ the $L$-function factors as
$$
L(s,\rho,\tG/G)=L(s,\rho_1,\tG/G)L(s,\rho_2,\tG/G).
$$
It follows that the zeta function of $\tG$ has a factorization
\begin{equation}
\zeta(s,\tG)=\zeta(s,G)\prod_{\rho}L(s,\rho,\tG/G)^{d(\rho)},
\label{eq:zetafactorization}
\end{equation}
where the product is taken over the distinct nontrivial irreducible representations of $K$.

\subsection{The order of the Prym group}

\label{subsec:Prymorder}

We now specialize to the case where $K=\ZZ/2\ZZ$ in order to
compute the order of the Prym group of a free double cover $p:\tG\to G$ of finite graphs. By~\eqref{eq:Northshield}, the leading Taylor coefficients of the zeta functions $\zeta(s,\tG)^{-1}$ and $\zeta(s,G)^{-1}$ at $s=1$ compute respectively the orders $|\Jac(\tG)|$ and $|\Jac(G)|$. Since $\zeta^{-1}(s,\tG)$ is the product of $\zeta^{-1}(s,G)$ with the inverse of the $L$-function evaluated at the nontrivial representation of $\ZZ/2\ZZ$, the leading Taylor coefficient of the latter computes the order of the Prym. 

By the results of~\cite{ABKS_Canonical}, the Jacobian group of a graph $G$ of genus $g$ (and, by extension, the Jacobian variety of a metric graph) admits a non-canonical combinatorial description in terms of certain $g$-element subsets of $E(G)$, specifically the complements of spanning trees. We now give an analogous definition for $(g-1)$-element subsets of $E(G)$, which, as we shall see, enumerate the elements of $\Prym(\tG/G)$, and control the geometry of the Prym varieties of double covers of metric graphs.

\begin{definition} Let $G$ be a graph of genus $g$, and let $p:\tG\to G$ be a connected free double cover. A subset $F\subset E(G)$ of $g-1$ edges of $G$ is called a \emph{genus one decomposition of rank} $r$ if the graph  $G\backslash F=G_0\cup\cdots\cup G_{r-1}$ has $r$ connected components, each of which has genus one. We say that a genus one decomposition $F$ is \emph{odd} if the preimage of each $G_k$ in $\tG$ is connected.

\end{definition}

We note that when removing edges from a graph we never remove vertices, even isolated ones. A simple counting argument shows that if $F\subset E(G)$ is a subset such that each connected component of $G\backslash F$ has genus one, then $F$ consists of $g-1$ edges, and a genus one decomposition cannot have rank greater than $g$.

A genus one graph has two free double covers: the disconnected trivial cover and a unique nontrivial connected cover. Hence we can equivalently require that the restriction of the cover $p$ to each $G_k$ is a nontrivial free double cover. If the cover $p$ is described by Construction~\ref{con:A} with respect to a choice of spanning tree $T\subset G$ and a nonempty subset $S\subset E(G)\backslash E(T)$, then a genus one decomposition $F\subset E(G)$ is odd if and only if each $G_k$ has an odd number of edges from $S$ on its unique cycle (see Remark~\ref{rem:oddS}).

\begin{theorem} Let $G$ be a graph of genus $g$, and let $p:\tG\to G$ be the connected free double cover determined by $T\subset G$ and $S\subset E(G)\backslash E(T)$. The order of the Prym group $\Prym(\tG/G)$ is equal to
\begin{equation}
|\Prym(\tG/G)|=\frac{1}{2}|\Ker \Nm|=\sum_{r=1}^g 4^{r-1}C_r,
\label{eq:orderofPrym}
\end{equation}
where $C_r$ is the number of odd genus one decompositions of $G$ of rank $r$.
\label{thm:discretePrym}
\end{theorem}

\begin{proof} Denote $n=|V(G)|$ and $m=|E(G)|=n+g-1$. According to~\eqref{eq:zetafactorization}, the zeta function of $\tG$ is the product of the zeta function of $G$ and the $L$-function of the cover $\tG/G$ evaluated at the nontrivial representation $\rho$ of the Galois group $\ZZ/2\ZZ$:
$$
\zeta(s,\tG)^{-1}=\zeta(s,G)^{-1}L(s,\tG/G,\rho)^{-1}.
$$
The class number formula~\eqref{eq:Northshield} gives the leading Taylor coefficients at $s=1$:
$$
\zeta(s,\tG)^{-1}=2^{2g-1} (2g-2)|\Jac(\tG)|(s-1)^{2g-1}+O\left((s-1)^{2g}\right),
$$
$$
\zeta(s,G)^{-1}=(-1)^{g-1}2^g (g-1)|\Jac(G)|(s-1)^g+O\left((s-1)^{g+1}\right).
$$
The leading coefficient of the $L$-function is found directly from~\eqref{eq:L3term} (note that, unlike in formula~\eqref{eq:Northshield}, the determinant does not vanish at $s=1$):
$$
L(s,\rho,\tG/G)^{-1}=(-1)^{g-1}2^{g-1} \det(Q_{\rho}-A_{\rho}) (s-1)^{g-1}+O\left((s-1)^g\right). 
$$
Therefore, comparing the expansions of $L(s,\rho,\tG/G)^{-1}$ with $\zeta(s,\tG)^{-1}/\zeta(s,G)^{-1}$, we see that
$$
|\Prym(\tG/G)|=\frac{|\Jac(\tG)|}{2|\Jac(G)|}=\frac{1}{4}\det (Q_{\rho}-A_\rho).
$$
We now calculate this $n\times n$ determinant. First of all, $Q_{\rho}=Q$ since $\rho$ is one-dimensional. The Frobenius element $F(e)$ of an edge $e\in E(G)$ is the nontrivial element of $\ZZ/2\ZZ$, and hence $\rho(F(e))=-1$, if and only if $e\in S$. Putting this together, we see that the matrix $Q_{\rho}-A_{\rho}$ has the following form: 
$$
(Q_\rho-A_\rho)_{uv}=\left\{\begin{array}{cc}|\{\text{edges from $u$ to $v$ in $S$}\}|-
|\{\text{edges from $u$ to $v$ not in $S$}\}|, & u\neq v, \\
4|\{\text{loops at $u$ in $S$}\}|+|\{\text{non-loops at $u$}\}|,& u=v.
\end{array}\right.
$$
The matrix $Q_{\rho}-A_{\rho}$ turns out to be equal to the \emph{signed Laplacian matrix} of the graph $G$ (see Proposition 9.5 in~\cite{reiner2014critical}), and its determinant is computed using a standard argument involving an appropriate factorization and the Cauchy--Binet formula. We only give a sketch of these calculations, since they are not new (see Proposition 9.9 in~\emph{loc.~cit.}). 

Pick an orientation on $G$. We factorize the signed Laplacian as $Q_\rho-A_\rho=B_S(G){}^tB_S(G)$, where 
$$
(B_S(G))_{ve}=\left\{\begin{array}{cc} 1, & \text{$t(e)=v$ and $s(e)\neq v$, or $s(e)=v$, $t(e)\neq v$, and $e\in S$,}\\
-1, & \text{$s(e)=v$, $t(e)\neq v$, and $e\notin S$,}\\
2, & \text{$s(e)=t(e)=v$ and $e\in S$,}\\
 0, & \text{otherwise}.\end{array}\right.
$$
is the $n\times m$ \emph{$S$-twisted adjacency matrix} $B_S(G)$ of the graph $G$, whose rows and columns are indexed by respectively $V(G)$ and $E(G)$. By the Cauchy--Binet formula, we have
\begin{equation}
    |\Prym(\tG/G)|=\frac{1}{4}\det (Q_{\rho}-A_\rho)=\frac{1}{4}\sum_{F\subset E(G),\,|F|=g-1} \det B_S(G\backslash F)^2.
\label{eq:CauchyBinet}
\end{equation}
Here the sum is taken over all subsets $F$ of $E(G)$ consisting of $m-n=g-1$ elements, and $B_S(G\backslash F)$ is the matrix obtained from $B_S(G)$ by deleting the columns  corresponding to the edges that are in $F$, or, equivalently, the $S$-twisted adjacency matrix of the graph $G\backslash F$.

Let $F\subset E(G)$ be such a subset, and let $G\backslash F=G_0\cup\cdots\cup G_{r-1}$ be the decomposition of $G$ into connected components. The matrix $B_S(G\backslash F)$ is block-diagonal, with blocks $B_S(G_k)$ corresponding to the $G_k$. A block-diagonal matrix has nonzero determinant only if all blocks are square, meaning that $g(G_k)=1$ for all $k$, in which case
\begin{equation}
\det B_S(G\backslash F)^2=\prod_{k=0}^{r-1} \det B_S(G_k)^2.
\label{eq:product}
\end{equation}
The quantity $\det B_S(G_k)^2$ for a genus one graph $G_k$ is computed by induction on the extremal edges (if any), and turns out to be equal to $4$ if the unique cycle of $G_k$ has an odd number of edges from $S$, and $0$ if the number is even. Hence, only odd genus one decompositions contribute to the sum~\eqref{eq:CauchyBinet}, and the contribution of a decomposition of rank $r$ is equal to $4^r$. This completes the proof. 

\end{proof}

\begin{example} Consider the free double cover $p:\tG\to G$ shown in Figure~\ref{fig:doublecover1}. Here $g-1=1$, and it is easy to see that any edge of $G$ is an odd genus one decomposition. The edges $e_1$, $e_2$, $e_4$, and $e_5$ are decompositions of rank one, while $e_3$ is a decomposition of rank two. Hence by~\eqref{eq:orderofPrym}
$$
|\Prym(\tG/G)|=4+1\cdot 4=8,
$$
which agrees with the calculations in~Example~\ref{ex:doublecover1}.

\end{example}

\subsection{The volume of the tropical Prym variety} In this section, we prove a weighted version of Theorem~\ref{thm:discretePrym} that gives the volume of the Prym variety of a free double cover of metric graphs.
Let $\pi:\tGa\to \Ga$ be such a cover, where $\tGa$ and $\Ga$ have genera $2g-1$ and $g-1$, respectively. Choose a model $G$ for $\Ga$. Similarly to the discrete case, an \emph{odd genus one decomposition} $F$ of $\Ga$ of \emph{rank} $r(F)$ (with respect to the choice of model $G$) is a subset $F\subset E(G)$ of (necessarily) $g-1$ edges of $G$ such that $E(G)\backslash F$ consists of $r(F)$ connected components of genus one, each having a connected preimage in $\tGa$. 
For such an $F$, we denote by $w(F)$ the product of the lengths of the edges in $F$.

\begin{theorem} The volume of the Prym variety of a free double cover $\pi:\tGa\to \Ga$ of metric graphs is given by
\begin{equation}
\Vol^2(\Prym(\tGa/\Ga))=\sum_{F\subset E(\Ga)} 4^{r(F)-1}w(F),
\label{eq:Prymvolume}
\end{equation}
where the sum is taken over all odd genus one decompositions $F$ of $\Ga$.
\label{thm:Prymvolume}
\end{theorem}

\begin{remark} The right hand side of formula~\eqref{eq:Prymvolume} is defined with respect to a choice of model $G$ for $\Ga$. Let $G'$ be the model obtained from $G$ by subdividing an edge $e\in E(G)$ into edges $e'_1$ and $e'_2$, so that $\ell(e)=\ell(e'_1)+\ell(e'_2)$. If $e\in F$ for some odd genus one decomposition $F$ of $G$, then $(F\backslash \{e\})\cup\{e'_1\}$ and $(F\backslash \{e\})\cup\{e'_2\}$ are odd genus one decompositions of $G'$ of the same rank as $F$, and vice versa. It follows that the right hand side is invariant under edge subdivision, and hence does not depend on the choice of model for $\Ga$. We also note that $\Vol^2(\Prym(\tGa/\Ga))$ is computed with respect to the intrinsic principal polarization on $\Prym(\tGa/\Ga)$, which is half of the restriction of the principal polarization on $\Jac(\tGa)$. 

\end{remark}

We first establish the relationship between the volumes of the three tropical ppavs $\Jac(\tGa)$, $\Jac(\Ga)$, and $\Prym(\tGa/\Ga)$. To compute the last of the three volumes, we define (building on Construction~\ref{con:A}) an explicit basis for the kernel of the pushforward map $\pi_*:H_1(\tGa,/\ZZ)\to H_1(\Ga,\ZZ)$, which we also use later.

Let $G$ be a graph. Introduce the $\ZZ$-valued bilinear pairing
\begin{equation}
\label{eq:thirdpairing}
\langle\cdot,\cdot\rangle:C_1(G,\ZZ)\times C_1(G,\ZZ)\to \ZZ,\quad \left\langle \sum_{e\in E(G)} a_e e,\sum_{e\in E(G)} b_e e\right\rangle=\sum_{e\in E(G)}a_eb_e.
\end{equation}
We note that this pairing does not take edge lengths into account, and is not to be confused with the integration pairing $(\cdot,\cdot)$ on a metric graph.

\begin{mainconstruction} Let $\pi:\tGa\to \Ga$ be a connected free double cover of metric graphs. Choose an oriented model $p:\tG\to G$, and suppose that the cover $p$ is given by Construction~\ref{con:A} with respect to a spanning tree $T\subset G$ and a nonempty subset $S\subset E(G)\backslash E(T)=\{e_0,\ldots,e_{g-1}\}$ containing $e_0$. In this construction, we define an explicit basis of the kernel of the pushforward map $p_*:H_1(\tG,\ZZ)\to H_1(G,\ZZ)$, as well as bases for $H_1(\tG,\ZZ)$ and $H_1(G,\ZZ)$. We use these bases to compute Gramian determinants, hence we view them to be unordered sets. 

We first construct a basis for $H_1(G,\ZZ)$. Let $\ga_i\in H_1(G,\ZZ)$ for  $i=0,\ldots,g-1$ 
denote the unique cycle of $T\cup \{e_i\}$ such that $\langle \ga_i,e_i\rangle=1$. It is a standard fact that
$$
\calB=\{\ga_0,\ldots,\ga_{g-1}\}
$$
is a basis of $H_1(G,\ZZ)$, and furthermore any $\ga\in H_1(G,\ZZ)$ can be explicitly decomposed in terms of $\calB$ as follows:
$$
\ga=\langle\ga,e_1\rangle \ga_1+\cdots+\langle \ga,e_g\rangle \ga_g.
$$
Similarly, let $\tga_0\in H_1(\tG,\ZZ)$ and $\tga_i^{\pm}\in H_1(\tG,\ZZ)$ for $i=1,\ldots,g-1$ denote the unique cycle of respectively $\tT\cup \{\te^-_0\}$ and $\tT\cup \{\te_i^{\pm}\}$ such that respectively $\langle \tga_0,\te^-_0\rangle=1$ and $\langle \tga^{\pm}_i,\te^{\pm}_i\rangle=1$ for $i=1,\ldots,g-1$. Then
$$
\tcalB=\{\tga_0,\tga_1^{\pm},\ldots,\tga_{g-1}^{\pm}\}
$$
is a basis of $H_1(\tG,\ZZ)$, and we similarly have 
$$
\tga=\langle \tga,\te^-_0\rangle\tga_0+\langle\tga,\te^+_1\rangle \tga^+_1+\cdots+\langle \tga,\te^+_{g-1}\rangle \tga^+_{g-1}+\langle\tga,\te^-_1\rangle \tga^-_1+\cdots+\langle \tga,\te^-_{g-1}\rangle \tga^-_{g-1}
$$
for any $\tga\in H_1(\tG,\ZZ)$.

We now compute the action of the pushforward map $p_*:H_1(\tG,\ZZ)\to H_1(G,\ZZ)$ and the involution map $\iota_*:H_1(\tG,\ZZ)\to H_1(\tG,\ZZ)$ on the basis $\tcalB$. The cycle $\tga_0$ starts at the vertex $s(\te_0^-)=\widetilde{s(e_0)}^-$ on the lower sheet $\tT^-$, then proceeds via $+\te_0^-$ to the vertex $t(\te_0^-)=\widetilde{t(e_0)}^+$ on the upper sheet $\tT^+$, then to $\widetilde{s(e_0)}^+$ via a unique path in $\tT^+$, then back to $t(\te_0^+)=\widetilde{t(e_0)}^-$ on $T^-$ via $+\te_0^+$, and then back to $\widetilde{s(e_0)}^-$ via a unique path in $\tT^-$. In other words,
$$
\tga_0=\te_0^++\te_0^-+\mbox{edges of }\tT^{\pm},\quad \iota_*(\tga_0)=\te_0^++\te_0^-+\mbox{edges of }\tT^{\pm},\quad
p_*(\tga_0)=2e_0+\mbox{edges of }T,
$$
therefore computing the intersection numbers with $\tcalB$ and $\calB$ we see that 
$$
\iota_*(\tga_0)=\tga_0,\quad p_*(\tga_0)=2\ga_0.
$$

Now consider the cycle $\tga_i^+$ for $e_i\in S\backslash\{e_0\}$. We introduce the index
$$
\si_i=\left\{\begin{array}{cc} +1,& s(\te_i^+)=s(e_i)^+,\\  -1, & s(\te_i^+)=s(e_i)^-.\end{array}\right.
$$
If $\si_i=1$, then the cycle $\tga_i^+$ starts at $s(\te_i^+)=\widetilde{s(e_i)}^+$ on $\tT^+$, then moves to $t(\te_i^+)=\widetilde{t(e_i)}^-$ on $\tT^-$ via $\te_i^+$, and then back to $\widetilde{s(e_i)}^+$ on $\tT^+$ via a unique path in $\tT$. This path crosses from $\tT^-$ to $\tT^+$, and hence must contain the edge $-\te_0^+$. If $\si_i=-1$, then $\tga_i^+$ crosses from $\tT^+$ to $\tT^-$, and hence contains $\te_0^+$. 
Similarly, we calculate that the cycle $\tga_i^-$ contains the edge $\si_i\te_0^+$. In other words, for $e_i\in S\backslash\{e_0\}$ we have
$$
\tga_i^{\pm}=\te_i^{\pm}\mp \si_i\te_0^++\mbox{edges of }\tT^{\pm},\quad
\iota_*(\tga_i^{\pm})=\te_i^{\mp}\mp \si_i\te_0^-+\mbox{edges of }\tT^{\pm},\quad
p_*(\tga^{\pm}_i)=e_i\mp \si_ie_0+\mbox{edges of }T,
$$
and hence computing the intersection numbers we see that
$$
\iota_*(\tga_i^{\pm})=\tga_i^{\mp}\mp \si_i\tga_0,\quad
p_*(\tga^{\pm}_i)=\ga_i\mp \si_i\ga_0,\quad e_i\in S\backslash\{e_0\}.
$$
Finally, for $e_i\notin S$ the cycle $\ga_i^{\pm}$ is contained in $\tT^{\pm}\cup \{\te_i^{\pm}\}$ and hence does not contain the edge $\te_0^+$. It follows that 
$$
\iota_*(\tga_i^{\pm})=\te_i^{\mp}+\mbox{edges of }\tT^{\pm},\quad p_*(\tga_i^{\pm})=e_i+\mbox{edges of }T,\quad e_i\notin S,
$$
and therefore
$$
\iota_*(\tga_i^{\pm})=\tga_i^{\mp},\quad p_*(\tga_i^{\pm})=\ga_i,\quad e_i\notin S.
$$
It is now clear that
\begin{equation}
\tcalB'_2=\{\tga_i^+-\iota_*(\tga_i^+)\}_{i=1}^{g-1}=
\{\tga_i^+-\tga_i^-+\si_i\tga_0\}_{e_i\in S\backslash\{e_0\}}\cup \{\tga_i^+-\tga_i^-\}_{e_i\notin S}
\label{eq:Prymbasis}
\end{equation}
is a basis for $\Ker p_*$.

In Example~\ref{example:big}, we explicitly construct this basis for a double cover $\pi:\tGa\to \Ga$ with $g=3$.
\label{con:B}

\end{mainconstruction}

We now establish the relationship between the volumes of our three tropical ppavs.

\begin{proposition} Let $\pi:\tGa\to \Ga$ be a free double cover of metric graphs. Then the volumes of $\Jac(\tGa)$, $\Jac(\Ga)$, and $\Prym(\tGa/\Ga)$ are related as
$$
\Vol^2(\Prym(\tGa/\Ga))=\frac{\Vol^2(\Jac(\tGa))}{2\Vol^2(\Jac(\Ga))},
$$
where the volume of each tropical ppav is calculated using its intrinsic principal polarization.
\label{prop:3volumes}
\end{proposition}

\begin{proof} We first introduce the following alternative basis $\calB'$ for $H_1(G,\ZZ)$:
\begin{equation}\label{eq:basisTarget}
\calB'=\{\ga_0\}\cup\{\ga_i-\si_i\ga_0\}_{e_i\in S\backslash \{e_0\} }\cup \{\ga_i\}_{e_i\notin S}. 
\end{equation}
We now compute the pullback $\tcalB'_1$ of $\calB'$ to $H_1(\tG,\ZZ)$ via the map
$$
p^*:H_1(G,\ZZ)\to H_1(\tG,\ZZ),\quad \sum_{e\in E(G)} a_e e\mapsto \sum_{e\in E(G)} a_e (\te^++\te^-).
$$
Since $\ga_i$ consists of $+e_i$ and edges of $T$, we have
$$
p^*(\ga_i)=\te^+_i+\te^-_i+\mbox{edges of }T^{\pm}
$$
for $i=0,\ldots,g-1$. Computing intersection numbers as before, we see that
$$
p^*(\ga_0)=\tga_0,\quad p^*(\ga_i)=\tga^+_i+\tga^-_i,\quad i=1,\ldots,g-1. 
$$
Hence
\[
\tcalB'_1=p^*(\calB')=\{\tga_0\}\cup \{\tga_i^++\tga_i^--\si_i\tga_0\}_{e_i\in S\backslash\{e_0\}}\cup \{\tga_i^++\tga_i^-\}_{e_i\notin S}.
\]
Let $(\cdot,\cdot)_{\tG}$ and $(\cdot,\cdot)_G$ denote the intersection pairings~\eqref{eq:intersectionform} on $H_1(\tG,\ZZ)$ and $H_1(G,\ZZ)$, respectively, and let $(\cdot,\cdot)_P=\frac{1}{2}(\cdot,\cdot)_{\tG}$ denote the intersection pairing on $\Ker p_*$ corresponding to the principal polarization on $\Prym(\tGa/\Ga)$. We add the corresponding subscripts to each Gramian determinant, in order to keep track of the inner product that is used to compute it. Thus the volumes of $\Jac(G)$ and $\Prym(\tGa/\Ga)$ are given by
$$
\Vol^2(\Jac(\Ga))=\Gram_G(\calB'),\quad \Vol^2(\Prym(\tGa/\Ga))=\Gram_P(\tcalB'_2).
$$

We now consider the set $\tcalB'=\tcalB'_1\cup \tcalB'_2$, which is a basis for the vector space $H_1(\tGa,\QQ)$. By appropriately ordering the basis elements, the change-of-basis matrix from $\tcalB$ to $\tcalB'$ becomes block-triangular with a $1$ in the top left corner and a block 
$\begin{pmatrix}
1 & 1\\
-1 & 1
\end{pmatrix}$
for each edge $e_i, i=1,2\ldots,g-1$. Its determinant is therefore $\pm 2^{g-1}$, and it follows that
\[
\Vol^2(\Jac(\tGa))=\Gram_{\tG}(\tcalB)=2^{2-2g}\Gram_{\tG}(\tcalB').
\]
We now compute $\Gram(\tcalB')$ using its block structure. First, we note that $\iota_*(\tga'_1)=\tga'_1$ for all $\tga'_1\in \tcalB'_1$ and $\iota_*(\tga'_2)=-\tga'_2$ for all $\tga'_2\in \tcalB'_2$. Since $\iota_*$ preserves the pairing $(\cdot,\cdot)_{\tG}$, it follows that $(\tga'_1,\tga'_2)_{\tG}=0$ for all $\tga'_1\in \tcalB'_1$ and all $\tga'_2\in \tcalB'_2$, therefore
$$
\Gram_{\tG}(\tcalB')=\Gram_{\tG}(\tcalB'_1)\Gram_{\tG}(\tcalB'_2).
$$
Since $(\cdot,\cdot)_P=\frac{1}{2}(\cdot,\cdot)_{\tG}$, it is clear that
$$
\Gram_{\tG}(\tcalB'_2)=2^{g-1}\Gram_P(\tcalB'_2).
$$
Finally, for any $\ga_1,\ga_2\in H_1(G,\ZZ)$ we have $(p^*(\ga_1),p^*(\ga_2))_{\tG}=2(\ga_1,\ga_2)_G$, and therefore
$$
\Gram_{\tG}(\tcalB'_1)=2^g\Gram_G(\calB'),
$$
because $\tcalB'_1$ is the pullback of $\calB'$. Putting all this together, we have
$$
\frac{\Vol^2(\Jac(\tGa))}{2\Vol^2(\Jac(\Ga))}=\frac{2^{2-2g}\Gram_{\tG}(\tcalB')}{2\Gram_G(\calB')}=
\frac{2^{1-2g}\Gram_{\tG}(\tcalB'_1)\Gram_{\tG}(\tcalB'_2)}{\Gram_G(\calB')}=\Gram_P(\tcalB'_2),
$$
which is equal to $\Vol^2(\Prym(\tGa/\Ga))$, as required.

\end{proof}

The proof of Theorem~\ref{thm:Prymvolume} now follows from Theorem~\ref{thm:discretePrym} and Equation~\eqref{eq:ABKSformula} by an elementary scaling argument.

\begin{proof}[Proof of Theorem~\ref{thm:Prymvolume}] The right hand side of~\eqref{eq:Prymvolume} is a homogeneous degree $g-1$ polynomial in the edge lengths of $\Ga$, and so is the left hand side, being the determinant of a $(g-1)\times (g-1)$ Gramian matrix. Hence, it is sufficient to prove Equation~\eqref{eq:Prymvolume} in the case when $\Ga$, and hence $\tGa$, have integer edge lengths. Choose a model $p:\tG\to G$ for $\pi$ such that each edge of $\tG$ and $G$ has length one. In this case $\Vol(F)=1$ for any set of edges, hence by Kirchhoff's theorem and~\eqref{eq:ABKSformula} we have
$$
\Vol^2(\Jac(\tGa))=|\Jac(\tG)|,\quad \Vol^2(\Jac(\Ga))=|\Jac(G)|.
$$
It follows by Proposition~\ref{prop:3volumes} that
$$
\Vol^2(\Prym(\tGa/\Ga))=\frac{\Vol^2(\Jac(\tGa))}{2\Vol^2(\Jac(\Ga))}=\frac{|\Jac(\tG)|}{2|\Jac(G)|}=|\Prym(\tG/G)|.
$$
But $|\Prym(\tG/G)|$ can be computed using~\eqref{eq:orderofPrym}, which agrees with the right hand side of~\eqref{eq:Prymvolume} when all edge lengths are equal to one. This completes the proof.

\end{proof}

\begin{example} Let $\Ga$ be the genus two dumbbell graph, with loops $e_1$ and $e_2$ of lengths $x_1$ and $x_2$, connected by a bridge $e_3$ of length $x_3$. The unique spanning tree of $\Ga$ consists of the edge $e_3$. The graph $\Ga$ has two topologically distinct connected free double covers $\pi_1:\tGa_1\to \Ga$ and $\pi_2:\tGa_2\to \Ga$, corresponding to flipping the edges $S_1=\{e_1,e_2\}$ and $S_2=\{e_1\}$ (see Figure~\ref{fig:twocovers1}).

\begin{figure}[ht]
\begin{tikzpicture}

\node at (-1.5,0) {$\Ga$};
\draw[thick,blue] (-0.5,0) circle(.5);
\draw[fill](0,0) circle(.08);
\draw[thick] (0,0) -- (1,0);
\draw[thick,blue] (1.5,0) circle(.5);
\draw[fill](1,0) circle(.08);
\node at (-0.7,0) {$e_1$};
\node at (0.5,0.2) {$e_3$};
\node at (2.3,0) {$e_2$};

\begin{scope}[shift={(0,1.5)}]

\draw[thick] (0,0) -- (1,0);
\draw[thick] (0,2) -- (1,2);
\draw[thick,blue] (0,0) -- (0,2);
\draw[thick,blue] (0,0) .. controls (-1.3,0) and (-1.3,2) .. (0,2);
\draw[thick,blue] (1,0) -- (1,2);
\draw[thick,blue] (1,0) .. controls (2.3,0) and (2.3,2) .. (1,2);

\node at (-1.5,1){$\tGa_1$};
\node at (-0.7,1) {$\te^+_1$};
\node at (0.3,1) {$\te^-_1$};
\node at (0.5,0.3) {$\te^-_3$};
\node at (0.5,2.3) {$\te^+_3$};
\node at (1.3,1) {$\te^+_2$};
\node at (2.3,1) {$\te^-_2$};

\draw[fill](0,0) circle(.08);
\draw[fill](1,0) circle(.08);
\draw[fill](0,2) circle(.08);
\draw[fill](1,2) circle(.08);

\end{scope}

\begin{scope}[shift={(5,0)}]
\node at (-1.5,0) {$\Ga$};
\draw[thick,blue] (-0.5,0) circle(.5);
\draw[fill](0,0) circle(.08);
\draw[thick] (0,0) -- (1,0);
\draw[thick] (1.5,0) circle(.5);
\draw[fill](1,0) circle(.08);
\node at (-0.7,0) {$e_1$};
\node at (0.5,0.2) {$e_3$};
\node at (2.3,0) {$e_2$};

\begin{scope}[shift={(0,1.5)}]

\draw[thick] (0,0) -- (1,0);
\draw[thick] (0,2) -- (1,2);
\draw[thick,blue] (0,0) -- (0,2);
\draw[thick,blue] (0,0) .. controls (-1.3,0) and (-1.3,2) .. (0,2);
\draw[thick] (1.5,0) circle(.5);
\draw[thick] (1.5,2) circle(.5);

\node at (-1.5,1){$\tGa_2$};
\node at (-0.7,1) {$\te^+_1$};
\node at (0.3,1) {$\te^-_1$};
\node at (0.5,0.3) {$\te^-_3$};
\node at (0.5,2.3) {$\te^+_3$};
\node at (2.3,0) {$\te^-_2$};
\node at (2.3,2) {$\te^+_2$};

\draw[fill](0,0) circle(.08);
\draw[fill](1,0) circle(.08);
\draw[fill](0,2) circle(.08);
\draw[fill](1,2) circle(.08);

\end{scope}

\end{scope}

\end{tikzpicture}

\caption{Two free double covers of the dumbbell graph. Flipped edges are blue.}
\label{fig:twocovers1}

\end{figure}
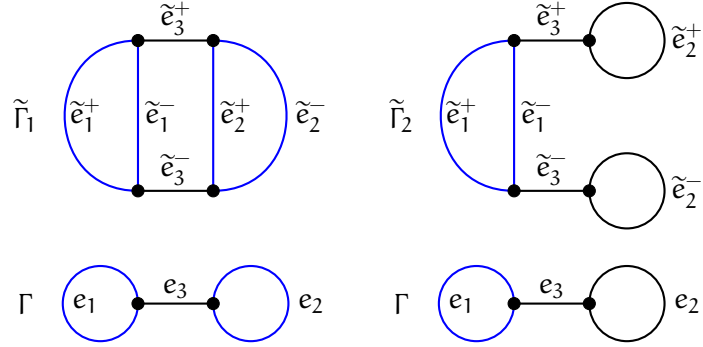

For the cover $\pi_1$, the odd genus one decompositions are $\{e_1\}$ and $\{e_2\}$ of rank one, and $\{e_3\}$ of rank two. For $\pi_2$, the only odd genus one decomposition is $\{e_2\}$ of rank one. Hence Theorem~\ref{thm:Prymvolume} states that
$$
\Vol^2(\Prym(\tGa_1/\Ga))=x_1+x_2+4x_3,\quad \Vol^2(\Prym(\tGa_2/\Ga))=x_2.
$$
Note that in each case the Prym variety is a circle, and the square of its volume is its circumference (see Remark~\ref{rem:wrongunits}). 

\label{ex:twocovers1}
\end{example}

\section{The local structure of the Abel--Prym map}\label{sec:AbelPrym}

In the remaining two chapters, we provide a geometrization of the volume formula~\eqref{eq:Prymvolume} for the Prym variety of a free double cover of graphs, in the spirit of the analogous formula~\eqref{eq:ABKSformula} for the volume of the Jacobian variety of a metric graph derived in~\cite{ABKS_Canonical}.

Let $\pi:\tGa\to \Ga$ be a free double cover of metric graphs and let $\iota:\tGa\to \tGa$ be the associated involution. For any integer $d$, we denote $\Prym^{[d]}(\tGa/\Ga)$ the connected component of the kernel of the pushforward map $\Nm:\Jac(\tGa)\to\Jac(\Ga)$ having the same parity as $d$, so that $\Prym^{[d]}(\tGa/\Ga)=\Prym(\tGa/\Ga)$ if $d$ is even, and $\Prym^{[d]}(\tGa/\Ga)$ is the odd connected component if $d$ is odd. In this section, we study the \emph{Abel--Prym map}
\begin{equation}
\Psi^d:\Sym^d(\tGa)\to \Prym^{[d]}(\tGa/\Ga),\quad \Psi^d(\tD)=\tD-\iota(\tD),
\label{eq:AP}
\end{equation}
for $d\leq g-1$. The space $\Sym^d(\tGa)$ has a natural cellular structure, with cells enumerated by the edges and vertices of $\tGa$ supporting the divisor. The restriction of $\Psi^d$ to each cell is an affine linear map, and we determine the cells on which $\Psi^d$ has maximal rank. 

Choose an oriented model $p:\tG\to G$ for $\pi$ such that $\tG$ and $G$ have no loops. Let $0\leq k\leq d$, let $\tF=\{\tf_1,\ldots,\tf_k\}$ be a multiset of $k$ edges of $\tG$, and let $\tZ$ be an effective divisor of degree $d-k$ supported on $V(\tG)$. Denote by $C^k(\tF,\tZ)$ the $k$-dimensional set of effective divisors on $\tGa$ of the form $\tD=\tP_1+\cdots+\tP_k+\tZ$, where each $\tP_i$ lies on $\tf_i$. Any effective degree $d$ divisor on $\tGa$ can be split up in such a way (uniquely if each point lies in the interior of an edge), hence we have a cellular decomposition
\begin{equation}
\Sym^d(\tGa)=\bigcup_{k=0}^d \bigcup_{\tF,\tZ} C^k(\tF,\tZ),
\label{eq:Sympolyhedral}
\end{equation}
where the union is taken over all $\tF\in \Sym^k(E(\tG))$ and $\tZ\in \Sym^{d-k}(V(\tG))$. 

The principal result of this section describes the cells $C^k(\tF,\tZ)$ that are not contracted by the Abel--Prym map $\Psi^d$. It is clear that the divisor $\tZ$ plays no role in this question, hence we assume that $k=d$, $\tZ=0$, and only consider the top-dimensional cells, which we denote
$$
C(\tF)=C^d(\tF,0)=\{\tP_1+\cdots+\tP_d:\tP_i\in\tf_i\}\subset \Sym^d(\tGa),\quad \tF=\{\tf_1,\ldots,\tf_d\}\in \Sym^{d}(E(\tG)).
$$

\begin{theorem} Let $\pi:\tGa\to \Ga$ be a free double cover of metric graphs with oriented loopless model $p:\tG\to G$, and let $\Psi^d:\Sym^d(\tGa)\to \Prym^{[d]}(\tGa/\Ga)$ be the degree $d$ Abel--Prym map, where $1\leq d\leq g-1$. Let $\tF=\{\tf_1,\ldots,\tf_d\}\subset E(\tG)$ be a multiset of edges of $\tG$, let $C(\tF)\subset \Sym^d(\tGa)$ be the corresponding top-dimensional cell, and denote $F=\{f_1,\ldots,f_d\}$, where $f_i=p(\tf_i)$. 

\begin{enumerate}
    \item If the edges in $F$ are not distinct (in particular, if the edges in $\tF$ are not distinct), then $\Psi^d$ contracts $C(\tF)$.

    \item If the edges in $F$ are distinct, then $\Psi^d$  does not contract $C(\tF)$ if and only if the preimage under $p$ of each connected component of $G\backslash F$ is connected. 
\end{enumerate}

\label{thm:APlocal}
\end{theorem}

The proof of the first part of the theorem is quite elementary: for any $\tD\in C(\tF)$ we construct a nearby divisor $\tD'$ such that $\Psi^d(\tD)=\tD-\iota(\tD)$ is linearly equivalent to $\Psi^d(\tD')=\tD'-\iota(\tD')$ via an explicit rational function. To prove the second part, we calculate the matrix of $\Psi^d$ with respect to a convenient basis, and compute its rank. This part, and the necessary constructions, will occupy the greater part of this section.

\begin{proof}[Proof of Theorem~\ref{thm:APlocal}, part (1)] Let $\tF=\{\tf_1,\ldots,\tf_d\}$ be a multiset such that not all $f_i=p(\tf_i)$ are distinct. Without loss of generality we assume that $f_1=f_2$, which means that either $\tf_1=\tf_2$ or $\tf_1=\iota(\tf_2)$. Let $\tD=\tP_1+\cdots+\tP_d$ be a point of $C(\tF)$, where each $\tP_i$ lies in the interior of $\tf_i$.

If $\tf_1=\tf_2$, we can assume that either $\tP_1=\tP_2$, or that the direction from $\tP_1$ to $\tP_2$ is positive with respect to the orientation. Denote $\tD'=\tP'_1+\tP'_2+\tP_3+\cdots+\tP_d$, where $\tP'_1$ and $\tP'_2$ are obtained by moving $\tP_1$ and $\tP_2$ a small distance of $\varepsilon>0$ in respectively the negative and the positive directions along $\tf_1=\tf_2$. Then the divisor
$$
\Psi^d(\tD)-\Psi^d(\tD')=\tD-\iota(\tD)-\tD'+\iota(\tD')=\tP_1+\tP_2-\tP'_1-\tP'_2-\iota(\tP_1)-\iota(\tP_2)+\iota(\tP'_1)+\iota(\tP'_2)
$$
is the principal divisor of a piecewise linear function on $\tGa$. Indeed, consider the function $M:\tGa\to \RR$ having the following slopes on the edges of $\tGa$:
\begin{itemize}
    \item On $\tf_1=\tf_2$, $M$ has slope zero to the left of $\tP_1'$, slope $+1$ on $[\tP'_1,\tP_1]$, slope zero on $[\tP_1,\tP_2]$, slope $-1$ on $[\tP_2,\tP'_2]$, and slope zero to the right of $\tP'_2$.
    
    \item On $\iota(\tf_1)=\iota(\tf_2)$, $M$ has slope zero to the left of $\iota(\tP_1)$, slope $-1$ on $[\iota(\tP'_1),\iota(\tP_1)]$, slope zero on $[\iota(\tP_1),\iota(\tP_2)]$, slope $+1$ on $[\iota(\tP_2),\iota(\tP'_2)]$, and slope zero to the right of $\iota(\tP'_2)$.
    \item The function $M$ has zero slope on all other edges of $\tGa$.
\end{itemize}
The net changes of $M$ along $\tf_1=\tf_2$ and $\iota(\tf_1)=\iota(\tf_2)$ cancel out, hence the function $M$ is continuous, and it is clear that $\Psi^d(\tD)-\Psi^d(\tD')$ is the divisor of $M$. Therefore, $\Psi^d$ is not locally injective at $\tD$.

The case $\tf_1=\iota(\tf_2)$ is similar. We consider $\tD'=\tP'_1+\tP'_2+\tP_3+\cdots+\tP_d$, where $\tP'_1$ and $\tP'_2$ are obtained by moving $\tP_1$ and $\tP_2$ a small distance of $\varepsilon>0$ in the same direction along respectively $\tf_1$ and $\tf_2=\iota(\tf_1)$. It is easy to check that $\Psi^d(\tD)-\Psi^d(\tD')$ is a principal divisor, hence $\Psi^d$ is not locally injective at $\tD$.

\end{proof}

To prove part (2) of Theorem~\ref{thm:APlocal}, we give an explicit description of the Abel--Prym map $\Psi^d$ on a cell $C(\tF)$. We first consider the case $d=1$. Fix a base point $q\in \tGa$, and for each point $p\in \tGa$ fix a path $\ga(q,p)$ from $q$ to $p$. The Abel--Prym map $\Psi^1$ is a difference of Abel--Jacobi maps~\eqref{eq:AJ}:
$$
\Psi^1:\tGa\to \Prym^{[1]}(\tGa/\Ga)\subset \Jac(\tGa),\quad \Psi^1(p)(\om)=\int_{\ga(q,p)}\om-\int_{\ga(q,\iota(p))}\om.
$$
It is more convenient to work with the even component $\Prym(\tGa/\Ga)$. The odd component $\Prym^{[1]}(\tGa/\Ga)$ is a torsor over $\Prym(\tGa/\Ga)$, and we can pass to the even component by translating by any element of the odd component, for example the element $\Psi^1(\iota(q))=\int_{\ga(q,\iota(q))}$. We can further assume that the path $\ga(q,\iota(p))$ in the formula above consists of $\ga(q,\iota(q))$ followed by the path $\iota_*(\ga(q,p))$. Putting this together, we obtain the translated Abel--Prym map, which we also denote $\Psi^1$ by abuse of notation:
\begin{equation}
\Psi^1:\tGa\to \Prym(\tGa/\Ga),\quad \Psi^1(p)(\om)=\int_{\ga(q,p)}\om-\int_{\iota_*(\ga(q,p))}\om.
\label{eq:APdeg1}
\end{equation}
Choose a basis $\tga_1,\ldots,\tga_{g-1}$ for $\Ker \pi_*:H_1(\tGa,\ZZ)\to H_1(\Ga,\ZZ)$. As explained in Subsection~\ref{subsec:Prymvarieties}, the functionals $\om^*_j\in \Om^*(\tGa)$ dual to $\om_j=\psi(\tga_j)$ define a coordinate system on $\Prym(\tGa/\Ga)$, so we can write 
$$
\Psi^1(p)=\int_{\ga(q,p)}-\int_{\iota_*(\ga(q,p))}=\sum_{j=1}^{g-1}a_j(p)\om_j^*,
$$
where we find the coefficients $a_j(p)$ by pairing with $\om_j$:
$$
a_j(p)=\int_{\ga(q,p)}\om_j-\int_{\iota_*(\ga(q,p))}\om_j.
$$
We now assume that $p$ lies on the interior of an edge $\tf\in E(\tGa)$, which we identify using the orientation with the segment $(0,\ell(\tf))$. Pick $x_1,x_2\in (0,\ell(\tf))$ such that $x_1<x_2$, then we see that
$$
a_j(x_2)-a_j(x_1)=\int_{\ga(q,x_2)}\om_j-\int_{\iota_*(\ga(q,x_2))}\om_j-\int_{\ga(q,x_1)}\om_j+\int_{\iota_*(\ga(q,x_1))}\om_j=\int_{\ga(x_1,x_2)}\om_j-\int_{\iota_*(\ga(x_1,x_2))}\om_j.
$$
We can assume that $\ga(x_1,x_2)$ is the segment $[x_1,x_2]\subset (0,\ell(\tf))$. The integral of $\om_j$ over $\ga(x_1,x_2)$ is equal to the length $x_2-x_1$ multiplied by the coefficient with which $\tf$ occurs in $\om_j$, which is $\frac{1}{2}\langle\tga_j,\tf\rangle$ (the $\frac{1}{2}$ coefficient comes from using the principal polarization of the Prym). Similarly, the integral of $\om_j$ over $\iota_*(\ga(x_1,x_2))$ is equal to $\frac{1}{2}(x_2-x_1)\langle\tga_j,\iota(\tf)\rangle$, and therefore
$$
a_j(x_2)-a_j(x_1)=\frac{1}{2}(x_2-x_1) \langle \tga_j,\tf-\iota(\tf)\rangle,
$$
where $\langle\cdot,\cdot\rangle$ is the edge pairing~\eqref{eq:thirdpairing}. It follows that, with respect to the coordinate vectors $\om^*_j$ defined by the basis $\tga_j$, the restriction of the Abel--Prym map~\eqref{eq:APdeg1} to an edge $\tf=[0,\ell(\tf)]$ is an affine $\ZZ$-linear map of the form
$$
\Psi^1(x)=\frac{1}{2}\sum_{j=1}^{g-1}\langle \tga_j,\tf-\iota(\tf)\rangle \om^*_jx+C,
$$
where $C$ is a constant vector.

This formula readily generalizes to any degree. Let $\tF=\{\tf_1,\ldots,\tf_d\}$ be a set of distinct edges of $\tG$. We identify the cell $C(\tF)$ with the parallelotope $[0,\ell(\tf_1)]\times\cdots\times [0,\ell(\tf_d)]$, where the point $(x_1,\ldots,x_d)$ corresponds to the divisor $\tD=\tP_1+\cdots+\tP_d$, where $\tP_i$ lies on $\tf_i$ at a distance of $x_i$ from the starting vertex. Translating by any odd Prym divisor and moving to the even component if $d$ is odd, we see that the restriction of the Abel--Prym map~\eqref{eq:AP} to the cell $C(\tF)$ is affine $\ZZ$-linear:
$$
\Psi^d(x_1,\ldots,x_d)=\sum_{i=1}^d\sum_{j=1}^{g-1}(\Psi^d)_{ji}\om^*_jx_i+C\in \Prym(\tGa/\Ga),
$$
where $(\Psi^d)_{ji}$ is the $(g-1)\times d$ matrix
\begin{equation}
(\Psi^d)_{ji}=\frac{1}{2}\left\langle \tga_j, \tf_i-\iota(\tf_i)\right\rangle,
\label{eq:dPsi}
\end{equation}
and $C$ is some constant vector.

To prove part (2) of Theorem~\ref{thm:APlocal}, we compute the rank of $(\Psi^d)_{ji}$ with respect to a carefully chosen basis $\tga_j$ of $\Ker \pi_*$. Specifically, the "if" and "only if" statements will require slightly different choices of basis.

\begin{proof}[Proof of Theorem~\ref{thm:APlocal}, part (2)] We consider the decomposition of $G\backslash F$ into connected components, which we enumerate as follows:
$$
G\backslash F=G_0\cup \cdots\cup G_{r-1}.
$$
Before proceeding, we derive a relationship between the genera $g_k=|E(G_k)|-|V(G_k)|+1$ of the components $G_k$: 
\begin{equation}
g_0+\cdots+g_{r-1}=\sum_{k=0}^{r-1}|E(G_k)|-\sum_{k=0}^{r-1}|V(G_k)|+r=|E(G)|-d-|V(G)|+r=g+r-d-1.
\label{eq:genuscount}
\end{equation}

We now consider the two possibilities.

\medskip \noindent {\bf The preimage of one of the connected components is disconnected.} Equivalently, we assume that the restriction of the cover $p$ to one of the connected components is isomorphic to the trivial free double cover. Let $G_k$ be a connected component of genus $g_k$. If $p^{-1}(G_k)$ is connected, then Construction~\ref{con:B}, applied to the cover $p|_{p^{-1}(G_k)}:p^{-1}(G_k)\to G_k$, produces $g_k-1$ linearly independent cycles $\ga_{kl}\in H_1(\tG,\ZZ)$ that are supported on $p^{-1}(G_k)$ and that lie in $\Ker p_*$. However, if the restriction of $p$ to, say, $G_k$ is trivial, then we can find $g_k$ such cycles, by applying $(\Id-\iota)_*$ to the lifts of a linearly independent collection of cycles on $G_k$. In this case, it follows from~\eqref{eq:genuscount} that there are at least
$$
(g_0-1)+(g_1-1)+\ldots+g_k+\ldots+(g_{r-1}-1)=g-d
$$
linearly independent cycles $\ga_{kl}\in H_1(\tG,\ZZ)$, lying in the kernel of $p_*$ and supported on
$$
p^{-1}(G_0)\cup \ldots\cup p^{-1}(G_{r-1})=
\tG\backslash (\tF\cup \iota(\tF)).
$$
Any such cycle $\ga_{kl}$ pairs trivially with each $\tf_i$ and $\iota(\tf_i)$. Therefore, by completing these cycles to a basis of $\Ker p_*$ (passing to $\QQ$-coefficients if necessary), we see that the matrix~\eqref{eq:dPsi} of $\Psi^d$ with respect to this basis has at least $g-d$ rows of zeroes, hence has rank less than $d$ and contracts $C(\tF)$.

\medskip \noindent {\bf The preimage of each connected component is connected.} Equivalently, the restriction of $p$ to each connected component is a nontrivial free double cover. This implies that $g_k\geq 1$ for each $k$, since any free double cover of a tree is trivial.

We show that the matrix~\eqref{eq:dPsi} has rank $d$ with respect to an explicit choice of basis $\tga_i$ of $\Ker \pi_*$. The construction of this basis is somewhat involved, and will be used again in the proof of Theorem~\ref{thm:APharmonic}, so we typeset it separately.

\begin{mainconstruction}\label{con:C} Let $\pi:\tGa\to \Ga$ be a connected free double cover of metric graphs of genera $2g-1$ and $g$, respectively, and let $p:\tG\to G$ be an oriented model. Let $\tF=\{\tf_1,\ldots,\tf_d\}\subset E(\tG)$ be a set of $d$ edges so that the edges $f_i=p(\tf_i)$ are distinct, and denote $F=p(\tF)$. Let 
$$
G\backslash F=G_0\cup \cdots\cup G_{r-1}
$$
be the decomposition of $G\backslash F$ into connected components, and further assume that $p^{-1}(G_k)$ is connected for each $k$. In this Construction, we elaborate on Constructions~\ref{con:A} and~\ref{con:B} by carefully choosing a spanning tree $T\subset G$ and a corresponding basis $\tga_1,\ldots,\tga_{g-1}$ for $\Ker \pi_*:H_1(\tG,\ZZ)\to H_1(G,\ZZ)$, with respect to which the matrix of the Abel--Prym map on the cell $C(\tF)$ has a convenient triangular structure. This construction involves a number of choices and relabelings, so we break it down into steps. At the same time, we will contract the portions of the graphs $\tG$ and $G$ that are irrelevant to our intersection calculations.

\emph{Contracting the $G_k$ and choosing the spanning tree $T\subset G$.} Choose a spanning tree $T_k$ for each connected component $G_k$. Denote by $G^c$ the graph obtained from $G$ by contracting each subtree $T_k$ to a separate vertex $v_k$. Specifically, the vertex set of $G^c$ is $V(G^c)=\{v_0,\ldots,v_{r-1}\}$, and the edge set of $G^c$ is the set of edges $G$ that are not in any $T_k$. There is a natural contraction map $(\cdot)^c:G\to G^c$; this map sends each vertex and edge of $T_k$ to $v_k$, and sends each edge of $G$ that is not in any $T_k$ to the corresponding edge of $G^c$. The contracted graph $G^c$ has the same genus $g$ as $G$, hence it has $g+r-1$ edges, namely the edges $F=\{f_1,\ldots,f_d\}$ and $g+r-d-1$ loops corresponding to the uncontracted edges of the $G_k$. Choose a spanning tree $T^c$ for the contracted graph $G^c$; the $r-1$ edges of $T^c$ are a subset of the edges $\{f_1,\ldots,f_d\}$.

Before proceeding, we relabel the edges $\tf_i$ and $f_i=p(\tf_i)$ so that $E(T^c)=\{f_1,\ldots,f_{r-1}\}$. We then choose $v_0$ as the root vertex of $T_c$, and further relabel and reorient the edges $f_1,\ldots,f_{r-1}$ away from $v_0$. Specifically, we require that, along the unique path in $T^c$ starting at $v_0$ and ending at any other vertex, the edges are oriented in the direction of the path and appear in increasing order. Finally, we relabel the vertices $v_1,\ldots,v_{r-1}$ so that $t(f_j)=v_j$ for $j=1,\ldots,r-1$; this implies that $s(f_j)=v_{\alpha(j)}$ for some index $\alpha(j)<j$.

We now form a spanning tree $T$ for $G$ by joining the subtrees $T_k$ with the edges of $T^c$ (viewed as edges of $G$):
$$
T=T_0\cup\cdots \cup T_{r-1}\cup \{f_1,\ldots,f_{r-1}\}.
$$
We denote the complementary edges of $T$ in $G$ by $E(G)\backslash E(T)=\{e_0,e_1,\ldots,e_{g-1}\}$, and we explain below how the $e_j$ are chosen.

\emph{Choosing $e_0$ and contracting the preimages of the $G_k$.} We now describe the cover $p$ using the spanning tree $T$ and Construction~\ref{con:A}. The tree $T$ has two disjoint lifts $\tT^{\pm}$ to $\tG$, and we denote $\tT^{\pm}_k=\tT^{\pm}\cap p^{-1}(T_k)$ the corresponding lifts of the $T_k$. For each $i=1,\ldots,d$, each of the trees $\tT^{\pm}$ contains exactly one of the two edges $\tf_i$ and $\iota(\tf_i)$. The cover $p:f^{-1}(G_0)\to G_0$ is not trivial, so we can pick an edge $e_0\in E(G_0)\backslash E(T_0)$ having a lift $\te_0=\te_0^+$ that connects $\tT_0^+$ and $\tT_0^-$. Then
$$
\tT=\tT^+\cup \tT^-\cup\{\te_0^+\}
$$
is a spanning tree for $\tG$. We note that, by our labeling convention, for $k=1,\ldots,r-1$ the target vertex $t(\tf_k)$ lies on either $\tT^+_k$ or $\tT^-_k$, while $t(\iota(\tf_k))$ lies on the other subtree. 

We now perform a contraction on the graph $\tG$ that is consistent with the contraction $(\cdot)^c:G\to G^c$ defined above. The tree $\tT^+_0\cup\tT^-_0\cup\{\te_0^+\}$ is a spanning tree for the preimage $p^{-1}(G_0)$, while for each $k=1,\ldots,r-1$ the two disjoint trees $\tT^+_k$ and $\tT^-_k$ form a spanning forest for $p^{-1}(G_k)$. Let $\tG^c$ denote the graph obtained from $\tG$ by contracting $\tT^+_0\cup\tT^-_0\cup\{\te_0^+\}$ to a vertex $\tv_0$, and contracting each $\tT^{\pm}_k$ to a separate vertex $\tv^{\pm}_k$. We denote the contraction map by $(\cdot)^c:\tG\to \tG^c$, and for a non-contracted edge $\te\in E(\tG)$ (i.e. for any edge neither in $\tT^+_0\cup\tT^-_0\cup\{\te_0^+\}$ nor in any $\tT^{\pm}_k$) we denote $(\te)^c=\te$ by abuse of notation. The double cover $p:\tG\to G$ descends to a map $p:\tG^c\to G^c$ (which is almost a double cover of graphs, except that $v_0$ and $e_0$ each have a single preimage), and the projections commute with the contractions. The contraction $\tT^c$ of $\tT$ is a spanning tree for $\tG^c$, having vertex and edge sets
$$
V(\tT^c)=V(\tG^c)=\{\tv_0,\tv^{\pm}_1,\ldots,\tv^{\pm}_{r-1}\},\quad E(\tT^c)=\{\tf_1,\ldots,\tf_{r-1},\iota(\tf_1),\ldots,\iota(\tf_{r-1})\}.
$$
The tree $\tT^c$ can be viewed as two copies of $T^c$ joined at the common vertex $\tv_0$.

\emph{Labeling the complementary edges.} Finally, we label the complementary edges
$$
E(G)\backslash E(T)=E(G_0)\backslash E(T_0)\cup\cdots\cup E(G_{r-1})\backslash E(T_{r-1})\cup\{f_r,\ldots,f_d\}=
\{e_0,\ldots,e_{g-1}\}.
$$
Each of the $e_j$ is either an edge of a subgraph $G_k$ that does not lie on the spanning tree $T_k$, or is one of the $f_k$. Since the restriction of the double cover $p:\tG\to G$ to each of the $G_k$ is nontrivial, for each $k=1,\ldots,r-1$ we can choose an edge in $E(G_k)\backslash E(T_k)$ whose preimages cross $\tT^{\pm}_k$, and we label this edge $e_k$. We pick the preimage $\te_k=\te_k^+\in E(\tG)\backslash E(\tT)$ in such a way that the source vertex $s(\te_k)$ lies on the same subtree $\tT^+_k$ or $\tT^-_k$ as the target vertex $t(\tf_k)$. Furthermore, for $k=r,\ldots,d$, we denote $e_k=f_k$ and $\te_k=\te^+_k=\tf_k$. The remaining edges $e_k$ for $k=d+1,\ldots,g-1$ and their preimages $\te^{\pm}_k$ are labeled arbitrarily. We note that in this case, $e_k\in S$ for $k=0,\ldots,r-1$.

To help follow the intersection calculations in the following proofs, we summarize the structure of the graphs $\tG^c$ and $G^c$ and the map $p:\tG^c\to G^c$:

\begin{enumerate}
    \item The vertices of the graphs $\tG^c$ and $G^c$ are
$$
V(\tG^c)=\{\tv_0,\tv^{\pm}_1,\ldots,\tv^{\pm}_{r-1}\},\quad V(G^c)=\{v_0,v_1,\ldots,v_{r-1}\}.
$$
The map $p$ sends $\tv_0$ to $v_0$ and $\tv^{\pm}_k$ to $v_k$.
    
    \item The edges of the graph $G^c$ are as follows:
$$
E(G^c)=\{f_1,\ldots,f_{r-1},e_0,\ldots,e_{r-1},e_r=f_r,\ldots,e_d=f_d,e_{d+1},\ldots,e_{g-1}\}.
$$
The $r-1$ edges $\{f_1,\ldots,f_{r-1}\}$ are the edges of the spanning tree $T_c$; these are oriented in increasing order away from the root vertex $v_0$, so that $t(f_j)=v_j$ for $j=1,\ldots,r-1$. The remaining $g$ edges $\{e_0,\ldots,e_{g-1}\}$ are split into three groups: for $j=0,\ldots,r-1$ the edge $e_j$ is a loop at $v_j$, for $j=r,\ldots,d$ the edge $e_j=f_j$ may or may not be a loop, and for $j=d+1,\ldots,g-1$ each remaining $e_j$ is a loop at one of the $v_k$.

\item The edges of the graph $\tG^c$ are as follows. For $j=1,\ldots,r-1$, each $f_j$ has two preimages labeled $\tf_j$ and $\iota(\tf_j)$; these $2r-2$ edges form the spanning tree $\tT^c$. Furthermore, for $j=1,\ldots,r-1$ the target vertex $t(\tf_j)$ is one of the two vertices $\tv_j^{\pm}$. The edge $e_0$ has a unique preimage $\te^-_0$, while for $j=1,\ldots,g-1$ each $e_j$ has two preimages $\te_j^+=\te_j$ and $\te_j^-=\iota(\te_j)$. For $j=1,\ldots,r-1$ the edges $\te_j$ and $\iota(\te_j)$ (lying over the loop $e_j$) connect $\tv_j^+$ and $\tv_j^-$ in opposite directions, in such a way that $s(\te_j)=t(\tf_j)$. For $j=r,\ldots,d$ we have $\te_j=\tf_j$, and for each $j=d+1,\ldots,g-1$ the two edges $\te_j$ and $\iota(\te_j)$ over $e_j$ are either parallel edges between $\tv^+_k$ and $\tv^-_k$ for some $k$, or form a pair of loops.

\end{enumerate}

We now employ Construction~\ref{con:B} to produce a basis $\tga_1,\ldots,\tga_{g-1}$ of $\Ker \pi_*:H_1(\tG,\ZZ)\to H_1(G,\ZZ)$, with respect to the chosen spanning tree $T\subset G$. Let $\tga_0$ and $\tga_j^{\pm}$ for $j=1,\ldots,g-1$ be the unique cycle of respectively $\tT\cup\{\te^-_0\}$ and $\tT\cup\{\te^{\pm}_j\}$ such that $\langle \tga_0,\te^-_0\rangle=1$ and $\langle \tga^{\pm}_j,\te^{\pm}_j\rangle=1$ for $j=1,\ldots,g-1$. Then the cycles
\begin{equation}
\tga_j=\tga_j^+-\iota_*(\tga_j^+)=\left\{\begin{array}{cc}\tga^+_j-\tga^-_j+\si_j\tga_0, & e_j\in S, \\ \tga^+_j-\tga^-_j, & e_j\notin S,\end{array}\right.,\quad j=1,\ldots,g-1
\label{eq:mainbasis}
\end{equation}
form a basis for $\Ker \pi_*$.

In Example~\ref{example:big}, we deploy this construction for a specific cover $\pi:\tGa\to \Ga$ with $g=3$.

\end{mainconstruction}

We now return to the proof. We need to compute the rank of the matrix~\eqref{eq:dPsi} with respect to the basis~\eqref{eq:mainbasis}:
$$
(\Psi^d)_{ji}=\frac{1}{2}\left\langle \tga_j, \tf_i-\iota(\tf_i)\right\rangle=\frac{1}{2}\left\langle \tga^+_j-\iota_*(\tga^+_j), \tf_i-\iota(\tf_i)\right\rangle=\left\langle \tga^+_j, \tf_i -\iota(\tf_i)\right\rangle.
$$
To compute the intersection numbers $\langle \tga^+_j,\tf_i\rangle$ and $\langle \tga^+_j,\iota(\tf_i)\rangle$, we pass to the contracted graph $\tG^c$. First of all, we note that none of edges $\tf_i$ or $\iota(\tf_i)$ on $\tG$ are contracted. Therefore, for any cycle $\tga$ on $\tG$, its intersection with $\tf_i$ or $\iota(\tf_i)$ can be computed on the contracted graph $\tG^c$:
$$
\langle \tga,\tf_i\rangle=\langle \tga^c,\tf_i\rangle,\quad \langle \tga, \iota(\tf_i)\rangle=\langle \tga^c,\iota(\tf_i)\rangle,\quad i=1,\ldots,d.
$$
Furthermore, we observe that, since $\tga^+_j$ is the unique cycle on $\tT\cup \{\te_j\}$ such that $\langle \tga^+_j,\te_j\rangle=1$, the contracted cycle $(\tga_j)^c$ is the unique cycle on $\tT^c\cup \{\te_j\}$ such that $\langle (\tga^+_j)^c,\te_j\rangle=1$.

We first look at the cycles $(\tga^+_j)^c$ for $j=1,\ldots,r-1$. The edge $e_j\in E(G^c)$ is a loop at $v_j$. Its lift $\te_j\in E(\tG^c)$ starts at $t(\tf_j)$, which is one of the two vertices $\tv^{\pm}_j$ (say $\tv^+_j$), and ends at the other vertex $\tv^-_j$ (by our labeling convention $\tv^-_j=t(\iota(\tf_j))$). The contracted cycle $(\tga^+_j)^c$ is the unique cycle of the graph $\tT^c\cup\{\te_j\}$ containing $+\te_j$: it starts at $\tv^+_j$, proceeds to $\tv^-_j$ via $\te_j$, then to $\tv_0$ via the unique path on $\tT^c$ from $\tv^+_j$ (the first edge of this path being $-\iota(\tf_j)$), and then from $\tv_0$ to $\tv^-_j$ via a unique path (the last edge of this path being $+\tf_j$). By the ordering convention that we chose for $T^c$ and hence $\tT^c$, the only edges that can occur on this (other that the generating edge $\te_j$, which does not lie on $\tT^c$) are $\tf_i$ and $\iota(\tf_i)$ with $i\leq j$. Furthermore, $\tf_j$ and $\iota(\tf_j)$ occur, as we have seen, with coefficients $+1$ and $-1$, respectively. It follows that for $j=1,\ldots,r-1$ we have
$$
(\Psi^d)_{ji}=\left\langle \tga^+_j, \tf_i -\iota(\tf_i)\right\rangle=\langle (\tga^+_j)^c,\tf_i-\iota(\tf_i)\rangle=\left\{\begin{array}{cc} 0\mbox{ or }\pm 2,& i<j, \\ 2, & i=j, \\ 0, & i>j. \end{array}\right.
$$

We now calculate the intersection numbers $\langle\tga^+_j,\tf_i-\iota(\tf_i)\rangle$ for $j=r,\ldots,d$. Recall that we have chosen $\te_j=\tf_j$, so the cycle $(\tga^+_j)^c$ is the unique cycle on $\tT^c\cup \{\tf_j\}$ containing $+\tf_j$. By our ordering convention, the edges of the tree $\tT^c$ are $\tf_i$ and $\iota(\tf_i)$ for $1\leq i\leq r-1$. Hence $(\tga^+_j)^c$ intersects $\tf_j$ with multiplicity $+1$, and does not intersect any $\tf_i$ or $\iota(\tf_i)$ with $i\geq r$. Therefore, for $j\geq r$ we have
$$
(\Psi^d)_{ji}=\left\langle \tga^+_j, \tf_i -\iota(\tf_i)\right\rangle=\langle (\tga^+_j)^c,\tf_i-\iota(\tf_i)\rangle=
\left\{
\begin{array}{cc} 0,\pm 1,\mbox{ or }\pm 2,& i\leq r-1, \\ 1, & i=j, \\ 0, & i\geq r,i\neq j. \end{array}\right.
$$
Putting everything together, we see that the $d\times d$ minor of $(\Psi^d)_{ji}$ corresponding to the partial basis $\tga_1,\ldots,\tga_d$ is a lower-triangular matrix, whose first $r-1$ diagonal entries are 2, and the remaining equal to 1. Hence, $\Psi^d$ has rank $d$.

\end{proof}

We now restrict our attention to the Abel--Prym map in degree $d=g-1$, which we denote $\Psi$:
$$
\Psi:\Sym^{g-1}(\tGa)\to \Prym^{[g-1]}(\tGa/\Ga),\quad \Psi(\tD)=\tD-\iota(\tD),
$$
In this case the source has the same dimension as the target, and we can compute the determinant of the matrix~\eqref{eq:dPsi} of $\Psi$ on any top-dimensional cell $C(\tF)$. This determinant depends on a choice of basis $\tga_1,\ldots,\tga_{g-1}$ for $\Ker \pi_*$, but only up to sign, hence the quantity
\begin{equation}
\deg_{\Psi} (\tF)=|\det \Psi(\tD)|,\quad \tD\in C(\tF),
\label{eq:degPsi}
\end{equation}
which we call the \emph{degree} of $\Psi$ on $C(\tF)$, is well-defined.

We now compute the degree of $\Psi$ on the top-dimensional cells of $\Sym^{g-1}(\tGa)$. We recall from Sec.~\ref{subsec:Prymorder} that, given a connected free double cover $p:\tG\to G$ of a graph $G$ of genus $g$, a subset $F\subset E(G)$ of $g-1$ edges of $G$ is called an \emph{odd genus one decomposition of rank $r$} if $G\backslash F$ consists of $r$ connected components of genus one, and each of them has connected preimage in $\tG$.

\begin{corollary} Let $\pi:\tGa\to \Ga$ be a free double cover with model $p:\tG\to G$, let $\Psi:\Sym^{g-1}(\tGa)\to \Prym^{[g-1]}(\tGa/\Ga)$ be the Abel--Prym map, let $C(\tF)=C^{g-1}(\tF,0)\subset \Sym^{g-1}(\tGa)$ be a top-dimensional cell corresponding to the multiset $\tF=\{\tf_1,\ldots,\tf_{g-1}\}\subset E(\tG)$, and let $F=p(\tF)$. Then $\deg_{\Psi}(\tF)$ is equal to
\begin{equation}
\deg_{\Psi}(\tF)=\left\{\begin{array}{cc} 2^{r-1}, & \mbox{edges of }F\mbox{ are distinct and form an odd genus one decomposition of rank }r, \\
0 & \mbox{otherwise.}
\end{array}\right.
\label{eq:degPsi2}
\end{equation}
In particular, the volume of the image of $C(\tF)$ in $\Prym^{[g-1]}(\tGa/\Ga)$ is equal to
\begin{equation}
\Vol(\Psi(C(\tF)))=\frac{2^{r(\tF)-1}w(F)}{\Vol(\Prym(\tGa/\Ga))},\quad w(F)=w(\tF)=\ell(\tf_1)\cdots \ell(\tf_{g-1})
\label{eq:volumeofPrymcell}
\end{equation}
if $F$ is an odd genus one decomposition of rank $r(\tF)$, and zero otherwise.
\label{cor:degreeformula}
\end{corollary}

\begin{proof} This follows directly from the proof of Theorem~\ref{thm:APlocal}.  If the edges of $F=p(\tF)$ are not all distinct, then $\Psi$ contracts the cell $C(\tF)$ and hence $\det \Psi=0$ on $C(\tF)$. If $F$ consists of distinct edges, let $G\backslash F=G_0\cup \cdots\cup G_{r-1}$ be the decomposition into connected components. By~\eqref{eq:genuscount} we have that $g_0+\cdots+g_{r-1}=r$, hence either $g_k=0$ for some $k$, or all $g_k=1$. In the former case $G_k$ is a tree, so the restriction of the cover $p$ to $G_k$ is trivial and hence $\det \Psi=0$ on $\tF$. In the latter case, $\Psi$ has rank $d=g-1$ if and only if the restriction of $p$ to each $G_k$ is nontrivial, which is true precisely when $F$ is an odd genus one decomposition. Furthermore, the matrix of $\Psi$ with respect to the basis~\eqref{eq:mainbasis} is lower triangular, with the first $r-1$ diagonal entries equal to 2, and the remaining equal to 1. Hence $\det \Psi=2^{r-1}$ on $C(\tF)$, as required.

To prove~\eqref{cor:degreeformula}, it is sufficient to note that $C(\tF)$ is a parallelotope with volume $w(\tF)$, and that $\Vol(\Prym(\tGa/\Ga))^{-1}$ is the volume of the unit cube in the coordinate system on $\Prym(\tGa/\Ga)$ that we used to compute the matrix of $\Psi$ (see Equation~\eqref{eq:volumeofunitcube}). 
\end{proof}

\section{Harmonicity of the Abel--Prym map}

In this section, we consider the degree $g-1$ Abel--Prym map $\Psi:\Sym^{g-1}(\tGa)\to \Prym^{[g-1]}(\tGa/\Ga)$ associated to a free double cover $\pi:\tGa\to \Ga$. The cellular decomposition of $\Sym^{g-1}(\tGa)$ induces a decomposition of $\Prym^{[g-1]}(\tGa/\Ga)$ (which is locally modelled on $\RR^{g-1}$). Pulling this decomposition back to $\Sym^{g-1}(\tGa/\Ga)$ and refining cells as needed, the Abel--Prym map $\Psi$ is a map of polyhedral spaces. We show that $\Psi$ is a harmonic map of polyhedral spaces of global degree $2^{g-1}$, with respect to the degree function~\eqref{eq:degPsi2}.

\begin{theorem} Let $\pi:\tGa\to \Ga$ be a free double cover of metric graphs. Then the Abel--Prym map
$$
\Psi:\Sym^{g-1}(\tGa)\to \Prym^{[g-1]}(\tGa/\Ga),\quad \Psi(\tD)=\tD-\iota(\tD)
$$
is a harmonic map of polyhedral spaces of global degree $2^{g-1}$, with respect to the degree function $\deg_{\Psi}$ defined on the codimension zero cells of $\Sym^{g-1}(\tGa)$ by~\eqref{eq:degPsi2}.

\label{thm:APharmonic}
\end{theorem}

The proof consists of two parts. First, we show that $\Psi$ is harmonic at each codimension one cell of $\Sym^{g-1}(\tGa/\Ga)$, and hence has a well-defined global degree $d$ because the polyhedral space $\Prym(\tGa/\Ga)$ is connected through codimension one. We then show that $d=2^{g-1}$. The proof of the first part is a somewhat involved calculation. We separate this result into a Proposition, and give its proof after the proof of the main Theorem~\ref{thm:APharmonic}. 

\begin{proposition} The degree $g-1$ Abel--Prym map
$$
\Psi:\Sym^{g-1}(\tGa)\to \Prym^{[g-1]}(\tGa/\Ga),\quad \Psi(\tD)=\tD-\iota(\tD)
$$
is harmonic at each codimension one cell of $\Sym^{g-1}(\tGa)$.

\label{prop:localharmonicity}
\end{proposition}

\begin{proof}[Proof of Theorem~\ref{thm:APharmonic}] By Proposition~\ref{prop:localharmonicity}, the Abel--Prym map has a certain global degree $d$. It may be possible to directly show that $d=2^{g-1}$, by somehow counting the preimages in a single fiber of $\Psi$, but we employ a different method. Namely, we use the harmonicity of the map $\Psi$ to give an alternative calculation of the volume of the Prym variety, in terms of the unknown global degree $d$ (similarly to how the volume of $\Jac(\Ga)$ is computed in~\cite{ABKS_Canonical}). However, we have already computed the volume of the Prym variety in Theorem~\ref{thm:Prymvolume}, using an entirely different method. Comparing the two formulas, we find that in fact $d=2^{g-1}$.

Let $M_i$ for $i=1,\ldots,N$ be the codimension zero cells of $\Prym^{[g-1]}(\tGa/\Ga)$, and let $\tM_{ij}$ for $j=1,\ldots,k_i$ be the codimension zero cells of $\Sym^{g-1}(\tGa)$ mapping surjectively to $M_i$. The cells $\tM_{ij}$ are obtained by refining the natural cellular decomposition of $\Sym^{g-1}(\tGa)$, in other words each $\tM_{ij}$ is a subset of a cell $C(\tF_{ij})$, where $\tF_{ij}\subset E(\tGa)$ is a subset of edges such that $p(\tF_{ij})$ is an odd genus one decomposition of $\Ga$ of some rank $r_{ij}=r(p(\tF_{ij}))$. Equation~\eqref{eq:volumeofPrymcell} gives the volume dilation factor of $\Psi$ on the cell $C(\tF_{ij})$, and hence on $\tM_{ij}$. Therefore
$$
\Vol(M_i)=\frac{2^{r_{ij}-1}\Vol(\tM_{ij})}{\Vol(\Prym(\tGa/\Ga))} 
$$
for all $i$ and $j$. On the other hand, the harmonicity condition implies that for each $i$ we have
$$
d=\sum_{j=1}^{k_i} \deg_{\Psi}(\tM_{ij})=\sum_{j=1}^{k_i} 2^{r_{ij}-1}.
$$
Putting this together, we can write
$$
\Vol(M_i)= \frac{1}{d}\sum_{j=1}^{k_i} 2^{r_{ij}-1}\cdot \Vol(M_i)=\frac{1}{d\cdot\Vol(\Prym(\tGa/\Ga))}\sum_{j=1}^{k_i}4^{r_{ij}-1}\Vol(\tM_{ij}).
$$
The sum of the volumes of the $M_i$ is the volume of the Prym variety:
$$
\Vol(\Prym(\tGa/\Ga))=\Vol(\Prym^{[g-1]}(\tGa/\Ga))=\sum_{i=1}^N\Vol(M_i)=
\frac{1}{d\cdot\Vol(\Prym(\tGa/\Ga))}\sum_{i,j}4^{r_{ij}-1}\Vol(\tM_{ij}).
$$
On the other hand, corresponding to each odd genus one decomposition $F$ of $\Ga$ there are $2^{g-1}$ subsets $\tF\subset E(\tGa)$ such that $p(\tF)=F$, because each decomposition has exactly $g-1$ edges. The volume $\Vol(C(\tF))$ of each of these cells is equal to $w(F)$. Each cell $C(\tF)$ corresponding to an odd genus one decomposition $F=p(\tF)$ is a disjoint union of some of the $M_{ij}$, and each $M_{ij}$ lies in some $C(\tF)$. Hence in fact the sum in the right hand side can be written as 
$$
\sum_{i,j}4^{r_{ij}-1}\Vol(M_{ij})=\sum_{\tF\subset E(\tGa)} 4^{r(p(\tF))-1}\Vol (\tF)=
2^{g-1}\sum_{F\subset E(\Ga)} 4^{r(F)-1}w(F),
$$
where the last sum is over all odd genus one decompositions $F$ of $\Ga$. Therefore,
$$
\Vol(\Prym(\tGa/\Ga))=\frac{2^{g-1}}{d\cdot\Vol(\Prym(\tGa/\Ga))}\sum_{F\subset E(\Ga)} 4^{r(F)-1}w(F).
$$
Comparing this formula with~\eqref{eq:Prymvolume}, we see that $d=2^{g-1}$.

\end{proof}

\begin{remark} Given a Prym divisor class represented by $\Psi(\tD)=\tD-\iota(\tD)$, the degree of $\Psi$ at $\tD\in \Sym^{g-1}(\tGa)$ in general depends on the choice of representative (in other words, the degree of $\Psi$ is not constant in fibers). We give an example of a free double cover $\pi:\tGa\to \Ga$ and effective divisors $\tD_1$ and $\tD_2$ on the source, such that $\tD_1-\iota(\tD_1)\simeq \tD_2-\iota(\tD_2)$, but the degrees of $\Psi$ at $\tD_1$ and $\tD_2$ are different. 

\begin{figure}
    \centering
    \includegraphics{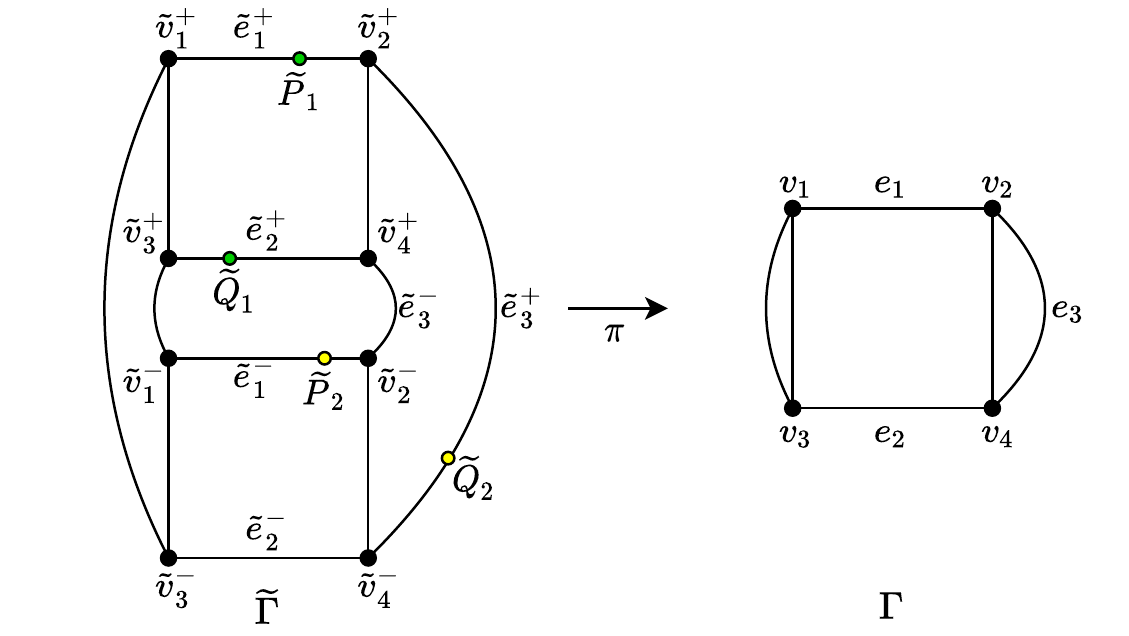}
    \caption{A double cover with a Prym divisor with representatives of distinct degrees.}
    \label{fig:irregularCover}
\end{figure}

Consider the free double cover $\pi:\tGa\to \Ga$ shown on Figure~\ref{fig:irregularCover}. The curves $\tGa$ and $\Ga$ have genera $5$ and $3$, respectively. The edge $\te_3$ has length at least $3$, while all other edges have length $1$. Fix real numbers $0<y<x<1$. Let $\tP_1$, $\tQ_1$, $\tP_2$, and $\tQ_2$ be the points on the edges $\te_1^+$, $\te_2^+$, $\te^-_1$, and $\te_3^+$, respectively, located at the following distances from the corresponding end vertices:
$$
d(\tv_1^+,\tP_1)=x,\quad d(\tv_3^+,\tQ_1)=y,\quad d(\tv_2^-,\tP_2)=x-y,\quad d(\tv_2^+,\tQ_2)=1+2y.
$$
Let $\tD_1=\tP_1+\tQ_1$ and $\tD_2=\tP_2+\tQ_2$. The divisor $(\tD_1-\iota \tD_1) - (\tD_2-\iota \tD_2)$ is seen to be equivalent to 0 by repeatedly applying Dhar's burning algorithm \cite[Section 5.1]{BakerShokrieh_Trees} at each point where the divisor has a negative number of chips. As a consequence,    $\tD_1-\iota \tD_1$ and $\tD_2-\iota \tD_2$ are linearly equivalent. However, $\pi(\tD_1)$ is supported on the odd genus one decomposition $\{\te_1,\te_2\}$ of rank $2$, while $\pi(\tD_2)$ is supported on the odd genus one decomposition $\{\te_1,\te_3\}$ of rank $1$, so the degrees of $\Psi$ at $\tD_1$ and $\tD_2$ are distinct.

By varying $x$ and $y$, we obtain two polyhedral cells in $\Sym^2(\tGa)$ having the same image in $\Prym(\tGa/\Ga)$: the subset $C_1=\{(x-y,1+2y):0<y<x<1\}$ of $\te^-_1\times \te^+_3$ and the subset $C_2=\{(x,y):0<y<x<1\}$ of $\te^+_1\times \te^+_2$. The volumes of $C_1$ and $C_2$ are equal to $1$ and $1/2$, respectively, which agrees with the fact that the degree of $\Psi$, equal to $1$ on $C_1$ and $2$ on $C_2$, is the volume dilation factor. We also observe that the global degree of $\Psi$ is equal to $2^{g-1}=4$. Therefore, Theorem~\ref{thm:C} implies that there is a third divisor $\tD_3$ (effective of degree two) such that $\tD_1-\iota \tD_1 \simeq \tD_3-\iota \tD_3$, and such that $\Gamma\setminus\pi(\tD_3)$ consists of  a single connected component.

\end{remark}

Before giving the proof of Proposition~\ref{prop:localharmonicity}, we describe the structure of the Abel--Prym map for the covers of a genus two dumbbell graph.

\begin{example} Consider the two covers $\pi_1:\tGa_1\to \Ga$ and $\pi_2:\tGa_2\to \Ga$ of the dumbbell graph $\Ga$ described in Example~\ref{ex:twocovers1}. In this case $g-1=1$, and the Abel--Prym maps $\Psi_1:\tGa_1\to \Prym(\tGa_1/\Ga)$ and $\Psi_2:\tGa_2\to \Prym(\tGa_2/\Ga)$ are harmonic morphisms of metric graphs of degree two, which we now describe.

With respect to the cover $\pi_1$, each edge of $\Ga$ is an odd genus one decomposition, hence $\Psi_1$ does not contract any edges. The edges $\te^{\pm}_1$ and $\te^{\pm}_2$ are mapped onto edges $f_1$ and $f_2$, respectively. The degree of $\Psi_1$ on these edges is equal to one, hence the lengths of $f_1$ and $f_2$ are $x_1$ and $x_2$, respectively. Each of the two edges $\te^{\pm}_3$ is mapped onto an edge $f^{\pm}_3$, of length $2x_3$ because the degree of $\Psi_1$ is equal to two. Hence $\Prym(\tGa_1/\Ga)$ is a circle of circumference $x_1+x_2+4x_3$, as we have already seen in Example~\ref{ex:twocovers1}.

The map $\Psi_2$, on the other hand, contracts the edges $\te^{\pm}_1$ and $\te^{\pm}_3$ because $\{e_1\}$ and $\{e_3\}$ are not genus one decompositions, and maps $\te^{\pm}_2$ to a unique loop edge $f_2$ of $\Prym(\tGa_2/\Ga)$ of length $x_2$. The morphisms $\Psi_1$ and $\Psi_2$ are given in Figure~\ref{fig:twocovers2}.

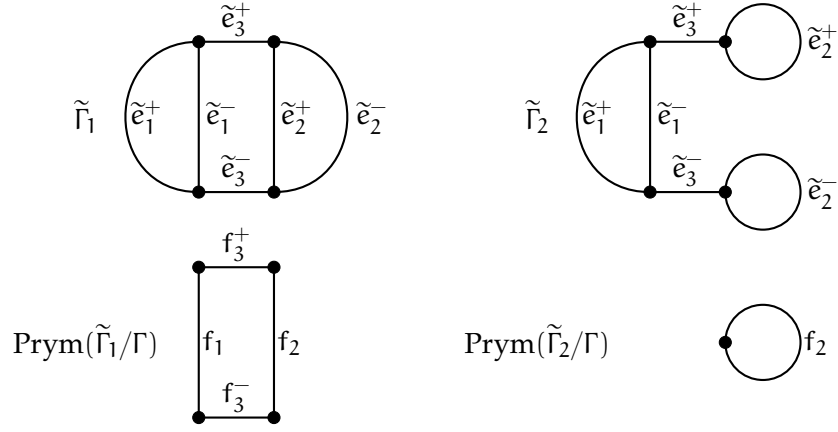
\begin{figure}

\begin{tikzpicture}

\draw[thick] (0,0) -- (1,0);
\draw[thick] (0,2) -- (1,2);
\draw[thick] (0,0) -- (0,2);
\draw[thick] (0,0) .. controls (-1.3,0) and (-1.3,2) .. (0,2);
\draw[thick] (1,0) -- (1,2);
\draw[thick] (1,0) .. controls (2.3,0) and (2.3,2) .. (1,2);

\node at (-1.5,1){$\tGa_1$};
\node at (-0.7,1) {$\te^+_1$};
\node at (0.3,1) {$\te^-_1$};
\node at (0.5,0.3) {$\te^-_3$};
\node at (0.5,2.3) {$\te^+_3$};
\node at (1.3,1) {$\te^+_2$};
\node at (2.3,1) {$\te^-_2$};

\draw[fill](0,0) circle(.08);
\draw[fill](1,0) circle(.08);
\draw[fill](0,2) circle(.08);
\draw[fill](1,2) circle(.08);

\begin{scope}[shift={(0,-3)}]

\draw[thick] (0,0) -- (1,0) -- (1,2) -- (0,2) -- (0,0);
\draw[fill](0,0) circle(.08);
\draw[fill](1,0) circle(.08);
\draw[fill](0,2) circle(.08);
\draw[fill](1,2) circle(.08);
\node at (0.2,1) {$f_1$};
\node at (0.5,2.3) {$f^+_3$};
\node at (0.5,0.3) {$f^-_3$};
\node at (1.2,1) {$f_2$};

\node at (-1.5,1) {$\Prym(\tGa_1/\Ga)$};

\end{scope}

\begin{scope}[shift={(6,0)}]

\draw[thick] (0,0) -- (1,0);
\draw[thick] (0,2) -- (1,2);
\draw[thick] (0,0) -- (0,2);
\draw[thick] (0,0) .. controls (-1.3,0) and (-1.3,2) .. (0,2);
\draw[thick] (1.5,0) circle(.5);
\draw[thick] (1.5,2) circle(.5);

\node at (-1.5,1){$\tGa_2$};
\node at (-0.7,1) {$\te^+_1$};
\node at (0.3,1) {$\te^-_1$};
\node at (0.5,0.3) {$\te^-_3$};
\node at (0.5,2.3) {$\te^+_3$};
\node at (2.3,0) {$\te^-_2$};
\node at (2.3,2) {$\te^+_2$};

\draw[fill](0,0) circle(.08);
\draw[fill](1,0) circle(.08);
\draw[fill](0,2) circle(.08);
\draw[fill](1,2) circle(.08);

\end{scope}

\begin{scope}[shift={(6.5,-3)}]

\draw[thick] (1,1) circle(.5);
\draw[fill](0.5,1) circle(.08);
\node at (1.7,1) {$f_2$};

\node at (-2,1) {$\Prym(\tGa_2/\Ga)$};

\end{scope}

\end{tikzpicture}

\caption{Abel--Prym maps corresponding to the covers in Example~\ref{ex:twocovers1}}
\label{fig:twocovers2}
\end{figure}

\end{example}

\begin{proof}[Proof of Proposition~\ref{prop:localharmonicity}] Let $\tC$ be a codimension one cell of $\Sym^{g-1}(\tGa)$ such that its image $C=\Psi(\tC)$ is a codimension one cell in $\Prym(\tGa/\Ga)$. Since $\Prym(\tGa/\Ga)$ is a torus, it locally looks like $\RR^{g-1}$, and we can think of the cell $C$ as lying in a hyperplane $H_0\subset \RR^{g-1}$, with respect to an appropriate local coordinate system. There are exactly two codimension zero cells $M^{\pm}$ attached to $C$, each contained in a corresponding half-space of $\RR^{g-1}$, which we also denote $M^{\pm}$ by abuse of notation. To show that $\Psi$ is harmonic at $\tC$, we need to show that the sum of $|\det\Psi|$ over the codimension zero cells of $\Sym^{g-1}(\tGa/\Ga)$ mapping to $M^+$ is the same as the sum over those mapping to $M^-$, in which case this sum is the degree of $\Psi$ on $\tC$. 

If $\Psi$ contracts every codimension zero cell of $\Sym^{g-1}(\tGa)$ attached to $\tC$, then the harmonicity condition is trivially verified, and we set $\deg_{\Psi}(\tC)=0$. Hence we assume that $\Psi$ does not contract some codimension zero cell attached to $\tC$. By Corollary~\ref{cor:degreeformula}, this cell is a subset of $C(\tF)$, where $F=p(\tF)$ is an odd genus one decomposition of $G$. If $\tC$ lies in the interior of $C(\tF)$, then $\Sym^{g-1}(\tGa)$ also locally looks like $\RR^n$ in a neighborhood of $\tC$, the map $\Psi$ is simply an affine linear map near $\tC$, and therefore harmonic (such cells $\tC$ do not occur in the standard polyhedral decomposition~\eqref{eq:Sympolyhedral} of $\Sym^{g-1}(\tGa/\Ga)$, but may occur in the refined decomposition induced by the map $\Psi$). 

We therefore assume that $\tC$ lies on the boundary of a cell $C(\tF)$, where $\tF=\{\tf_1,\ldots,\tf_{g-1}\}$ is a set of edges of $\tG$ such that $F=\{f_1,\ldots,f_{g-1}\}$, $f_i=p(\tf_i)$ is an odd genus one decomposition of $G$ of rank $r$. To simplify notation, we assume that in fact $\tC$ is a codimension one cell of $C(\tF)$ with respect to the standard polyhedral decomposition~\eqref{eq:Sympolyhedral} of $\Sym^{g-1}(\tGa/\Ga)$. In other words, we assume that $\tC=C^{g-2}(\tF\backslash\{\tf_a\},\tv)$ for some $a\in \{1,\ldots,g-1\}$, and where $\tv=s(\tf_a)$ is the starting vertex of $\tf_a$ with respect to an appropriate orientation (we shall later specify which edge $\tf_a$ we pick, in order to make our notation consistent with Construction~\ref{con:C}). 

The top dimensional cells of $\Sym^{g-1}(\tGa)$ that are adjacent to $\tC$ have the form $C(\tF')$, where $\tF'=(\tF\backslash \{\tf_a\})\cup\{\tf'\}$ and where $\tf'$ is any edge rooted at $\tv$. We assume that all edges $\tf'$ are oriented in such a way that $s(\tf')=\tv$. By Corollary~\ref{cor:degreeformula}, $\Psi$ does not contract $C(\tF')$ if and only if $p(\tF')$ is an odd genus one decomposition of $G$. To prove harmonicity, we need to show that the sum of $|\det \Psi|$ on those cells $C(\tF')$ mapping to $M^+$ is equal to the sum of those that map to $M^-$. By Corollary~\ref{cor:degreeformula}, the value of $|\det \Psi|$ on a non-contracted cell $C(\tF')$ is a power of two. In fact, as we shall see, adjacent to any cell $\tC$ there are either two, three, or four non-contracted cells $C(\tF)$, with the degrees distributed as shown on Figure~\ref{fig:localharmonic} (plus an arbitrary number of contracted cells).

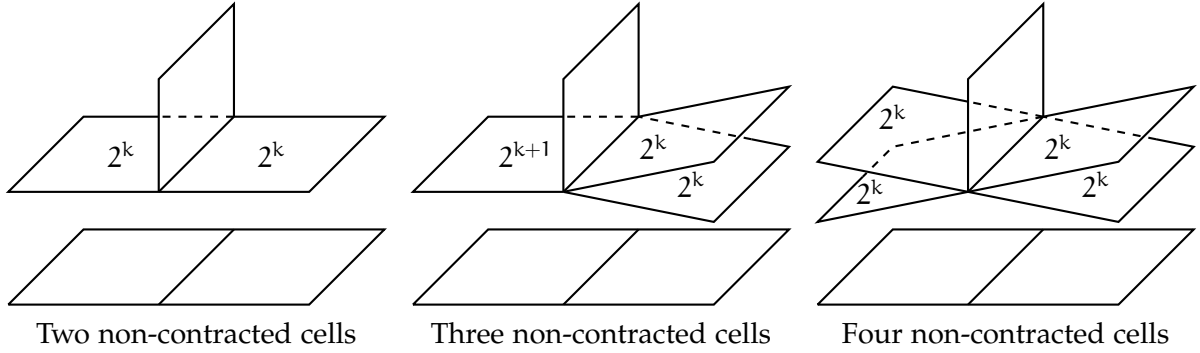
\begin{figure}[h]
    \centering
    \begin{tabular}{ccc}

\begin{tikzpicture}
\begin{scope}[shift={(0,1.5)}]
\draw[thick] (2,1) -- (1,1) -- (0,0) -- (4,0) -- (5,1) -- (3,1);
\draw[thick,dashed] (3,1) -- (2,1);
\draw[thick] (2,0) -- (3,1);
\draw[thick] (2,0) -- (2,1.5) -- (3,2.5) -- (3,1);
\node at (1.5,0.5) {$2^k$};
\node at (3.5,0.5) {$2^k$};

\end{scope}

\draw[thick] (0,0) -- (4,0) -- (5,1) -- (1,1) -- (0,0);
\draw[thick] (2,0) -- (3,1);

\end{tikzpicture}
& 
\begin{tikzpicture}
\begin{scope}[shift={(0,1.5)}]
\draw[thick] (2,1) -- (1,1) -- (0,0) -- (2,0);
\draw[thick,dashed] (3,1) -- (2,1);
\draw[thick] (2,0) -- (3,1);
\draw[thick] (2,0) -- (2,1.5) -- (3,2.5) -- (3,1);
\node at (1.5,0.5) {$2^{k+1}$};

\draw[thick] (2,0) -- (4,0.4) -- (5,1.4) -- (3,1);
\draw[thick] (2,0) -- (4,-0.4) -- (5,0.6) -- (4.5,0.7);
\draw[thick,dashed] (4.5,0.7) -- (3,1);
\node at (3.2,0.6) {$2^k$};
\node at (3.7,0.1) {$2^k$};

\end{scope}

\draw[thick] (0,0) -- (4,0) -- (5,1) -- (1,1) -- (0,0);
\draw[thick] (2,0) -- (3,1);
\end{tikzpicture}
&
\begin{tikzpicture}
\begin{scope}[shift={(0,1.5)}]
\draw[thick,dashed] (3,1) -- (2,1.2);
\draw[thick] (2,1.2) -- (1,1.4) -- (0,0.4) -- (2,0);
\draw[thick] (2,0) -- (3,1);
\draw[thick] (2,0) -- (2,1.5) -- (3,2.5) -- (3,1);

\draw[thick,dashed] (3,1) -- (1,0.6) -- (0.6,0.2);
\draw[thick] (0.6,0.2) -- (0,-0.4) -- (2,0);

\draw[thick] (2,0) -- (4,0.4) -- (5,1.4) -- (3,1);
\draw[thick] (2,0) -- (4,-0.4) -- (5,0.6) -- (4.5,0.7);
\draw[thick,dashed] (4.5,0.7) -- (3,1);
\node at (3.2,0.6) {$2^k$};
\node at (3.7,0.1) {$2^k$};

\node at (1,1) {$2^k$};
\node at (0.7,0) {$2^k$};

\end{scope}

\draw[thick] (0,0) -- (4,0) -- (5,1) -- (1,1) -- (0,0);
\draw[thick] (2,0) -- (3,1);
\end{tikzpicture}
\\
Two non-contracted cells & Three non-contracted cells & Four non-contracted cells
\end{tabular}
    \caption{The Abel--Prym map near a non-contracted codimension one cell.}
    \label{fig:localharmonic}
\end{figure}
\begin{center}

\end{center}

We calculate the matrix of $\Psi$ (or, rather, some of its entries) on each cell $C(\tF')$ with respect to an appropriate coordinate system, in the same way that we proved part (2) of Theorem~\ref{thm:APlocal}. First, we choose local coordinates on $\Sym^{g-1}(\tGa/\Ga)$. As before, we identify $C(\tF)$ with the parallelotope $[0,\ell(\tf_1)]\times \cdots\times [0,\ell(\tf_{g-1})]$ lying in the half-space $H^+=\{x:x_a\geq 0\}\subset \RR^{g-1}$. Under this identification, the cell $\tC$ lies in the hyperplane $H^0=\{x:x_a=0\}$, and the corresponding cells of $\Prym(\tGa/\Ga)$ are $C=\Psi(\tC)\subset \Psi(H^0)$ and $M^{\pm}=\Psi(H^{\pm})$, where $H^-=\{x:x_a\leq 0\}\subset \RR^{g-1}$. Similarly, we think of each of the other $C(\tF')$ as lying in its own $H^+$.

To construct coordinates on $\Prym(\tGa/\Ga)$, we apply Construction~\ref{con:C} to the set $\tF$. The output is a basis $\tga_1,\ldots,\tga_{g-1}$ of $\Ker \pi_*:H_1(\tGa,\ZZ)\to H_1(\Ga,\ZZ)$ given by Equation~\eqref{eq:mainbasis}. As explained in Subsection~\ref{subsec:Prymvarieties}, the basis $\tga_1,\ldots,\tga_{g-1}$ defines a coordinate system on $\Prym(\tGa/\Ga)$, with respect to which the map $\Psi$ on the cell $C(\tF)$ is affine linear, and the $(g-1)\times (g-1)$ matrix of the linear part is given by Equation~\eqref{eq:dPsi}:
$$
\Psi(\tF)_{ji}=\frac{1}{2}\langle \tga_j,\tf_i-\iota(\tf_i)\rangle=
\frac{1}{2}\langle \tga^+_j-\iota_*(\tga^+_j),\tf_i-\iota(\tf_i)\rangle=
\langle \tga^+_j,\tf_i-\iota(\tf_i)\rangle
$$
We recall that we showed in Theorem~\ref{thm:APlocal} and Corollary~\ref{cor:degreeformula} that $\Psi(\tF)_{ij}$ is a lower triangular matrix with determinant $2^{r-1}$, where $r$ is the rank of $\tF$. 

Now let $C(\tF')$ be another codimension zero cell of $\Sym^{g-1}(\tGa/\Ga)$ adjacent to $\tC$, so $\tF'=(\tF\backslash \{\tf_a\})\cup\{\tf'\}$, where $\tf'$ is an edge rooted at $\tv$ other than $f_a$. We calculate the matrix of $\Psi$ on $C(\tF')$ using the same basis $\tga_1,\ldots,\tga_{g-1}$ (in other words, we do not recalculate the basis by replacing $\tF$ with $\tF'$ in Construction~\ref{con:C}). The resulting matrix differs from $\Psi(\tF)_{ji}$ by a single column only:
\begin{equation}
\Psi(\tF')_{ji}=\left\{\begin{array}{cc} \Psi(\tF)_{ji}, & i\neq a, \\
\langle \tga^+_j,\tf'-\iota(\tf')\rangle, & i=a.
\end{array}\right.
\label{eq:Psiprime}
\end{equation}

To check the harmonicity of $\Psi$ around $\tC$, it suffices to compute the determinants $\det \Psi(\tF')$ for all $\tF'$. Indeed, the Abel--Prym map $\Psi$ contracts the cell $C(\tF')$ if and only if $\det \Psi(\tF')=0$. Furthermore, $C(\tF')$ maps to $M^+$ if $\det \Psi(\tF')>0$ and to $M^-$ if $\det \Psi(\tF')<0$, and to prove harmonicity we need to check that the positive determinants exactly cancel the negative determinants.

The set $F$ is an odd genus one decomposition of $G$ of some rank $r$, and we denote 
$$
G\backslash F=G_0\cup \cdots\cup G_{r-1}
$$
the decomposition into connected components. Each $G_k$ has genus one, and each $p^{-1}(G_k)$ is connected. We denote
$$
\tu=t(\tf_a), \quad v=p(\tv)=s(f_a), \quad u=p(\tu)=t(f_a).
$$

There are two separate cases that we need to consider: either both endpoints $u$ and $v$ of the edge $f_a=p(\tf_a)$ that we are removing lie on one connected component, or the edge $f_a$ connects two different components.

\medskip \noindent {\bf Both endpoints of the edge $f_a$ lie on a single connected component of $G\backslash F$.} Without loss of generality, we assume that $f_a$ is rooted on the component $G_0$. The edge $f_a$ is a loop on the contracted graph $G^c$ rooted at the vertex $v_0$. Since a loop cannot be part of a spanning tree, we can further assume without loss of generality that $\tf_a=\tf_{g-1}$. We observe that, on the contracted graph $\tG^c$, the edge $\tf_{g-1}$ is a loop rooted at $\tv_0$. The contraction of the cycle $\tga^+_{g-1}$ is the unique cycle containing $+\tf_{g-1}$, but since this is already a loop we see that $(\tga^+_{g-1})^c=\tf_{g-1}$.

It follows that the intersection of $\tga^+_{g-1}$ with all other edges $\tf_i$ and $\iota(\tf_i)$ for $i=1,\ldots,g-2$ is zero. Hence the matrix $\Psi_{ji}$ is block upper triangular, having a $(g-2)\times (g-2)$ lower triangular block with determinant $2^{r-1}$ in the upper left corner, and a 1 in the lower right corner. Therefore, the images of the subspaces $H^{\pm}$ and the hyperplane $H^0$ are
$$
M^+=\Psi(H^+)=\{y:y_{g-1}\geq 0\}, \quad M^-=\Psi(H^-)=\{y:y_{g-1}\leq 0\}, \quad  \Psi(H^0)=\{y:y_{g-1}=0\}.
$$
Now let $\tf'$ be an edge at $\tv$, so that $\tF'=\{\tf_1,\ldots,\tf_{g-2},\tf'\}$ defines a cell $C(\tF')$ adjacent to $C(\tF)$ via $C'$. The matrix $\Psi(\tF')_{ji}$ is given by~\eqref{eq:Psiprime}, and is obtained from the matrix $\Psi(\tF)$ by replacing the last column. Hence it is also block upper triangular, and to compute $\det \Psi(\tF')$ it suffices to find the new entry
\begin{equation}
\Psi(\tF')_{g-1,g-1}=\langle \tga^+_{g-1},\tf'-\iota(\tf')\rangle
\label{eq:Psig-1}
\end{equation}
in the lower right corner. Furthermore, the sign of this entry determines the sign of $\det \Psi(\tF')$ and hence the image cell $\Psi(C(\tF'))$ of $\Prym(\tGa/\Ga)$: if the entry is positive, then $\Psi$ maps $C(\tF')$ to the same half-space $M^+$ as $C(\tF)$, while if it is negative then $\Psi(C(\tF'))\subset M^-$, and if it is zero then $C(\tF')$ is contracted. 

There are several possibilities to consider, depending on the relative positions of $v=s(f_{g-1})$ and $u=t(f_{g-1})$ on the component $G_0$. Let $\ga(G_0)$ denote the unique cycle on $G_0$ (oriented in any direction), then any vertex of $G_0$ has a unique (possibly trivial) shortest path to $\ga(G_0)$. For two distinct vertices $v_1,v_2\in V(G_0)$, we write $v_1<v_2$ if the unique path from $v_2$ to $\ga(G_0)$ passes through $v_1$; this defines a partial order on $V(G_0)$. 

\begin{enumerate}
    \item \label{case:A1} The vertex $v$ does not lie on $\ga(G_0)$, and $v\nless u$. In other words, $v$ lies on a tree attached to $\ga(G_0)$, and $u$ does not lie higher up on the same tree. 
    
    Let $g_1$ be the unique edge rooted at $v$ that points in the direction of the cycle $\ga(G_0)$. Since the unique path from $u$ to $\ga(G_0)$ avoids $v$, the graph $G_0'=G_0\cup\{f_{g-1}\}\backslash \{g_1\}$ is connected, has genus one, and has connected preimage, since the unique cycle of $G_0'$ is $\ga(G_0)$. Therefore,  $F_1=\{f_1,\ldots,f_{g-2},g_1\}$ is an odd genus one decomposition of $G$, of the same length $r$ as $F$. For any other edge $e'$ rooted at $v$, removing it disconnects the corresponding branch of the tree  from $G_0$, and attaching $f_{g-1}$ does not reconnect this branch. Hence $G\backslash \{f_1,\ldots,f_{g-2}, e'\}$ has a connected component of genus zero, and $\{f_1,\ldots,f_{g-2}, e'\}$ is not a genus one decomposition. The graph $G_0$ and its preimage $\pi^{-1}(G_0)$ are shown on Figure~\ref{fig:wcA1}.
    
    \begin{figure}
    \centering
    \includegraphics{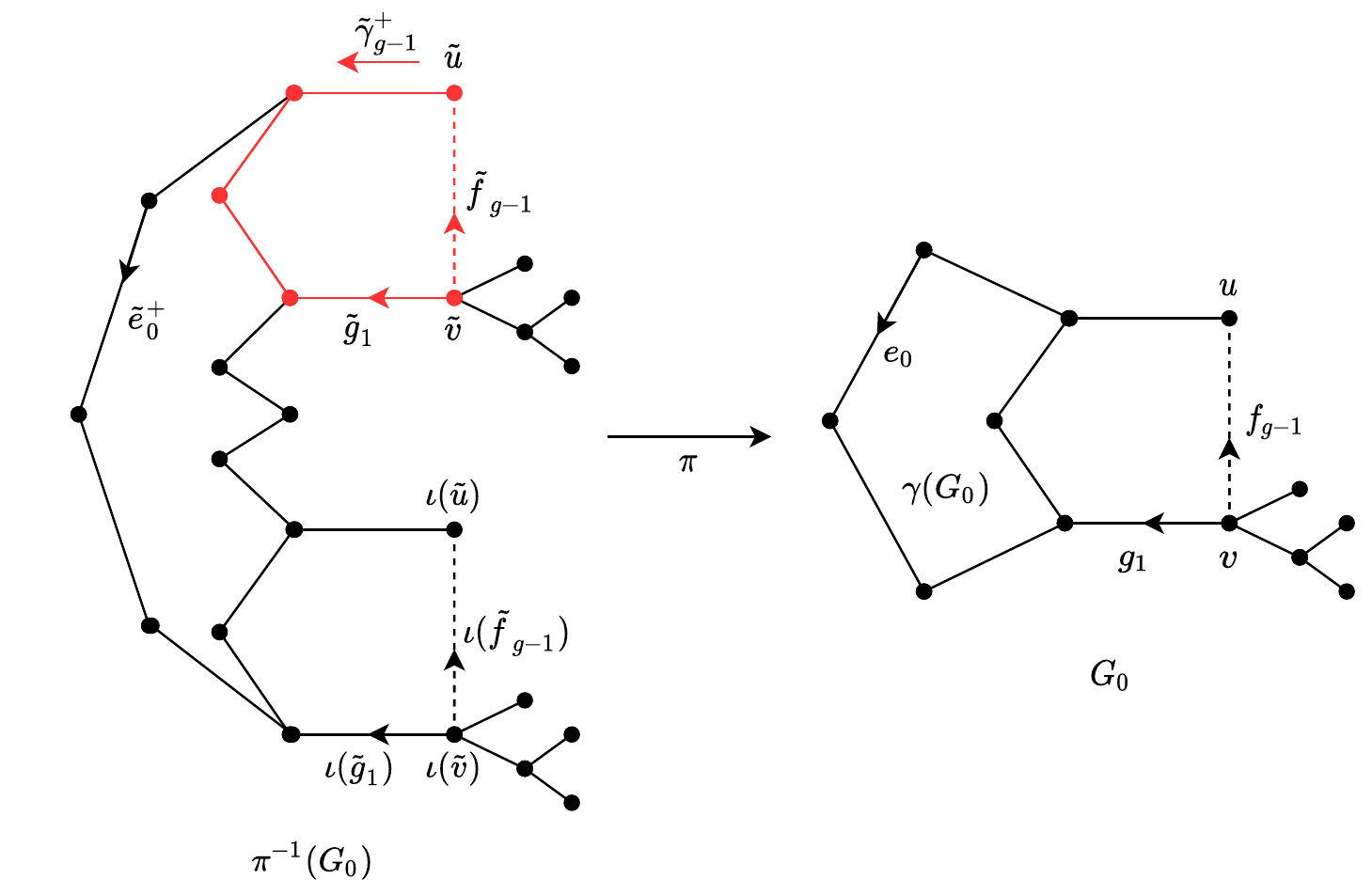}
    \caption{The cycle $\tga^+_{g-1}$ in Case (\ref{case:A1})}
    \label{fig:wcA1}
\end{figure}
    
    We see that the only cell $C(\tF')$ adjacent to $C(\tF)$ through $C'$ on which $|\det \Psi|$ is nonzero corresponds to $\tF'=\tF_1=\tF\cup \{\tg_1\}\backslash \{\tf_{g-1}\}$, where $\tg_1$ is the unique edge rooted at $\tv$ that maps to $g_1$. Furthermore, $F$ and $F_1=p(\tF_1)$ have the same rank $r$, hence the value of $|\det \Psi|$ on the two cells $C(\tF)$ and $C(\tF_1)$ is equal to $2^{r-1}$, so to prove harmonicity we only need to show that $\Psi$ maps $C(\tF_1)$ to the half-space $M^-$. As explained above, it suffices to compute the last diagonal entry~\eqref{eq:Psig-1} of $\Psi(\tF_1)$, where $\tf'=\tg_1$.

    The cycle $\tga^+_{g-1}$ is the unique cycle of the graph $\tT\cup \{\tf_{g-1}\}$ containing $+\tf_{g-1}$. It starts at the vertex $\tv=s(\tf_{g-1})$, proceeds to $\tu=t(\tf_{g-1})$ via $+\tf_{g-1}$, and then from $\tu$ back to $\tv$ via the unique path in the tree $\tT$. This path actually lies in the spanning tree $\tT_0^+\cup \tT_0^-\cup \{\te^+_0\}$ of $p^{-1}(G_0)$. The last edge of the path is $\tg_1$, oriented in the opposite direction since we've assumed that $s(\tg_1)=\tv$, hence $\langle \tga^+_{g-1},\tg_1\rangle=-1$. In addition, the path does not contain $\iota(\tg_1)$. It follows that
    $$
\Psi(\tF_1)_{g-1,g-1}=\langle\tga^+_{g-1},\tg_1-\iota(\tg_1)\rangle=-1.
$$
Therefore $\Psi$ maps the cell $C(\tF_1)$ to the half-space $M^-$, hence $\Psi$ is harmonic.

    \item \label{case:A2} The vertex $v$ does not lie on $\ga(G_0)$, and $v<u$. As before, let $g_1$ denote the unique edge at $v$ pointing towards $\ga(G_0)$, and let $g_2$ be the unique edge rooted at $v$ which lies on the path from $v$ to $u$ (this path, when reversed, is part of the unique path from $u$ to $\ga(G_0)$). Attaching $f_{g-1}$ to $G_0$ produces a graph of genus two. Any edge $e'$ rooted at $v$ other than $f_{g-1}$, $g_1$, or $g_2$ is the starting edge of a separate branch of $G_0\cup \{f_{g-1}\}$, so removing $e'$ creates a genus zero connected component. Hence the only genus one decompositions of the form $(F\backslash \{f'\})\cup \{f_{g-1}\}$ are $F_1=\{f_1,\ldots,f_{g-2},g_1\}$, $F_2=\{f_1,\ldots,f_{g-2},g_2\}$, and $F$ itself. The decompositions $F$ and $F_2$ have length $r$, while $F_1$ has length $r+1$, because the edge $g_1$ is a bridge edge of $G_0\cup \{f_{g-1}\}$, and removing it produces two genus one connected components. 
    
    We now consider the edges $\tg_1$, $\tg_2$, and $\tf_{g-1}$ on $\tG$, lying above $g_1$, $g_2$, and $f_{g-1}$ and rooted at $\tv=s(\tf_{g-1})$. Denote $\tF_1=\{\tf_1,\ldots,\tf_{g-2},\tg_1\}$ and $\tF_2=\{\tf_1,\ldots,\tf_{g-2},\tg_2\}$. The edges $g_1$ and $g_2$ lie on the same tree attached to the cycle $\ga(G_0)$ as the vertex $v$, and the lift of a tree is a tree. Hence the endpoints of the edges $\tg_1$ and $\tg_2$ both lie on the same subtree $\tT_0^{\pm}$ of $p^{-1}(G_0)$ as $\tv$, and we assume without loss of generality that this component is $\tT_0^+$. For $t(\tf_{g-1})$, however, there are two sub-possibilities, as shown on Figure~\ref{fig:wcA2}.
    
    \begin{figure}
    \centering
    \includegraphics{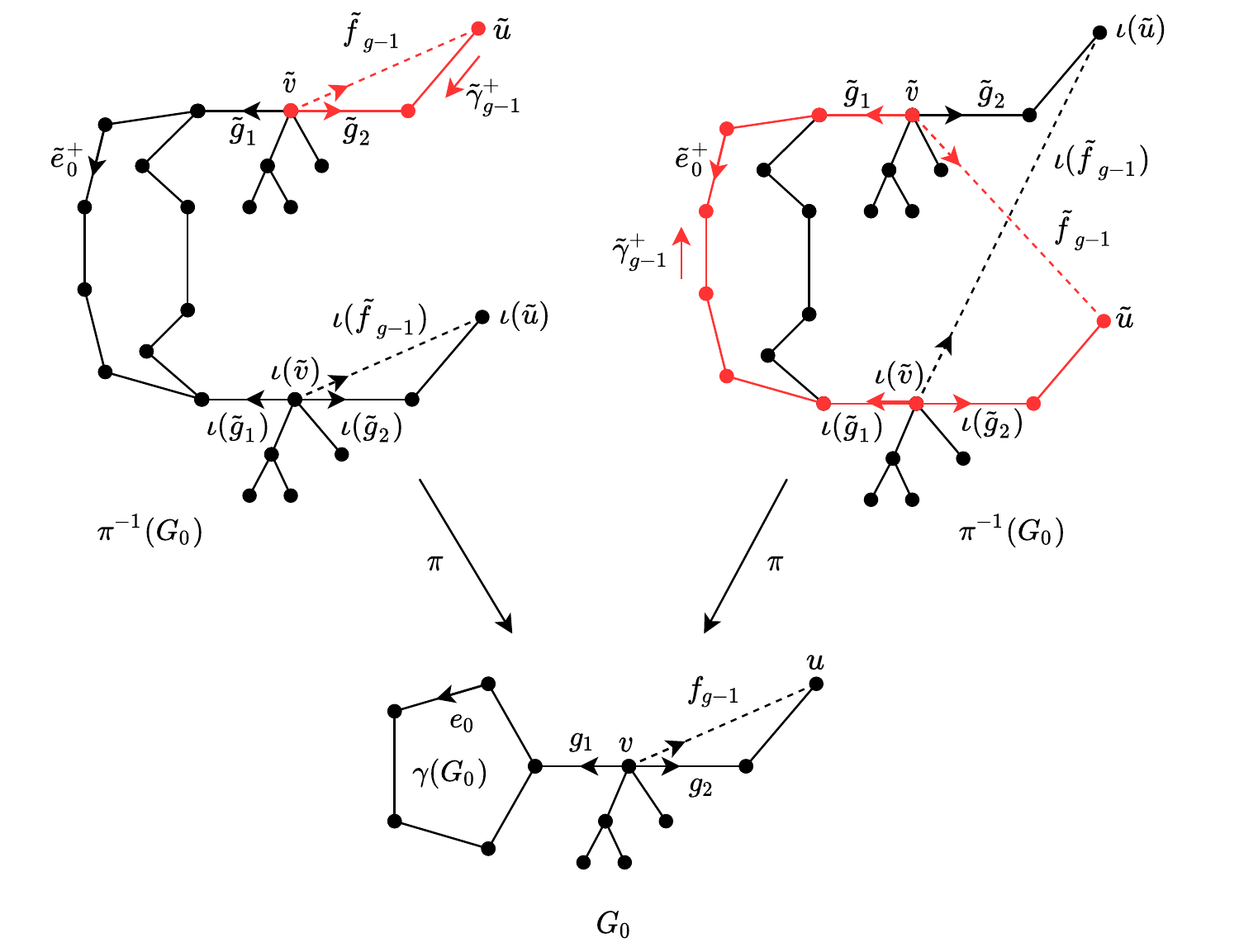}
    \caption{The cycle $\tga^+_{g-1}$ in the two sub-cases of Case (\ref{case:A2})}
    \label{fig:wcA2}
\end{figure}
    
    \begin{enumerate}
        \item The target vertex $\tu=t(f_{g-1})$ lies on $\tT_0^+$. In this case, the unique cycle $\tga^+_{g-1}$ of the graph $\tT\cup \{\tf_{g-1}\}$ actually lies on $\tT_0^+\cup\{\tf_{g-1}\}$: it starts at $\tv$, proceeds to $\tu$ via $+\tf_{g-1}$, and then returns to $\tv$ via the unique path that ends with the edge $-\tg_2$. It follows that $\langle \tga^+_{g-1},g_2\rangle =-1$ and $\langle \tga^+_{g-1},\iota(g_2)\rangle =0$. Furthermore, the cycle $\tga^+_{g-1}$ does not intersect the edges $\tg_1$ and $\iota(\tg_1)$. Hence we can compute the last diagonal entries of the upper-triangular matrices $\Psi(\tF_1)$ and $\Psi(\tF_2)$:
        $$
        \Psi(\tF_1)_{g-1,g-1}=\langle\tga^+_{g-1},\tg_1-\iota(\tg_1)\rangle=0,\quad \Psi(\tF_2)_{g-1,g-1}=\langle\tga^+_{g-1},\tg_2-\iota(\tg_2)\rangle=-1. 
        $$
        It follows that $|\det \Psi(\tF_1)|=0$, hence the cell $C(\tF_1)$ is contracted. Also, $\Psi$ maps the cell $C(\tF_2)$ to the opposite half-space $M^-$ as $C(\tF)$, but with the same determinant, since $\tF$ and $\tF_2$ have the same rank $r$. Hence $\Psi$ is harmonic.
        
        \item \label{case:2b} The target vertex $\tu=t(f_{g-1})$ lies on $\tT_0^-$. In this case, the cycle $\tga^+_{g-1}$ starts at $\tv$, proceeds to $\tu$ via $\tf_{g-1}$, and proceeds to $\iota(\tv)$ via a unique path that ends with the edge $-\iota(\tg_2)$. From there the path returns from $\iota(\tv)$ to $\tv$ via the unique path that passes through the edge $\te^+_0$ that links the two trees $\tT^{\pm}_0$; this path starts with the edge $\iota(\tg_1)$ and ends with $-\tg_1$. Summarizing, we see that
        $$
        \langle \tga^+_{g-1},\tg_1 \rangle=-1,\quad \langle \tga^+_{g-1},\iota(\tg_1) \rangle=1,\quad \langle \tga^+_{g-1},\tg_2 \rangle=0,\quad \langle \tga^+_{g-1},\iota(\tg_1) \rangle=-1.
        $$
        Hence we calculate the final diagonal entries of $\Psi(\tF_1)$ and $\Psi(\tF_2)$:
        $$
        \Psi(\tF_1)_{g-1,g-1}=\langle\tga^+_{g-1},\tg_1-\iota(\tg_1)\rangle =-2,\quad
        \Psi(\tF_2)_{g-1,g-1}=\langle\tga^+_{g-1},\tg_2-\iota(\tg_2)\rangle =1.
        $$
        It follows that $\Psi$ maps $C(\tF_2)$ to the same half-space $M^+$ with the same determinant $|\det d \Psi(\tF_2)|=|\det d \Psi(\tF)|=2^{r-1}$, while $C(\tF_1)$ is mapped to the opposite space $M^-$ with determinant $|\det d \Psi(\tF_1)|=2^r$. Since $2^r=2^{r-1}+2^{r-1}$, the map $\Psi$ is harmonic.
    \end{enumerate}
    
    \item \label{case:A3} The vertex $v$ lies on $\ga(G_0)$, and $v\nless u$. Let $g_1$ and $g_2$ be the two edges of $G_0$ rooted at $v$ that lie on the cycle $\ga(G_0)$, then $F_1=\{f_1,\ldots,f_{g-2},g_1\}$ and $F_2=\{f_1,\ldots,f_{g-2},g_2\}$ are genus one decompositions of $G_0$ of the same rank $r$ as $F$, since removing $g_1$ or $g_2$ from $G_0\cup \{f_{g-1}\}$ gives a connected graph of genus one. Any edge $f'\in T_vG_0$ other than $g_1$ and $g_2$ is the starting edge of a separate tree which does not contain $u=t(f_{g-1})$, so $G_0\cup\{f_{g-1}\}\backslash \{f'\}$ has a genus zero connected component, and $\{f_1,\ldots,f_{g-2},f'\}$ is not a genus one decomposition. 
    
    Let $\tg_1$ and $\tg_2$ be the edges of $\tG$ at $\tv$ lying above $g_1$ and $g_2$, respectively, and denote $\tF_1=\{\tf_1,\ldots,\tf_{g-2},\tg_1\}$ and $\tF_2=\{\tf_1,\ldots,\tf_{g-2},\tg_2\}$. The preimage of the cycle $\ga(G_0)$ is the unique cycle $\ga(p^{-1}(G_0))$ of the genus one graph $p^{-1}(G_0)$. We orient this cycle so that it starts with the edge $\tg_1$, passes through $\iota(\tv)$, and ends with $-\tg_2$. Let $\tu'$ be the end vertex of the unique shortest path from $\tu$ to $\ga(p^{-1}(G_0))$; this vertex may be $\tu$ itself but cannot be $\tv$ or $\iota(\tv)$, since we have assumed that the shortest path from $u$ to $\ga(G_0)$ does not pass through $v$. We now assume without loss of generality that $\tu'$ lies on the same path from $\tv$ to $\iota(\tv)$ as $\tg_1$, otherwise exchange $\tg_1$ and $\tg_2$ (see Figure~\ref{fig:wcA3}). 
    
    \begin{figure}
    \centering
    \includegraphics{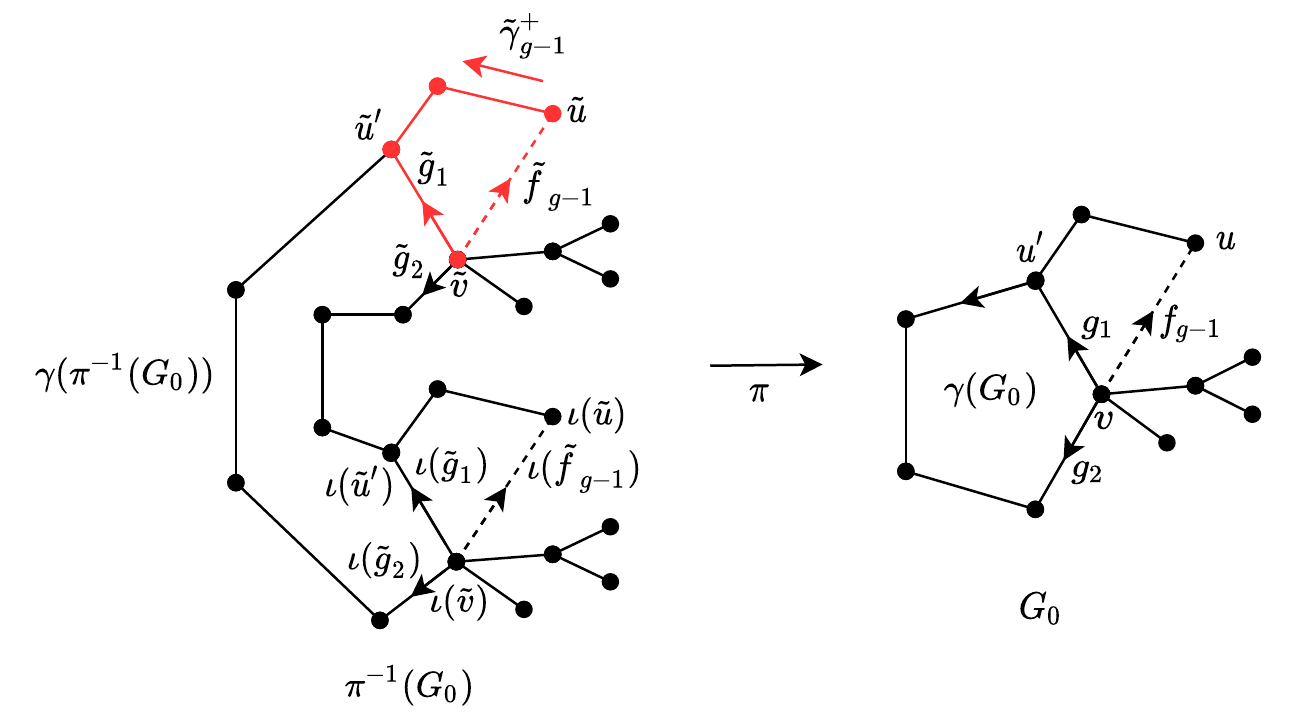}
    \caption{The cycle $\tga^+_{g-1}$ in Case (\ref{case:A3})}
    \label{fig:wcA3}
\end{figure}
    
    All cells adjacent to $\tC$ other than $C(\tF)$, $C(\tF_1)$, and $C(\tF_2)$ are contracted. For the last two, we need to compute the matrix entry~\eqref{eq:Psig-1}. We now calculate the relevant intersection numbers. The path $\tga^+_{g-1}$ starts at $\tv$, proceeds via $+\tf_{g-1}$ to $\tu$, then to $\tu'$, and then back to $\tv$ along a path lying in $\ga(p^{-1}(G_0))$ that ends with $-\tg_1$, and does not contain $\iota(\tg_1)$, $\tg_2$, or $\iota(\tg_2)$. It follows that
    $$
    \Psi(\tF_1)_{g-1,g-1}=\langle\tga^+_{g-1},\tg_1-\iota(\tg_1)\rangle =-1,\quad\Psi(\tF_2)_{g-1,g-1}=\langle\tga^+_{g-1},\tg_2-\iota(\tg_2)\rangle =0.
    $$
    Therefore $\Psi$ maps $C(\tF_1)$ to the opposite side $M^-$ as $C(\tF)$, but with the same determinant $|\det \Psi|=2^{r-1}$. On the other hand, $C(\tF_2)$ is contracted (this can also be seen by noting that the preimage of the graph $G_0\cup\{f_{g-1}\}\backslash \{g_2\}$ is disconnected). Hence $\Psi$ is harmonic. 
    
    \item \label{case:A4} Finally, we consider the possibility that $v$ lies on $\ga(G_0)$ and that $v<u$, in other words $u$ lies on a tree attached to $v$. In this case, there are three edges at $v$ that give genus one decompositions: the two edges $g_1$ and $g_2$ lying on the cycle $\ga(G_0)$, and the edge $g_3$ that starts the unique path from $v$ to $u$. All other edges $e'$ at $v$ support trees, and their removal from $G_0\cup \{f_{g-1}\}$ produces a connected component of genus zero. 
    
    For $i=1,2,3$ denote $\tg_i$ the lift of $g_i$ rooted at $\tv$, and denote $\tF_i=\{\tf_1,\ldots,\tf_{g-2},\tg_i\}$. As in Case~\ref{case:A2} above, there are two subcases, depending on whether the target vertex $\tu$ lies on the same tree $T^{\pm}_0$ as $\tv$ (say $T^+_0$), or on the other tree. The two possibilities are shown on Figure~\ref{fig:wcA4}.
    
        \begin{figure}
    \centering
    \includegraphics{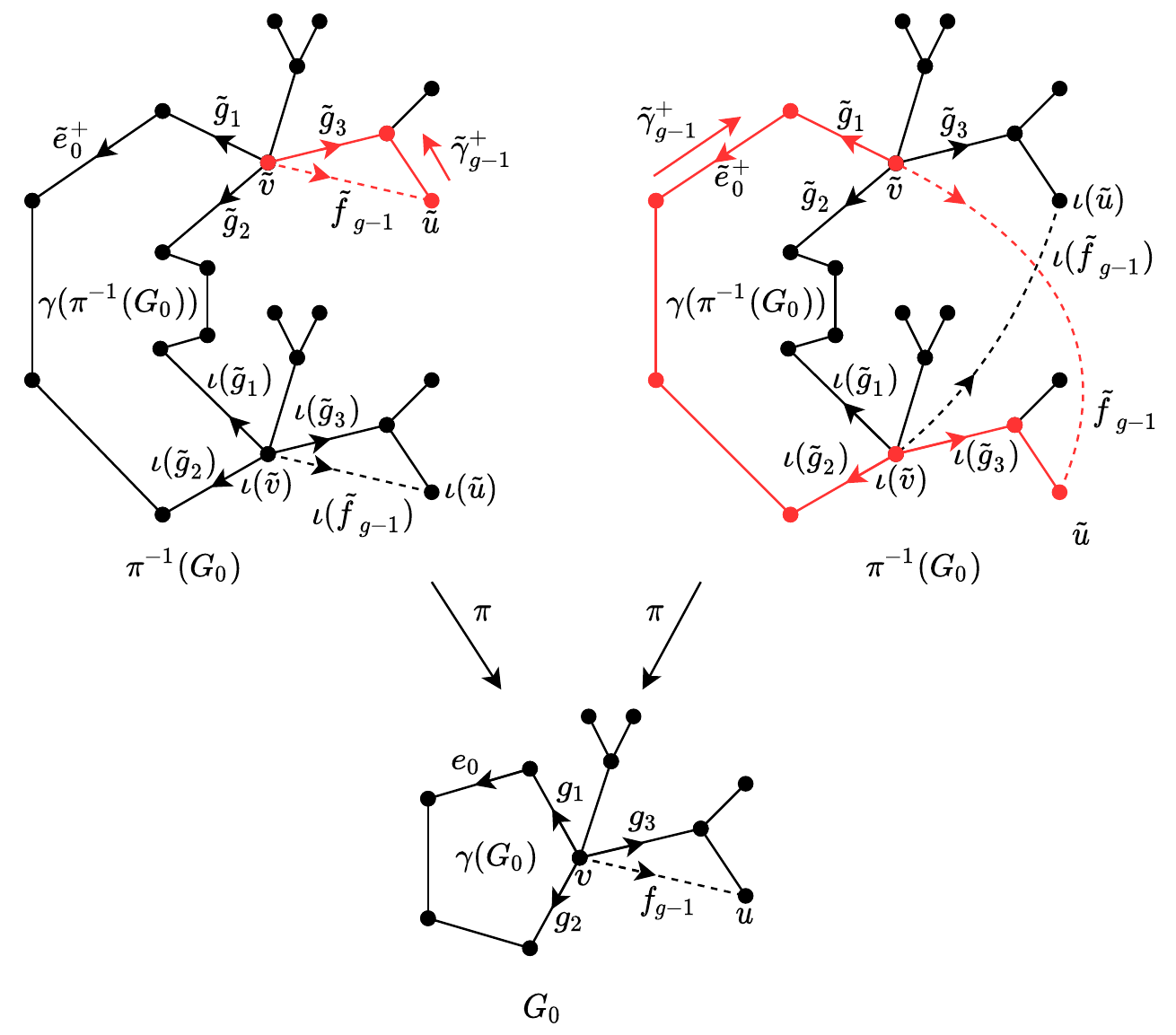}
    \caption{The cycle $\tga^+_{g-1}$ in the two sub-cases of Case (\ref{case:A4})}
    \label{fig:wcA4}
    \end{figure}

    \begin{enumerate}
        \item The vertex $\tu$ lies on $T^+_0$. In this case, any path on the graph $p^{-1}(G_0\cup \{f_{g-1}\})$ starting at $\tv$ and ending at $\iota(\tv)$ passes through the preimage $p^{-1}(\ga(G_0))$ of the unique cycle of $G_0$. Removing either $\{\tg_1,\iota(\tg_1)\}$ or  $\{\tg_2,\iota(\tg_2)\}$ from $p^{-1}(\ga(G_0))$ disconnects the cycle, and therefore the entire preimage graph $p^{-1}(G_0\cup \{f_{g-1}\})$. It follows that $F_1$ and $F_2$ are not odd genus one decompositions. To prove harmonicity, we need to compute $\Psi(\tF_3)_{g-1,g-1}$. The cycle $\tga^+_{g-1}$ starts at $\tv$, proceeds to $\tu$ via $\tf_{g-1}$, and then back to $\tv$ via a path in $T_0^+$ that ends in $-\tg_3$ and does not contain $\iota(\tg_3)$. It follows that
        $$
        \Psi(\tF_3)_{g-1,g-1}=\langle \tga^+_{g-1},\tg_3-\iota(\tg_3)\rangle =-1.
        $$
        Therefore $\Psi$ maps $C(\tF_3)$ to the opposite side $M^-$ as $C(\tF)$, but with the same determinant $2^{r-1}$. Hence $\Psi$ is harmonic.
        
        \item The vertex $\tu$ lies on $T^-_0$. In this case, all three genus one decompositions $F_1$, $F_2$, and $F_3$ are odd. There are two paths from $\tv$ to $\iota(\tv)$ along the cycle $p^{-1}(\ga(G_0))$, starting with edges $\tg_1$ and $\tg_2$. We assume without loss of generality that the path that contains the edge $\te^+_0$ (and hence lies in the spanning tree $\tT$) begins with $\tg_1$. In this case, the path $\tga^+_{g-1}$ begins at $\tv$, moves to $\tu$ via $\tf_{g-1}$, then to $\iota(\tv)$ via a path ending in $-\iota(\tg_3)$, and finally from $\iota(\tv)$ via the path (passing through $\te^+_0$) that starts with $\iota(\tg_2)$ and ends with $-\tg_1$. Hence $\tga^+_{g-1}$ contains $-\tg_1+\iota(\tg_2)-\iota(\tg_3)$ and does not contain the edges $\iota(\tg_1)$, $\tg_2$, or $\tg_3$, and therefore the diagonal entries are
        $$
        \Psi(\tF_1)_{g-1,g-1}=\langle\tga^+_{g-1},\tg_1-\iota(\tg_1)\rangle=-1,\quad
        \Psi(\tF_2)_{g-1,g-1}=\langle\tga^+_{g-1},\tg_2-\iota(\tg_2)\rangle=-1,
        $$
        $$
        \Psi(\tF_3)_{g-1,g-1}=\langle\tga^+_{g-1},\tg_3-\iota(\tg_3)\rangle=1.
        $$
        Therefore, $\Psi$ maps the two cells $C(\tF)$ and $C(\tF_3)$ to the half-space $M^+$ and the two cells $C(\tF_1)$ and $C(\tF_2)$ to the half-space $M^-$, all with the same determinant $2^{r-1}$. Hence $\Psi$ is harmonic.
        
    \end{enumerate}

\end{enumerate}

\medskip\noindent {\bf The endpoints of $f_a$ lie on different connected components of $G\backslash F$.} We assume without loss of generality that $v=s(f_a)$ lies on $G_0$ and that $u=t(f_a)$ lies on $G_1$. Furthermore, we assume that $f_a$ lies in the spanning tree $T^c$, and the ordering convention then implies that $f_a=f_1$ and $\tf_a=\tf_1$. Since the matrix $\Psi(\tF)$ is lower triangular, we see that $M^+=\Psi(H^+)=\{y:y_1\geq 0\}$ and $M^-=\Psi(H^-)=\{y:y_1\leq 0\}$.

Let $\tf'$ be an edge at $\tv$, and let $\tF'=\{\tf',\tf_2,\ldots,\tf_{g-1}\}$ define a cell $C(\tF')$ adjacent to $C(\tF)$ via $C'$. The matrix $\Psi(\tF')$ is obtained from the matrix $\Psi(\tF)$ by replacing the first column, so
we are only interested in the new entry $\Psi(\tF')_{11}=\langle\tga^+_1,\tf'-\iota(\tf')\rangle$ in the top left: if it is zero then $p(\tF')$ is not an odd genus one decomposition, and if it is nonzero then its sign determines whether $\Psi$ maps $C(\tF')$ to $M^+$ or $M^-$. 

The edge $f_1$ is a bridge edge of the graph $G_0\cup G_1\cup \{f_1\}$. We need to consider two possibilities:

\begin{enumerate}
    \item \label{case:B1} The vertex $v=s(f_1)$ does not lie on the unique cycle $\ga(G_0)$ of the graph $G_0$. There is a unique edge $g_1$ at $v$ pointing in the direction of $\ga(G_0)$, and the graph $G_0\cup G_1\cup \{f_1\}\backslash \{g_1\}$ has two connected components of genus one, namely $G_0\backslash \{g_1\}$ and $G_1$. Therefore, $F_1=\{g_1,f_2,\ldots,f_{g-1}\}$ is an odd genus one decomposition of the same length $r$ as $F$. Any other edge $f'$ at $v$ supports a tree rooted at $v$, hence removing $f'$ from $G_0\cup G_1\cup \{f\}$ separates a genus zero connected component, and the corresponding decomposition is not genus one (see Figure~\ref{fig:wcB1}).
    
        \begin{figure}
    \centering
    \includegraphics{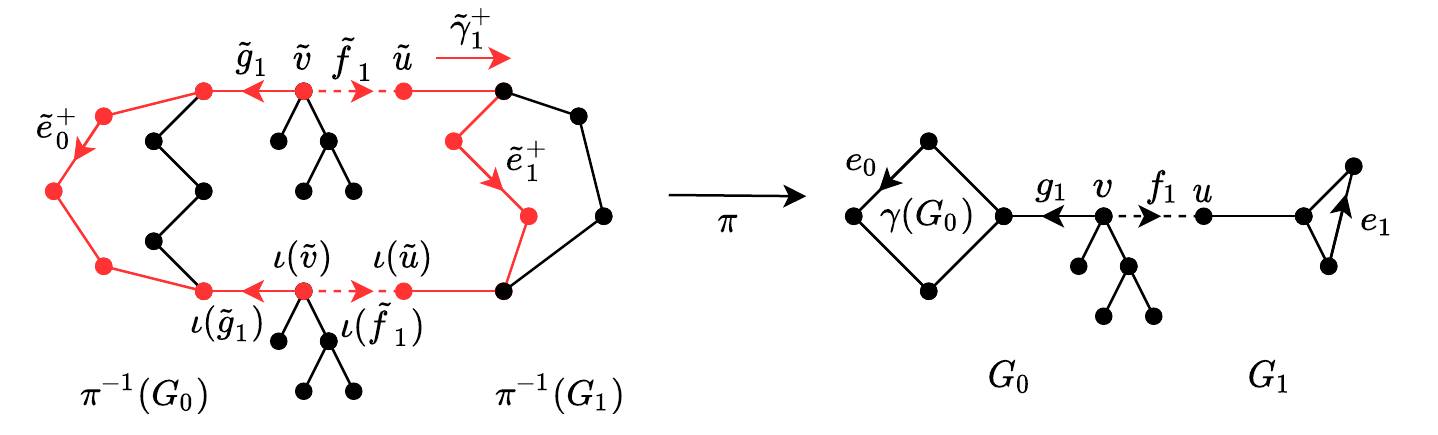}
    \caption{The cycle $\tga^+_1$ in Case (\ref{case:B1})}
    \label{fig:wcB1}
    \end{figure}

    Let $\tg_1$ denote the lift of $g_1$ at $\tv$, and denote $\tF_1=\{\tg_1,\tf_2,\ldots,\tf_{g-1}\}$. To show that $\Psi$ is harmonic, it remains to show that $\Psi$ maps the cell $C(\tF_1)$ to the opposite side $M^-$, in other words we need to show that the diagonal entry $\Psi(\tF_1)_{11}=\langle \tga^+_1,\tg_1-\iota(\tg_1)\rangle$ is negative.
    
    We have chosen an edge $e_1$ lying on the unique cycle $\ga(G_1)$ of $G_1$, and a lift $\te^+_1$ lying on the unique cycle of $p^{-1}(G_1)$, with the property that the path from $\tu=t(\tf)$ to $\iota(\tu)$ that passes through $\te^+_1$ has the same orientation as $\te^+_1$. Hence the path $\tga^+_1$ is constructed as follows: it starts at $\tv$, proceeds via $\tf_1$ to $\tu$, then via the aforementioned path to $\iota(\tu)$, then to $\iota(\tv)$ via $-\iota(\tf_1)$, and then from $\iota(\tv)$ to $\tv$ via the unique path in $p^{-1}(G_0)$ containing the edge $\te^+_0$. This path begins with $\iota(\tg_1)$ and ends with $-\tg_1$, hence
    $$
    \Psi(\tF_1)_{11}=\langle \tga^+_1,\tg_1-\iota(\tg_1)\rangle=-2.
    $$
    Therefore, $\Psi$ maps $C(\tF)$ and $C(\tF_1)$ to different sides of $\tC$ with the same determinant, so $\Psi$ is harmonic.
    
\begin{figure}[h]
    \centering
    \includegraphics{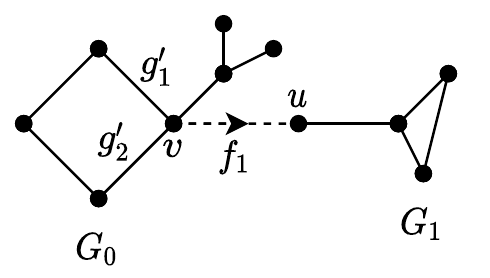}
    \caption{The configuration in Case (\ref{case:B2})}
    \label{fig:wcB2}
    \end{figure}
    \item \label{case:B2} The vertex $v=s(f_1)$ lies on the unique cycle $\ga(G_0)$ of $G_0$. It is easy to see that this case is in fact a relabeling of Case~\ref{case:2b} described above. Indeed, let $g'_1$ and $g'_2$ be the two edges at $v$ lying on $\ga(G_0)$. Replacing $f_1$, $g'_1$, and $g'_2$ with respectively $g_1$, $f_{g-1}$, and $g_2$, we obtain the same picture as in Case~\ref{case:2b} (see Figure~\ref{fig:wcB2}).

\end{enumerate}

This completes the proof of Proposition~\ref{prop:localharmonicity}.
\end{proof}

\begin{figure}[ht]
\begin{tikzpicture}

\node at (-2.5,0) {$\Ga$};
\draw[thick] (-1,0) circle(1);
\draw[fill](0,0) circle(.08);
\draw[thick] (0,0) -- (2,0);
\draw[fill](2,0) circle(.08);
\draw[thick] (3,0) circle(1);
\draw[fill](4,0) circle(.08);
\draw[thick] (4,0) -- (6,0);
\draw[fill](6,0) circle(.08);
\draw[thick] (7,0) circle(1);
\node at (-1.6,0) {$h_1$};
\node at (1,0.4) {$h_3$};
\node at (3,1.4) {$h_4$};
\node at (3,-0.6) {$h_5$};
\node at (8.4,0) {$h_7$};
\node at (5,0.4) {$h_6$};

\draw[thick](1.1,0.1) -- (0.9,0) -- (1.1,-0.1);
\draw[thick](3.1,1.1) -- (2.9,1) -- (3.1,0.9);
\draw[thick](2.9,-0.9) -- (3.1,-1) -- (2.9,-1.1);
\draw[thick](4.9,0.1) -- (5.1,0) -- (4.9,-0.1);
\draw[thick](-2.1,0.1) -- (-2,-0.1) -- (-1.9,0.1);
\draw[thick](7.9,0.1) -- (8,-0.1) -- (8.1,0.1);

\begin{scope}[shift={(0,2.5)}]

\draw[thick](-2.088,1.6) -- (-1.988,1.4) -- (-1.888,1.6);
\draw[thick](-0.1,1.4) -- (0,1.6) -- (0.1,1.4);

\draw[thick](8.088,1.6) -- (7.988,1.4) -- (7.888,1.6);
\draw[thick](5.9,1.4) -- (6,1.6) -- (6.1,1.4);

\draw[thick](2.528,0.972) -- (2.5,0.75) -- (2.694,0.861);

\draw[thick](2.528,2.028) -- (2.5,2.25) -- (2.694,2.139);

\draw[thick](1.1,0.1) -- (0.9,0) -- (1.1,-0.1);
\draw[thick](1.1,3.1) -- (0.9,3) -- (1.1,2.9);
\draw[thick](4.9,0.1) -- (5.1,0) -- (4.9,-0.1);
\draw[thick](4.9,3.1) -- (5.1,3) -- (4.9,2.9);
\draw[thick](2.9,0.1) -- (3.1,0) -- (2.9,-0.1);
\draw[thick](2.9,3.1) -- (3.1,3) -- (2.9,2.9);

\draw[thick] (0,0) -- (6,0);
\draw[thick] (0,3) -- (6,3);
\draw[thick] (0,0) -- (0,3);
\draw[thick] (0,0) .. controls (-2.65,0) and (-2.65,3) .. (0,3);
\draw[thick] (6,0) -- (6,3);
\draw[thick] (6,0) .. controls (8.65,0) and (8.65,3) .. (6,3);
\draw[thick] (2,0) -- (4,3);
\draw[thick] (4,0) -- (2,3);

\node at (-2.5,1.5){$\tGa$};
\node at (-1.6,1.5) {$\thh^+_1$};
\node at (0.4,1.5) {$\thh^-_1$};
\node at (1,0.5) {$\thh^-_3$};
\node at (1,3.5) {$\thh^+_3$};
\node at (3,0.5) {$\thh^-_5$};
\node at (3,3.5) {$\thh^+_5$};

\node at (3.8,1) {$\thh^-_4$};
\node at (3.8,1.9) {$\thh^+_4$};

\node at (5,0.5) {$\thh^-_6$};
\node at (5,3.5) {$\thh^+_6$};
\node at (6.4,1.5) {$\thh^+_7$};
\node at (8.4,1.5) {$\thh^-_7$};

\draw[fill](0,0) circle(.08);
\draw[fill](2,0) circle(.08);
\draw[fill](0,3) circle(.08);
\draw[fill](2,3) circle(.08);
\draw[fill](4,0) circle(.08);
\draw[fill](6,0) circle(.08);
\draw[fill](4,3) circle(.08);
\draw[fill](6,3) circle(.08);

\end{scope}

\end{tikzpicture}

\caption{Free double cover with $g=3$.}
\label{fig:bigexample}

\end{figure}
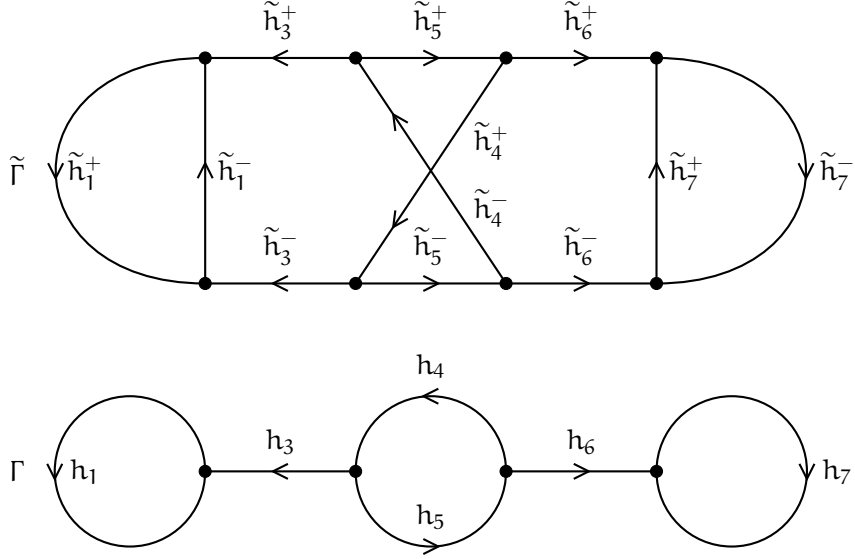

\begin{example} \label{example:big} We now illustrate Thm.~\ref{thm:APharmonic}, the structure of the Abel--Prym map, as well as Constructions~\ref{con:A},~\ref{con:B}, and~\ref{con:C}, using the free double cover $\pi:\tGa\to \Ga$ shown in Fig.~\ref{fig:bigexample}. We note that $g=3$, hence the Abel--Prym map
$$
\Psi:\Sym^2(\tGa)\to \Prym(\tGa/\Ga)
$$
maps to the identity connected component of the Prym variety. The graph $\tGa$ has no loops, so it is sufficient to use the minimal model $p:\tG\to G$, which is obtained by not subdividing any edges.

We begin by using Construction~\ref{con:C} to define a basis $\tga_1,\tga_2$ for $\Ker p_*:H_1(\tG,\ZZ)\to H_1(G,\ZZ)$. This basis defines a coordinate system on $\Prym (\tGa/\Ga)$. We will write down the matrix of $\Psi$ with respect to this coordinate system on each non-contracted cell of $\Sym^2(\tGa)$. First, we need to choose two edges $\tf_1,\tf_2\in E(\tG)$ mapping to distinct edges of $G$. It is convenient to set $\tf_1=\thh^+_3$ and $\tf_2=\thh^-_6$. Removing $f_1=h_3$ and $f_2=h_6$ decomposes $G$ into three connected components, which we denote as follows:
$$
G_0=\{h_4,h_5\}, \quad G_1=\{h_1\}, \quad G_2=\{h_7\}.
$$
For $G_0$ we choose the spanning tree $T_0=\{h_5\}$, while $G_1$ and $G_2$ have trivial spanning trees. The corresponding spanning tree $T$ of $G$ is
$$
T=\{h_3,h_5,h_6\}.
$$
To construct a spanning tree $\tT$ for $\tG$, we join the two lifts $\tT^{\pm}$ of $T$ with one of the lifts of $h_4$. We choose $\te^+_0=\thh^+_4$, so that
$$
\tT=\tT^+\cup\tT^-\cup\{\te^+_0\}=\{\thh^{\pm}_3,\thh^+_4,\thh^{\pm}_5,
\thh^{\pm}_6\}.
$$
In order to agree with the notation of Constructions~\ref{con:A}-\ref{con:C}, we denote $e_0=h_4$, $e_1=h_1$, $e_2=h_7$, $\te^+_1=\thh^+_1$, and $\te^+_2=\thh^+_7$ (recall that we require $t(\tf_k)=s(\te_k)$). 

According to Construction~\ref{con:B}, the cycles $\tga^+_1$ and $\tga^+_2$ are the unique cycles of $\tG$ containing the spanning tree $\tT$ and the edges $\te^+_1=\thh^+_1$ and $\te^+_2=\thh^+_7$, respectively. In other words, 
$$
\tga^+_1=\thh_1^+-\thh^-_3-\thh^+_4-\thh^+_5+\thh^+_3,\quad
\tga^+_2=\thh^+_7-\thh^+_6+\thh^+_4+\thh^-_5+\thh^-_6.
$$
The cycles $\tga^+_1-\iota_*(\tga^+_1)$ and $\tga^+_2-\iota_*(\tga^+_2)$ form a basis for $\Ker p_*$.

We now determine the matrix of the linear map $\Psi$ on each non-contracted cell of $\Sym^2(\tGa)$. There are 15 two-element subsets of $E(G)$, and all of them are ogods except for $\{h_1,h_3\}$ and $\{h_6,h_7\}$. Therefore, there are 52 cells of $\Sym^2(\tGa)$ on which $\Psi$ has maximal rank, corresponding to the 13 ogods. The matrix of $\Psi$ on a cell $\thh^+_i\times \thh^+_j$ is the $2\times 2$ matrix whose columns are obtained by intersecting respectively $\thh^+_i-\thh^-_i$ and $\thh^+_j-\thh^-_j$ with the cycles $\tga^+_1$ and $\tga^+_2$, and the matrix of $\Psi$ on the other three cells $\thh^{\pm}_i\times \thh^{\pm}_j$ is obtained by appropriately changing the signs of the columns.

The following table lists the 13 ogods of the graph $G$. For each ogod $F=\{h_i,h_j\}$, we provide the rank $r(F)$ (the number of connected components of $G\backslash F)$, the local degree $\deg \Psi=2^{r(F)-1}$ on each cell of $\Sym^2(\tGa)$ corresponding to the ogod, and the matrix of $\Psi$ on the cell $\thh^+_i\times \thh^-_i$ (the absolute value of its determinant is $\deg \Psi$). We also assign a color to each ogod.
$$
\begin{tabular}{c|c|c|c|c}
$F$               &  $r(F)$  &  $\deg \Psi|_{\thh^{\pm}_i\times \thh^{\pm}_j}$  & $\Psi|_{\thh^+_i\times \thh^+_j}$ & \mbox{color}\\
\hline $\{h_1,h_4\}$     & 1        & 1  &
$\left[\begin{array}{cc} 1 & -1 \\ 0 & 1 \end{array}\right]$ &
\crule[14]{0.8cm}{0.5cm}\\
\hline $\{h_1,h_5\}$     & 1        & 1  &
$\left[\begin{array}{cc} 1 & -1 \\ 0 & -1 \end{array}\right]$ & 
\crule[15]{0.8cm}{0.5cm}\\
\hline $\{h_1,h_6\}$     & 2        & 2  &
$\left[\begin{array}{cc} 1 & 0 \\ 0 & -2 \end{array}\right]$ &
\crule[16]{0.8cm}{0.5cm}\\
\hline $\{h_1,h_7\}$     & 1        & 1  &
$\left[\begin{array}{cc} 1 & 0 \\ 0 & 1 \end{array}\right]$ &
\crule[17]{0.8cm}{0.5cm}\\
\hline $\{h_3,h_4\}$     & 2        & 2  &
$\left[\begin{array}{cc} 2 & -1 \\ 0 & 1 \end{array}\right]$ &
\crule[34]{0.8cm}{0.5cm}\\
\hline $\{h_3,h_5\}$     & 2        & 2  &
$\left[\begin{array}{cc} 2 & -1 \\ 0 & -1 \end{array}\right]$ &
\crule[35]{0.8cm}{0.5cm}\\
\hline $\{h_3,h_6\}$     & 3        & 4  &
$\left[\begin{array}{cc} 2 & 0 \\ 0 & -2 \end{array}\right]$ &
\crule[36]{0.8cm}{0.5cm}\\
\hline $\{h_3,h_7\}$     & 2        & 2  &
$\left[\begin{array}{cc} 2 & 0 \\ 0 & 1 \end{array}\right]$ &
\crule[37]{0.8cm}{0.5cm}\\
\hline $\{h_4,h_5\}$     & 2        & 2  &
$\left[\begin{array}{cc} -1 & -1 \\ 1 & -1 \end{array}\right]$ &
\crule[45]{0.8cm}{0.5cm}\\
\hline $\{h_4,h_6\}$     & 2        & 2  &
$\left[\begin{array}{cc} -1 & 0 \\ 1 & -2 \end{array}\right]$ &
\crule[46]{0.8cm}{0.5cm}\\
\hline $\{h_4,h_7\}$     & 1        & 1  &
$\left[\begin{array}{cc} -1 & 0 \\ 1 & 1 \end{array}\right]$ &
\crule[47]{0.8cm}{0.5cm}\\
\hline $\{h_5,h_6\}$     & 2        & 2  &
$\left[\begin{array}{cc} -1 & 0 \\ -1 & -2 \end{array}\right]$ &
\crule[56]{0.8cm}{0.5cm}\\
\hline $\{h_5,h_7\}$     & 1        & 1  &
$\left[\begin{array}{cc} -1 & 0 \\ -1 & 1 \end{array}\right]$ &
\crule[57]{0.8cm}{0.5cm}\\
\end{tabular}
$$

We now explain how to construct the Abel--Prym map $\Psi:\Sym^2(\tGa)\to \Prym(\tGa/\Ga)$. We first choose a starting cell; it is convenient to begin with the cell $\tf_1\times \tf_2=\thh^+_3\times \thh^-_6$. According to our calculations, the matrix of $\Psi$ on this cell is
$$
\Psi|_{\thh^+_3\times \thh^-_6}=\left[\begin{array}{cc} 2 & 0 \\ 0 & 2 \end{array}\right];
$$
note that this matrix is lower-triangular with $r(\{h_3,h_6)\})-1=2$ diagonal entries equal to $2$, in accordance with Cor.~\ref{cor:degreeformula}. Viewing $\thh^+_3\times \thh^-_6$ as the rectangle $[0,\ell(h_3)]\times [0,\ell(h_6)]$ in $\RR^2$, we see that $\Psi(\thh^+_3\times \thh^-_6)$ is a rectangle with sides $2\ell(h_3)$ and $2\ell(h_6)$ in $\Prym(\tGa/\Ga)$ (with respect to the coordinate system defined by $\tga_1$ and $\tga_2$). We note that since $\deg\Psi=4$ on the cell $\thh^+_3\times \thh^-_6$, the Abel--Prym map is one-to-one over the interior of the cell; in other words, there are no other cells of $\Sym^2(\tGa)$ lying over $\Psi(\thh^+_3\times \thh^-_6)$.

We now look at the graph to locate the cells in $\Sym^2(\tGa)$ that are adjacent to $\thh^+_3\times \thh^-_6$, by changing one of the two edges to an adjacent edge. We then use the table to determine their images in $\Prym(\tGa/\Ga)$. For example, the terminal vertex of $\thh^+_3$ is the starting vertex of $\thh^+_1$ and the terminal vertex of $\thh^-_1$. Therefore, the images of the two cells $\thh^+_1\times \thh^-_6$ and $\thh^-_1\times \thh^-_6$ are glued to the right side of the rectangle $\Psi(\thh^+_3\times \thh^-_6)$. According to the table, each of the images $\Psi(\thh^+_1\times \thh^-_6)$ and $\Psi(\thh^+_1\times \thh^-_6)$ is a rectangle with sides $\ell(h_1)$ and $2\ell(h_6)$; this implies that these images coincide. Similarly, along the top edge of the rectangle $\Psi(\thh^+_3\times \thh^-_6)$ we glue two identical rectangles $\Psi(\thh^+_3\times \thh^+_7)$ and $\Psi(\thh^+_3\times \thh^-_7)$. On the other hand, the two cells that are glued along the left side of the rectangle $\Psi(\thh^+_3\times \thh^-_6)$ are the non-congruent parallelograms $\Psi(\thh^+_5\times \thh^-_6)$ and $\Psi(\thh^-_4\times \thh^-_6)$, and similarly there are two non-congruent parallelograpms along the bottom side of $\Psi(\thh^+_3\times \thh^-_6)$.

We then proceed in the same way for all non-contracted cells. For example, the right sides of the two identical rectangles $\Psi(\thh^+_3\times \thh^+_7)$ and $\Psi(\thh^+_3\times \thh^-_7)$ are attached to the left side of the rectangle $\Psi(\thh^-_3\times \thh^-_6)$, and so on. The result is a collection of overlapping parallelograms $\Psi(\thh^{\pm}_i\times \thh^{\pm}_j)$ tiling the torus $\Prym(\tGa/\Ga)$. Each parallelogram tile comes equipped with the degree of $\Psi$ (equal to 1, 2, or 4), and the total degree of all tiles above each point is equal to 4, so any point not lying on the boundary of a tile lies in anywhere between one and four tiles. The structure of the overlaps strongly depends on the exact values of the edge lengths $\ell(h_i)$.

The result, for a specific choice of edge lengths, is shown on Fig.~\ref{fig:blender}. The top part is an exploded view of the non-contracted cells of $\Sym^2(\tGa)$. Each cell $\thh^{\pm}_i\times \thh^{\pm}_j$ is represented by a parallelogram and is colored according to the corresponding ogod $\{h_i,h_j\}$ (see table). The cells are arranged in several horizontal layers and are oriented and aligned according to their images in $\Prym(\tGa/\Ga)$. For example, the four large dark grey cells in the middle layer represent the four cells $\thh^{\pm}_3\times \thh^{\pm}_6$, which are the only cells of degree four. No attempt has been made to show the incidence relations between the cells in $\Sym^2(\tGa)$, since this complex is not embeddable in $\RR^3$, and the numerous cells that are contracted by $\Psi$ are not shown.

The multicolored parallelogram on the bottom layer is $\Prym(\tGa/\Ga)$ with the induced cellular decomposition, obtained by intersecting all the tiles. A cell corresponding to a single parallelogram has the same color,  while if parallelograms of different colors intersect, then we blend the corresponding colors. The central point of the parallelogram is the identity element of $\Prym(\tGa/\Ga)$.

\begin{figure}
    \centering
    \includegraphics[width=6in]{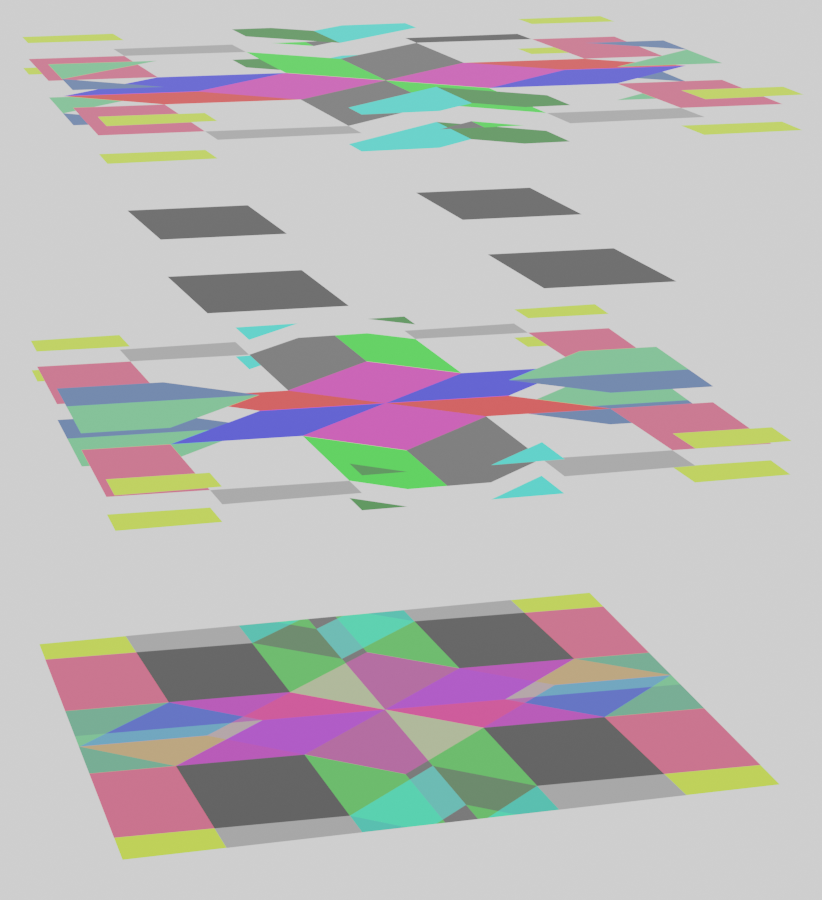}
    \caption{The structure of the Abel--Prym map $\Psi:\Sym^2(\tGa)\to \Prym(\tGa/\Ga)$ of the cover shown in Fig.~\ref{fig:bigexample}. The tesselated parallelogram on the bottom is $\Prym(\tGa/\Ga)$ with the cell decomposition induced by $\Psi$. The top and middle parts are an exploded view of the non-contracted cells of $\Sym^2(\tGa)$. Cells of $\Sym^2(\tGa)$ are colored according to ogod type, these colors are mixed in cells of $\Prym(\tGa/\Ga)$. Edge lengths are $\ell(h_1)=2.4$, $\ell(h_3)=0.8$, $\ell(h_4)=1$, $\ell(h_5)=1.4$, $\ell(h_6)=1.1$, $\ell(h_7)=1.4$.}
    \label{fig:blender}
\end{figure}

\end{example}


\appendix
\section{The algebraic Abel--Prym map (by Sebastian Casalaina-Martin)}

Let $\pi:\widetilde C\to C$ be a connected \'etale double cover of a smooth projective curve $C$ of genus $g\ge 2$ over an algebraically closed field $k$ of characteristic not equal to $2$, let $\iota:\widetilde C\to \widetilde C$ be the associated involution, and 
denote by $\Nm:J(\widetilde C)\to J(C)$ the norm map for $\pi$, where for a smooth projective curve $X$ over $k$ we denote by $J(X)=\operatorname{Pic}^0_{X/k}$ the Jacobian of $X$.  For any natural number $d$ the \emph{Abel--Prym map} in degree $d$ is defined to be the map 
$$
\delta_d: \widetilde C^{(d)}\longrightarrow \ker \Nm\subseteq J(\widetilde C),
$$
$$
\widetilde D\mapsto \mathcal O_{\widetilde C}(\widetilde D-\iota \widetilde D),
$$
where $\widetilde C^{(d)}$ is the $d$-fold symmetric product of the curve.  
The kernel of the norm map has two connected components, namely  the Prym variety $P=P(\widetilde C/C):=(\ker \Nm)^\circ$, the connected component of the identity, and the remaining component, which we will denote by $P'$; $P$ admits a principal polarization $\Xi$ with the property that if $\Theta_{\widetilde C}$ is the canonical principal polarization on $J(\widetilde C)$, then $\Theta_{\widetilde C}|_P = 2\cdot \Xi$ 
 (e.g., \cite[\S 6]{Mumford_Prym}).  
 The image of $\delta_d$ is contained in $P$ if $d$ is even and contained in $P'$ if $d$ is odd (e.g., \cite[Lem.~3.3, and p.159]{B77_schottky}).   

The Abel--Prym map in degree-$1$ has been studied quite extensively, and we recall this case in \S \ref{S:App-d=1}.  In particular, the map $\delta_1$ is a closed embedding if and only if $\widetilde C$ is not hyperelliptic, and has degree $2$ otherwise.    
The purpose of this appendix is to provide a proof of some basic facts regarding the Abel--Prym map for $d>1$.  We expect these are well known, but are not aware of a  reference in the literature.

\begin{proposition}[Corollary \ref{C:gen-fin}, Corollary \ref{C:deg}, and Proposition \ref{P:hype}]\label{P:App-main}
The Abel--Prym map $\delta_d$ is generically finite if and only if $d\le g-1$, and surjects onto $P$ (resp.~$P'$) if and only if  $d\ge  g-1$ and $d$ is even (resp.~$d$ is odd). 
  Moreover, $\deg \delta_{g-1}=2^{g-1}$, and if $\operatorname{char}(k)=0$, then for $d\le g-2$ we have  $\deg \delta_d=2^n\le 2^d$ for some integer $n\le d$, with equality holding if $\widetilde C$ is hyperelliptic.
\end{proposition}

\begin{remark}[Degree bound in positive characteristic]
If $\operatorname{char}(k)=p>0$,   then  for $d\le g-2$ we show that   $\deg \delta_d=p^{m}2^n$ for some integers $m$ and $n$ with $n\le d$.  The reason for the uncontrolled power of $p$ in the formula is that we compute the degree via a cohomology class computation in $\ell$-adic cohomology, with $\ell\ne p$.   A similar computation in crystalline cohomology allows one to remove the powers of $p$.  
\end{remark}

For general covers one has:
 
\begin{proposition}[Corollary \ref{C:GenAP}]\label{P:GenAP}
Let $\widetilde C/C$ be a general cover.  Then $\deg \delta_d=1$ for   $d<g/2$.
\end{proposition}

 While $\delta_1$ is finite, in contrast, 
 for all $d\ge 2$ there exist  positive dimensional fibers of $\delta_d$.  
  Indeed, it suffices to show this for $d$ even since for all $d$,  fibers of $\delta_{d}$ can be included in fibers of $\delta_{d+1}$ using the observation that if  $\delta_{d}(\widetilde D) = \delta_{d}(\widetilde D')$, then $\delta_{d+1}(\widetilde D+\tilde p) = \delta_{d+1}(\widetilde D' + \tilde p)$ for any $\tilde p\in \widetilde C$.    For $d$ even, we note that the composition $C^{(d/2)}\stackrel{\pi^*}{\to} \widetilde C^{(d)}\stackrel{\delta_{d}}{\to} P\subseteq J(\widetilde C)$ has image $\mathcal O_{\widetilde C}$.

\medskip 
Sometimes in the presentation it will be  convenient to fix a divisor $\widetilde D_0\in \widetilde C^{(d)}$, and then consider the associated \emph{pointed Abel--Prym map}
$$
\delta_{d,\widetilde D_0}: \widetilde C^{(d)}\longrightarrow P \subseteq J(\widetilde C)
$$
$$
D\mapsto \mathcal O_{\widetilde C}(\widetilde D-\iota \widetilde D)\otimes \mathcal O_{\widetilde C}(\iota \widetilde D_0-\widetilde D_0)
$$
which simply differs from the canonical Abel--Prym map $\delta_d$  by translation by $\mathcal O_{\widetilde C}(\iota \widetilde D_0-\widetilde D_0)$, and has image contained in the Prym variety.  We emphasize for clarity that the Abel--Prym map and pointed Abel--Prym map defined here are different from the map obtained by restricting the Abel map in degree $d$  to a chosen translate of the Prym variety (e.g., \cite{B82_sous,SV01}).

\subsection*{Acknowledgements} The author would like to thank Angela Ortega and Roy Smith  for useful comments that pointed to  an error in an earlier draft regarding fibers of the Abel--Prym map.  He would also like to thank Gavril Farkas for pointing out an error in an earlier draft regarding Brill--Noether theory  for  connected \'etale double covers, and directing the author to \cite{AF, schwarz} where the correct statements are proven, providing the needed result to complete the proof of  Proposition \ref{P:GenAP}.

\subsection{The Abel--Prym map in degree-$1$}\label{S:App-d=1}
We recall the following well-known result:

\begin{proposition}\label{P:d1}
For any prime number $\ell\ne \operatorname{char}(k)$ and any point $\tilde p_0\in \widetilde C$, the class of the push forward of $\widetilde C$ by the pointed Abel--Prym map is  
\begin{equation}\label{E:Pd1}
(\delta_{1,\tilde p_0})_*[\widetilde C]=2\cdot \frac{[\Xi]^{\rho-1}}{(\rho-1)!} \in H^{2\rho -2}(P,\mathbb Z_\ell(\rho-1)),
\end{equation}
where $\rho=\dim P=g-1$.  
In addition, if $\widetilde C$ is not hyperelliptic, then the Abel--Prym map $\delta_1$ is an embedding, so that $[\delta_{1,\tilde p_0}(\widetilde C)]= 2\cdot \frac{[\Xi]^{\rho-1}}{(\rho-1)!} $.  
If $\widetilde C$ is hyperelliptic, then $\delta_1$ has degree $2$, and the image $\Sigma:=\delta_1(\widetilde C)\subseteq P'$ is a smooth hyperelliptic curve of genus $g-1$ so that setting $\Sigma_{\tilde p_0}:=\delta_{1,\tilde p_0}(\widetilde C)\subseteq P$, we have  $[\Sigma_{\tilde p_0}]= \frac{[\Xi]^{\rho-1}}{(\rho-1)!} $, and $(P,\Xi)$ is isomorphic to the principally polarized Jacobian $(J(\Sigma),\Theta_\Sigma)$. 
\end{proposition}

\begin{proof}
Computing the degree of $\delta_1$ is a basic computation from the definition; the details can be found in \cite[Prop.~12.5.2]{BL_CAV} where the arguments are made over $\mathbb C$, but which  hold over any algebraically closed field of characteristic not equal to $2$. 
The statements regarding the differential of $\delta_1$  can be found in \cite[Cor.~12.5.5--7]{BL_CAV} over $\mathbb C$; those arguments also hold over any algebraically closed field of characteristic not equal to $2$, and we also review this in \S \ref{S:GeomDiff}.  
  As the map $\delta_1$ is finite, computing the class of $(\delta_{1,\tilde p_0})_*[\widetilde C]$ is a standard argument using the fact that $H^{2\rho - 2}(P,\mathbb Z_\ell)= \bigwedge ^{2\rho -2}H^1(P,\mathbb Z_\ell)$, and facts about first Chern classes of symmetric polarizations on abelian varieties.  This is essentially the same argument that is used to prove Poincar\'e's formula in \cite[p.25]{ACGHI}, and the arguments there are easily adapted to the Abel--Prym map, and the positive characteristic case (see \cite[Lem.~3.2]{Mas76} over $\mathbb C$).  
  
  The fact that $\Sigma$ is smooth and that $(P,\Xi)$ is isomorphic to the principally polarized Jacobian $(J(\Sigma),\Theta_\Sigma)$  follows from the fact that  $[\Sigma_{\tilde p_0}]= \frac{[\Xi]^{\rho-1}}{(\rho-1)!} $, and the criterion of Matsusaka--Ran \cite{C84}. One can conclude that $\Sigma$ is hyperelliptic by considering the image of the  $g^1_2$ on $\widetilde C$ under the norm map for $\widetilde C\to \Sigma$.  
\end{proof}

\begin{remark}  If $\widetilde C$ is hyperellptic, then $C$ is hyperelliptic, as well (consider the image of the  $g^1_2$ on $\widetilde C$ under the norm map for $\pi$). 
Thus the conclusion in Proposition \ref{P:d1} that if $\widetilde C$ is hyperelliptic then $(P,\Xi)$ is a hyperelliptic Jacobian is a special case of a result of Mumford, which states that for any $\pi:\widetilde C\to C$ with $C$ hyperelliptic, the Prym variety is a product of hyperelliptic Jacobians \cite[p.346]{Mumford_Prym}.
\end{remark}

\begin{remark}\label{R:RamHE}  In the case where $\widetilde C$ is hyperelliptic, we can say more.  
From the definition, we have that $\delta_1(\tilde p)=\delta_1(\tilde p')$ for distinct points $\tilde p,\tilde p'\in \widetilde C$ if and only if $\tilde p+\iota(\tilde p')$ is in the (unique) $g^1_2$ on $\widetilde C$.  
By Riemann--Hurwitz, $\delta_1$ is ramified at four points.  
The ramification points $\tilde r \in \widetilde C$ of $\delta_1$ are distinguished by the fact that the $2$-torsion line bundle $\eta$ determining the cover $\pi$ satisfies  $\eta=\mathcal O_C(\pi(\tilde r)-p)$ for some $p\in C$.  Note that this forces $\pi(\tilde r)$ and $p$   to be  branch points for the hyperelliptic involution on $C$.  
In summary, if $\widetilde C$ is hyperelliptic, then $\eta\cong \mathcal O_C(p'-p)$ for some distinct ramification points $p,p'\in C$ for the $g^1_2$ on $C$, and the four ramification  points of $\delta_1$ are the points  $\pi^{-1}(\{p,p'\})$.  
The details can be found in \cite[\S 12.5]{BL_CAV}, where again the arguments hold in positive characterstic, as well; see also \S \ref{S:GeomDiff}.  
\end{remark}

\subsection{The differential of the Abel--Prym map}
We next show that the Abel--Prym map is generically finite for $d\le g-1$ by showing that the differential is generically injective in that range.

\begin{proposition}\label{P:DiffAP}
 For $d\le g-1$, 
the differential of $\delta_d$ is generically injective.  
For all $d$ the differential generically has rank equal to $\min (d,g-1)$.  
 The differential of $\delta_d$ is injective if and only if $d=1$ and $\widetilde C$ is not hyperelliptic. 
 
  More precisely, the differential $\delta_d$ is injective at $\widetilde D\in \widetilde C^{(d)}$ if and only if the support of $\widetilde D$ satisfies $\operatorname{Supp}(\widetilde D)\cap \operatorname{Supp}(\iota \widetilde D)=\emptyset$ and, setting $D=\Nm\widetilde D$, there does not exist an effective divisor $E\in C^{(d)}$ such that $\eta =\mathcal O_C(D-E)$, where $\eta$ is the $2$-torsion line bundle determining the cover $\widetilde C/C$.  
\end{proposition}

\begin{proof}
 The case $d=1$ is handled in Proposition \ref{P:d1}.
 Since there exist positive dimensional fibers of $\delta_d$ for all $d\ge 2$, the differential cannot be injective for $d\ge 2$.   For the statements on the generic rank, once we have shown that the differential of $\delta_{g-1}$ is generically injective, by a dimension count it is generically surjective, and one can easily check that this implies that $\delta_d$ surjects onto $P$ or $P'$ for all $d\ge g-1$.  Thus, for the generic rank of the differential of $\delta_d$,  it only remains to show that for $d\le g-1$, 
the differential of $\delta_d$ is generically injective.

For this, we  factor $\delta_d$ as follows:
$$
\xymatrix@C=3em{
\widetilde C^{(d)}\ar[r]^<>(0.5){1\times \iota}&\widetilde C^{(d)} \times \widetilde C^{(d)}\ar[r]^<>(0.5){{\alpha_d}\times {\alpha_d}}& \operatorname{Pic}^{d}_{\widetilde C/k}\times \operatorname{Pic}^{d}_{\widetilde C/k} \ar[r]^<>(0.5){-}& \operatorname{Pic}^0_{\widetilde C/k}
}
$$
where ${\alpha_d}:\widetilde C^{(d)}\to \operatorname{Pic}^d_{\widetilde C/k}$ is the Abel map $\widetilde D\mapsto \mathcal O_{\widetilde C}(\widetilde D)$.  Then at the  level of the differential, at a divisor $\widetilde D\in\widetilde C^{(d)}$, these are given by the maps
$$
\xymatrix@C=2.55em{
H^0(\mathcal O_{\widetilde D}(\widetilde D)) \ar[r]^<>(0.5){1\times \iota}& H^0(\mathcal O_{\widetilde D}(\widetilde D)) \times H^0(\mathcal O_{\iota \widetilde D}(\iota \widetilde D)) \ar[r]^<>(0.5){T{\alpha_d}\times T{\alpha_d}}& H^1(\mathcal O_{\widetilde C})\times H^1(\mathcal O_{\widetilde C}) \ar[r]^<>(0.5){-}& H^1(\mathcal O_{\widetilde C}).
}
$$
The composition of the maps above agrees  with the composition of the maps 
\begin{equation}\label{E:Diff2}
\xymatrix@C=2.55em{
H^0(\mathcal O_{\widetilde D}(\widetilde D)) \ar[r]^<>(0.5){1\times (-\iota)}& H^0(\mathcal O_{\widetilde D}(\widetilde D)) \times H^0(\mathcal O_{\iota \widetilde D}(\iota \widetilde D)) \ar[r]^<>(0.5){T{\alpha_d}\times T{\alpha_d}}& H^1(\mathcal O_{\widetilde C})\times H^1(\mathcal O_{\widetilde C}) \ar[r]^<>(0.5){+}& H^1(\mathcal O_{\widetilde C})
}
\end{equation}
which in turn factors as
\begin{equation}\label{E:Diff3}
\xymatrix@C=2.2em{
H^0(\mathcal O_{\widetilde D}(\widetilde D)) \ar[r]^<>(0.5){1\times (-\iota)}\ar@{=}[d]& H^0(\mathcal O_{\widetilde D}(\widetilde D)) \times H^0(\mathcal O_{\iota \widetilde D}(\iota \widetilde D)) \ar[r]^<>(0.5){T{\alpha_d}\times T{\alpha_d}} \ar[d]^+& H^1(\mathcal O_{\widetilde C})\times H^1(\mathcal O_{\widetilde C}) \ar[r]^<>(0.5){+}& H^1(\mathcal O_{\widetilde C}) \ar@{=}[d]\\
H^0(\mathcal O_{\widetilde D}(\widetilde D)) \ar[r]^<>(0.5){ \beta_d}& H^0( \mathcal O_{\widetilde D+\iota \widetilde D}(\widetilde D+\iota \widetilde D)) \ar[rr]^<>(0.5){T{\alpha_{2d}} }& & H^1(\mathcal O_{\widetilde C})\\
}
\end{equation}
where we denote by $+:H^0(\mathcal O_{\widetilde D}(\widetilde D)) \times H^0(\mathcal O_{\iota \widetilde D}(\iota \widetilde D)) \to H^0( \mathcal O_{\widetilde D+\iota \widetilde D}(\widetilde D+\iota \widetilde D))$ the differential of the map $+:\widetilde C^{(d)}\times \widetilde C^{(d)}\to \widetilde C^{(2d)}$ at the point $(\widetilde D,\iota \widetilde D)$, we set  $\beta_d:= +\circ (1\times (-\iota))$, and $T\alpha_{2d}$ is the differential of the 
 Abel map $\alpha_{2d}: \widetilde C^{(2d)}\to \operatorname{Pic}^{2d}_{\widetilde C/k}$ at the point $\widetilde D+\iota \widetilde D$.
Recall that $T\alpha_{2d}$ is  identified with the 
coboundary map 
\begin{equation}\label{E:DiffCo}
\xymatrix{
\partial_{\widetilde D+\iota \widetilde D}:H^0( \mathcal O_{\widetilde D+\iota \widetilde D}(\widetilde D+\iota \widetilde D))\ar[r] & H^1( \mathcal O_{\widetilde C})
}
\end{equation}
from the long exact sequence associated to the short exact sequence
\begin{equation}\label{E:SEStD}
0\to \mathcal O_{\widetilde C}\to \mathcal O_{\widetilde C}(\widetilde D+\iota \widetilde D)\to \mathcal O_{\widetilde D+\iota \widetilde D}(\widetilde D+\iota \widetilde D)\to 0.
\end{equation}

Returning now to showing that the differential of $\delta_d$ is generically injective, 
note that $\beta_d$ in \eqref{E:Diff3} is injective if and only if $\operatorname{Supp}(\widetilde D)\cap \operatorname{Supp}(\iota \widetilde D)=\emptyset$.  
  Thus under the assumption $\operatorname{Supp}(\widetilde D)\cap \operatorname{Supp}(\iota \widetilde D)=\emptyset$,  it suffices to show that the coboundary map $\partial_{\widetilde D+\iota \widetilde D}$ in \eqref{E:DiffCo} is injective for general $\widetilde D$ of degree $d\le g-1$.  
From the long exact sequence in cohomology, it is enough to show that $h^0(\widetilde C,\mathcal O_{\widetilde C}(\widetilde D+\iota \widetilde D))=1$. 
For this, set $D=\Nm(\widetilde D)$ so that we have  $\widetilde D+\iota \widetilde D=\pi^*D$.  Then using that $\pi_*\mathcal O_{\widetilde C}=\mathcal O_C\oplus \eta$, where $\eta^{\otimes 2}\cong \mathcal O_C$ is the $2$-torsion line bundle defining the cover, and the projection formula, 
we have $h^0(\widetilde C,\mathcal O_{\widetilde C}(\widetilde D+\iota \widetilde D))= h^0(C,\mathcal O_C(D))+h^0(C,\eta(D))$.  
Thus it suffices to show that 
 $h^0(C,\mathcal O_C(D))=1$ and $ h^0(C,\eta(D))=0$,  if $D$ is general of degree $d\le g-1$ (since if $\widetilde D$ is generic, then $D$ will be, too).  By Riemann--Roch, this is equivalent to showing that $h^0(C,K_C(-D))=g-d$, and $h^0(C,K_C(-D)\otimes \eta)=(g-1)-d$; if $D$ is general  of degree $d\le g-1$, then these conditions are satisfied.  This completes the proof regarding generic injectivity of the differential.

 For the final statement of the proposition, still under the assumption  $\operatorname{Supp}(\widetilde D)\cap \operatorname{Supp}(\iota \widetilde D)=\emptyset$, we observe that the coboundary map \eqref{E:DiffCo} is identified with the direct sum of the coboundary maps
$$
\xymatrix@C=3em{
H^0( \mathcal O_C(D)|_D)\oplus H^0(\eta(D))|_D)\ar[r]^<>(0.5){\partial_D\oplus \partial_{\eta(D)}} & H^1(\mathcal O_C)\oplus H^1(\eta)
}
$$
associated to the short exact sequence $0\to \mathcal O_C\to \mathcal O_C(D)\to \mathcal O_C(D)|_D\to 0$, and the short exact sequence obtained by tensoring with $\eta$.   For this, apply $\pi_*$ to \eqref{E:SEStD}, and use $R^1\pi_*\mathcal O_{\widetilde C}=0$ and the projection formula to obtain the direct sum of the short exact sequences together with the identification  $\pi_*\mathcal O_{\widetilde D+\iota \widetilde D}(\widetilde D+\iota \widetilde D)= \mathcal O_C(D)|_D\oplus \eta(D)|_D$. 
Since the image of $\beta_d$ in \eqref{E:Diff3} is equal to the anti-invariant part of the target, which is in turn identified with $H^0(\eta(D))|_D)$, one is reduced to checking the injectivity of $\partial_{\eta(D)}$.
  From the long exact sequence in cohomology, this coboundary map is injective if and only if  $0=h^0(C,\eta)= h^0(C,\eta(D))$. 
 This fails if and only if $\eta\cong \mathcal O_C(D-E)$ for some effective divisor $E\in C^{(d)}$, completing the proof.
\end{proof}

\begin{remark}
It is elementary to see from the definition of the Abel--Prym map, or from \eqref{E:Diff2}, that if $h^0(\widetilde C,\mathcal O_{\widetilde C}(\widetilde D))\ge 2$, then the differential of $\delta_d$ fails to be injective at $\widetilde D$.  Therefore Proposition \ref{P:DiffAP} implies that if $\operatorname{Supp}(\widetilde D)\cap \operatorname{Supp}(\iota \widetilde D)=\emptyset$ and  $h^0(\widetilde C,\mathcal O_{\widetilde C}(\widetilde D))\ge 2$, then  $h^0(C,\eta(D))>0$. Here we give an elementary proof of this fact.  
 Indeed, let $\tilde \sigma \in  H^0(\widetilde C,\mathcal O_{\widetilde C}(\widetilde D))$ be such that its divisor of zeros is $(\tilde \sigma)_0=\widetilde D$.  Then for general $\tilde \sigma' \in H^0(\widetilde C,\mathcal O_{\widetilde C}(\widetilde D))$, with divisor of zeros $(\tilde \sigma')_0=\widetilde D'$,  we have that the section  $\tilde \sigma \cdot \iota \tilde \sigma ' \in H^0(\widetilde C,\mathcal O_{\widetilde C}(\widetilde D+\iota \widetilde D))$, with divisor of zeros $(\tilde \sigma \cdot \iota \tilde \sigma ' )_0=\widetilde D+\iota \widetilde D'$, has support that is not equal to the support of $\iota \widetilde D+\widetilde D'$; here we are using that $\operatorname{Supp}(\widetilde D)\cap \operatorname{Supp}(\iota \widetilde D)=\emptyset$.      
 Therefore,  $\tilde \sigma\cdot \iota \tilde  \sigma'\ne \tilde \sigma' \cdot \iota \tilde \sigma$, since their associated divisors have different supports.  Consequently, $\tilde  \sigma \cdot \iota \tilde \sigma'- \tilde \sigma' \cdot \iota \tilde \sigma\ne 0$ gives a non-trivial element of $H^0(\widetilde C,\mathcal O_{\widetilde C}(\widetilde D+\iota \widetilde D))^-= H^0(C,\eta (D))$.
\end{remark}

\begin{corollary}\label{C:gen-fin}
The Abel--Prym map $\delta_d$ is generically finite if and only if $d\le g-1$, and surjects onto $P$ (resp.~$P'$) if and only if  $d\ge  g-1$ and $d$ is even (resp.~$d$ is odd).   
 \qed
\end{corollary}

\begin{corollary}\label{C:genFib} Suppose that $\eta$ is not in the image of the difference map $C^{(d)}\times C^{(d)}\to J(C)$, which requires $d<g/2$, and holds if in addition 
$\widetilde C/C$ is general.   Then the exceptional locus of $\delta_d$ is exactly the locus of $\widetilde D\in \widetilde C^{(d)}$ such that $\operatorname{Supp}(\widetilde D)\cap \operatorname{Supp}(\iota \widetilde D)\ne \emptyset$.  
\end{corollary}

\begin{proof}
First, we explain that if 
$\widetilde C/C$ is general and $d<g/2$, then 
$\eta$ is not in the image of the difference map $C^{(d)}\times C^{(d)}\to J(C)$.  
Indeed, if 
 $\eta = \mathcal O_C(D-E)$ for some effective divisors $D$ and $E$ of degree $d$ on $C$, then since $\pi^*\eta\cong \mathcal O_{\widetilde C}$ (e.g., \cite[Lem.~p.332]{Mumford_Prym}), we would have $\pi^*D \sim \pi^*E$, so that in this case $\widetilde C$ would have a $g^1_{2d}$, which is not possible due to 
 \cite[Thm.~1.4]{AF}, which states that for a general cover $\widetilde C/C$, there is no $g^1_e$ on $\widetilde C$ for $e<g$ (in positive characteristic
 combine the proof of \cite[Thm.~1.4]{AF} given in \cite[Prop.~3.1]{AF} with \cite[Thm.~1.1]{O14}).

Now assuming that $\eta$ is not in the image of the difference map $C^{(d)}\times C^{(d)}\to J(C)$, which implies that $d<g/2\le g-1$ (e.g., \cite[Ex.~D-1, p.223]{ACGHI}; the argument holds in positive characteristic, as well), then Proposition \ref{P:DiffAP} implies that the exceptional locus is contained in the locus of $\widetilde D\in \widetilde C^{(d)}$ such that $\operatorname{Supp}(\widetilde D)\cap \operatorname{Supp}(\iota \widetilde D)\ne \emptyset$.  
On the other hand, it is easy to see that the exceptional locus of $\delta_d$ contains  the locus of $\widetilde D\in \widetilde C^{(d)}$ such that $\operatorname{Supp}(\widetilde D)\cap \operatorname{Supp}(\iota \widetilde D)\ne \emptyset$.  
Indeed, suppose  that $\widetilde D=\tilde p+\iota \tilde p+\widetilde E= \pi^*p +\widetilde E$ for some $\tilde p\in \widetilde C$, $p\in C$, and $\widetilde E\in \widetilde C^{(d-2)}$.  Then for all $p'\in C$ set $\widetilde D_{p'}=\pi^*p'+\widetilde E$, and we have $\delta_d(\widetilde D_{p'})= \delta_d(\widetilde D)$. 
\end{proof}

\subsubsection{Geometric interpretation of the differential of the Abel--Prym map}\label{S:GeomDiff}
The projectivized differential of $\delta_1$ factors as
$$
\xymatrix{
\widetilde C = \mathbb PT\widetilde C \ar[d]_\pi \ar@{-->}[r]^<>(0.5){\mathbb PT\delta_1}& \mathbb PTP=P\times \mathbb PT_0P\ar[d] \\
C \ar@{-->}[r]^<>(0.5){\phi_{K_C\otimes \eta}}& \mathbb PH^0(C,K_C\otimes \eta)^\vee = \mathbb PT_0P
}
$$
where $\eta$ is the $2$-torsion line bundle on $C$ determining the cover $\pi$, the bottom row is the Prym canonical map, given by the linear system $|K_C\otimes \eta|$, and the right vertical map is the projection onto the second factor.    
Moreover $\mathbb P T\delta_1$ is defined at $\tilde p\in \widetilde C$ if and only if $\phi_{K_C\otimes \eta}$ is defined at $p=\pi(\tilde p)$; i.e., $p$ is not a base point of $|K_C\otimes \eta|$.
One can find all of this in \cite[Prop.~12.5.2]{BL_CAV}  over $\mathbb C$; the arguments hold in positive characteristic, as well.
Indeed, factoring $\delta_1$ as $\widetilde C\stackrel{\alpha_1}{\to} \operatorname{Pic}^1_{\widetilde C/k}\stackrel{1-\iota}{\to}  P'\subseteq J(\widetilde C)$ where $\alpha_1$ is the Abel map, and taking the differential at a point $\tilde p\in \widetilde C$, one obtains the maps $H^0(\mathcal O_{\tilde p}(\tilde p))\stackrel{T\alpha_1}{\to} H^1(\mathcal O_{\widetilde C}) \stackrel{1-\iota}{\to} H^1(\mathcal O_{\widetilde C})^-\subseteq H^1(\mathcal O_{\widetilde C})$.  Identifying $H^1(\widetilde C,\mathcal O_{\widetilde C})=H^0(\widetilde C,K_{\widetilde C})^\vee$ and  $H^1(\widetilde C,\mathcal O_{\widetilde C})^-=H^0(C,K_C\otimes \eta)^\vee$, one uses the fact that the projectivized differential of the Abel map is the canonical map $\phi_{K_{\widetilde C}}$ on $\widetilde C$.

A Riemann--Roch computation gives that $|K_C\otimes \eta|$ has a base point if and only if $\eta=\mathcal O_C(p'-p)$ for some $p,p'\in C$, in which case $C$ is hyperelliptic and $p$ and $p'$ are the base points of $|K_C\otimes \eta|$ (e.g., \cite[Lem.~2.1(i)]{Cetal}). 
Note that  $\pi^*p'\sim \pi^*p$ gives a $g^1_2$ on $\widetilde C$.  
  In particular, this recovers the statements about the differential of $\delta_1$ in Proposition \ref{P:d1} and Remark \ref{R:RamHE}. 

As a consequence, if $K_C\otimes \eta$ is very ample, then  given $\widetilde D\in \widetilde C^{(d)}$, and setting $D=\Nm\widetilde D$,   the projectivization of  $(T_{\widetilde D}\delta_d)(T_{\widetilde D}\widetilde C^{(d)}) $  is identified with  the span of the scheme $\phi_{K_C\otimes \eta}(D)$ in $\mathbb PH^0(C,K_C\otimes \eta)^\vee$.  The same result holds more generally so long as the Prym canonical map is defined at the support of $D$ and the Prym canonical model of $C$ is smooth at the image of the support of $D$.  
Note that another Riemann--Roch computation gives that if $|K_C\otimes \eta|$ is base point free, then  $K_C\otimes \eta$ fails to be very ample if and only if $\eta \cong \mathcal O_C(p+q - p'-q')$ for some $p,q,p',q'\in C$, in which case $C$ is tetragonal,  and $p$ and $q$ are not separated by $|K_C\otimes \eta|$; as usual, when   $p=q$ we mean the differential drops rank at $p=q$  (e.g., \cite[Lem.~2.1(ii)]{Cetal}).

\subsection{Push-forward of the fundamental class under the Abel--Prym map}
The main result of this subsection is the following proposition:

\begin{proposition}\label{P:class-push} Let $\ell$ be a prime number not equal to $\operatorname{char}(k)$.  For $d\le g-1$, and taking $\widetilde D_0\in \widetilde C^{(d)}$,  the class of the push forward of the symmetric product under the pointed Abel--Prym map $\delta_{d,\widetilde D_0}$ is given by 
$$
( \delta_{d,\widetilde D_0})_*[\widetilde C^{(d)}] =  2^d\frac{[\Xi]^{\rho-d}}{(\rho-d)!}\in H^{2\rho-2d}(P,\mathbb Z_\ell(\rho-d)),
$$
where  $\rho=g-1=\dim P$.
\end{proposition}

While Proposition \ref{P:class-push} can be proven exactly as in  the $d=1$ case (i.e., as in the proof of \eqref{E:Pd1} in Proposition \ref{P:d1}), that computation is somewhat laborious, and we prefer to give an alternate proof using \eqref{E:Pd1} as the starting point.  
Similar computations can be found in \cite[\S 1]{B82_sous} and \cite{S89}.

For this, we  take  a brief detour. If  $X,Y\subseteq A$ are subvarieties of an abelian variety, and the map $a:X\times Y\to X+Y\subseteq A$ given by addition is generically finite, then it essentially follows from the definition of the Pontryagin product that in the Chow ring or in the cohomology ring:
$$
[X]\star [Y]= a_*[X\times Y]= \deg (a)[X+Y].
$$
We will want a slight generalization.  If we suppose that $f_X:X'\to X\subseteq A$ and $f_Y:Y'\to Y\subseteq A$ are generically finite surjective morphisms, 
and we set $a' = a\circ (f_X\times f_Y)$ to be the composition:
$$
\xymatrix@C=4em{
a':X'\times Y'\ar[r]^<>(0.5){f_X\times f_Y}& X\times Y\ar[r]^<>(0.5)a_<>(0.5)+& X+Y \subseteq A,
}
$$ 
then, still under the assumption that $a$ is generically finite,  we have
\begin{equation}\label{E:X'+Y'}
f_{X,*}[X']\star f_{Y,*}[Y']=a'_*[X'\times Y']= \deg(a')[X+Y].
\end{equation}
Indeed,  we have the following string of equalities: 
\begin{align*}
a'_*[X'\times Y']&=(\deg a')[X+Y]=(\deg (f_X\times f_Y))(\deg a) [X+Y]=(\deg f_X)( \deg f_Y)(\deg a) [X+Y]\\
&=(\deg f_X)(\deg f_Y)  [X]\star [Y] =  ((\deg f_X) [X'])\star ((\deg f_Y)[Y'])=f_{X,*}[X']\star f_{Y,*}[Y'].
\end{align*}

Finally, we will want to use the standard result that for a principally polarized abelian variety $(A,\Theta)$ of dimension $g$, in the Chow ring we have:
\begin{equation}\label{E:Pont1}
\frac{[\Theta]^{g-a}}{(g-a)!}\star \frac{[\Theta]^{g-b}}{(g-b)!}=\binom{a+b}{a}\frac{[\Theta]^{g-(a+b)}}{(g-(a+b))!},
\end{equation}
which we will use in the form
\begin{equation}\label{E:Pont2}
\left(\frac{[\Theta]^{g-1}}{(g-1)!}\right)^{\star d}=d!\frac{[\Theta]^{g-d}}{(g-d)!}.
\end{equation}
A reference for \eqref{E:Pont1}  over $\mathbb C$ is \cite[Cor.~16.5.8, p.538]{BL_CAV}, which uses as its starting point \cite[Thm., p.647]{B86_chow}, also proven over $\mathbb C$.
 However, 
 \cite[Thm.~2.19]{DM91_fourier}
 shows that Beauville's result holds over any algebraically closed field, and consequently, the arguments for \cite[Cor.~16.5.8, p.538]{BL_CAV} go through in positive characteristic, as well.  Of course,
 \eqref{E:Pont1} is elementary to prove in $\ell$-adic cohomology, and  this is, in fact, all we need.  

\begin{proof}[Proof of Proposition \ref{P:class-push}]
Let $\tilde p_0\in \widetilde C$, and set $\widetilde D_0=d\tilde p_0$.  From Corollary \ref{C:gen-fin} we know that $\delta_d$ is generically finite; therefore, 
if we factor the composition  $\delta_{d,\widetilde D_0}^\times:\widetilde C^d\stackrel{\Sym}{\to} \widetilde C^{(d)}\stackrel{\delta_{d,\widetilde D_0}}{\to} P $ as
$$
\xymatrix{
\delta^\times_{d,\widetilde D_0}:\widetilde C^d \ar[r]^<>(0.5){\delta_{1,\tilde p_0}^d}& P^{\times d}\ar[r]^+ & P
}
$$
then from the left hand side of  \eqref{E:X'+Y'},  \eqref{E:Pd1},  and \eqref{E:Pont2}, we have 
$$
(\delta^\times_{d,\widetilde D_0})_*[\widetilde C^d]= \left((\delta_{1,\tilde p_0})_*[\widetilde C]\right)^{\star^d}= \left(2\frac{[\Xi]^{\rho-1}}{(\rho-1)!}\right)^{\star^d} = 2^d d! \frac{[\Xi]^{\rho-d}}{(\rho-d)!}.
$$
On the other hand, we have 
$$
(\delta^\times_{d,\widetilde D_0})_*[\widetilde C^d]=  (\delta_{d,\widetilde D_0})_* \Sym_*[\widetilde C^d] =  d!(\delta_{d,\widetilde D_0})_* [\widetilde C^{(d)}],
$$
completing the proof.
\end{proof}

\subsection{The degree of the Abel--Prym map}

We start with the following consequence of Proposition \ref{P:class-push}:

\begin{corollary}\label{C:class-push} Let $\ell$ be a prime number not equal to $\operatorname{char}(k)$.
For $d\le g-1$, and taking $\widetilde D_0\in \widetilde C^{(d)}$,  the class $[\operatorname{Im} \delta_{d,\widetilde D_0}]$ of the image of the pointed Abel--Prym map  $ \delta_{d,\widetilde D_0}$ (as a set, or rather as an irreducible scheme, with the reduced induced scheme structure) is
$$
[\operatorname{Im} \delta_{d,\widetilde D_0}] = \frac{2^d}{\deg \delta_{d,\widetilde D_0}}\frac{[\Xi]^{\rho-d}}{(\rho-d)!}\in H^{2\rho-2d}(P,\mathbb Z_\ell (\rho-d)),
$$
where   $\rho=g-1=\dim P$. 
\end{corollary}

\begin{proof}
This follows from Proposition \ref{P:class-push} using the fact that $\delta_d$ is generically finite (Corollary \ref{C:gen-fin}) so that 
 $(\delta_{d,\widetilde D_0})_*[\widetilde C^{(d)}]= \deg (\delta_{d,\widetilde D_0})[\operatorname{Im} \delta_{d,\widetilde  D_0}]$.
\end{proof}

This gives the following corollary:

\begin{corollary}\label{C:deg}
We have  $\deg \delta_{g-1}=2^{g-1}$, and if $\operatorname{char}(k)=0$, then for $d\le g-2$ we have $\deg \delta_d=2^n\le 2^d$ for some integer $n\le d$.
If $\operatorname{char}(k)=p>0$,   then  for $d\le g-2$ we have $\deg \delta_d=p^{m}2^n$ for some integers $m$ and $n$ with $n\le d$. 
\end{corollary}

\begin{proof}
In the case where $d=g-1$, we know, in addition, the class of the image  $\operatorname{Im}(\delta_{g-1,\widetilde D_0})$; indeed,  $\delta_{g-1,\widetilde D_0}$ surjects on to $P$ (Corollary \ref{C:gen-fin}), so that $\operatorname{Im}(\delta_{g-1,\widetilde D_0})=P$.  The fact that $\deg \delta_{g-1}=2^{g-1}$ then follows immediately from Corollary \ref{C:class-push}, completing the proof. 

The case where $d\le g-2$ follows from the fact that 
 $\frac{[\Xi]^{\rho-d}}{(\rho-d)!}\in H^{2\rho-2d}(P,\mathbb Z_\ell (\rho-d))$ is a minimal cohomology class (i.e., it is not divisible by $\ell$), so that Corollary \ref{C:class-push} (considered for all primes $\ell\ne \operatorname{char}(k)$) implies that $\deg \delta_d$ must be a power of the characteristic exponent of $k$ times a power of $2$ that is at most $2^d$.  
\end{proof}

\subsection{The Abel--Prym map for hyperelliptic covers} 

We now discuss the degree of the Abel--Prym map in  the case where $\widetilde C$ is hyperelliptic.

\begin{proposition}\label{P:hype}
Assume that $\widetilde C$ is hyperelliptic.  Then for $d\le g-1$, we have $\deg \delta_d=2^d$.
\end{proposition}

\begin{proof} From Corollary \ref{C:deg} we only need to show that $\deg \delta_d\ge 2^d$.  
So fix  $\widetilde D=\tilde p_1+\cdots +\tilde p_d\in \widetilde C^{(d)}$ to be  general; in particular, such that the fiber of $\delta_d$ over $\delta_d(\widetilde D)$ is finite.   For  $\widetilde D'=\tilde p_1'+\cdots +\tilde p_d'\in \widetilde C^{(d)}$ we have    $\delta_d(\widetilde D')=\delta_d(\widetilde D)$ if and only if 
$
\tilde p_1 -\iota \tilde p_1 +\cdots +\tilde p_d -\iota\tilde p_d \sim \tilde p_1' -\iota \tilde p_1' +\cdots +\tilde p_d' -\iota\tilde p_d'
$
or equivalently, 
\begin{equation*}\label{E:HE1}
\tilde p_1 +\iota \tilde p_1' +\cdots +\tilde p_d +\iota\tilde p_d'\sim \tilde p_1' +\iota \tilde p_1 +\cdots +\tilde p_d' +\iota\tilde p_d.
\end{equation*}
Denoting by $h$ the hyperelliptic involution on $\widetilde C$, the assumption that $\widetilde D$ is general means that we can assume that the $\tilde p_i$ and $\iota h (\tilde p_j)$ are all distinct.  Therefore there are $2^d$ distinct choices of $\widetilde D'$  such that  $\delta_d(\widetilde D')=\delta_d(\widetilde D)$, 
determined by the $2^d$ choices of taking $\tilde p_i'$ either equal to $\tilde p_i$ or to  $\iota h(\tilde p_i)$; note that since the $g^1_2$ on $\widetilde C$ is unique, $\tilde p + h(\tilde p)$ is in the $g^1_2$ if and only if $\iota \tilde p +\iota h(\tilde p)$ is in the $g^1_2$.
\end{proof}

\subsection{The Abel--Prym map for general covers}
 
 We now prove that for a general cover $\widetilde C/C$, the degree of the Abel--Prym map is $1$ for $d<g/2$.

\begin{corollary}[The Abel--Prym map for general covers]\label{C:GenAP}
Let $\widetilde C/C$ be a general cover.  Then  $\deg \delta_d=1$ for   $d<g/2$.
\end{corollary}

\begin{proof}
Let $\widetilde D\in \widetilde C^{(d)}$ be a general point; in particular assume that $\widetilde D$ is reduced and  $\operatorname{Supp}(\widetilde D)\cap \operatorname{Supp}(\iota \widetilde D)=\emptyset$.   Now suppose that $\widetilde D'\in \widetilde C^{(d)}$ and $\delta_d(\widetilde D)=\delta_d(\widetilde D')$.  Then we have
$\widetilde D+\iota \widetilde D'\sim \widetilde D'+\iota \widetilde D$.  
If $\widetilde D+\iota \widetilde D'= \widetilde D'+\iota \widetilde D$, then considering the supports of the divisors, we must have that $\widetilde D'=\widetilde D$.  Otherwise, $\widetilde C$ admits a $g^1_{2d}$.     
However, by \cite[Thm.~1.4]{AF}, for a general cover $\widetilde C/C$, there is no $g^1_e$ on $\widetilde C$ for $e<g$. Thus if $d<g/2$, then $\widetilde C$ does not admit a $g^1_{2d}$,  and so $\delta_d$ is generically injective. 
\end{proof}

\begin{remark}
The Donagi--Smith \cite[\S 3]{DS} approach to computing the degree  of a generically finite morphism by computing the local degree along a fiber,  together with the  General Position Theorem (e.g., \cite[p.109]{ACGHI} or \cite[Thm.~3.1]{EH92lin}) allows one to compute  $\deg \delta_d=1$ in characteristic $0$ for $\widetilde C/C$ general and $d\le g-2$. 
\end{remark}

\bibliographystyle{amsalpha}
\bibliography{biblio, bibliYano}{}

\end{document}